\documentclass[a4paper,english,10pt]{amsart}
\usepackage[applemac]{inputenc}
\usepackage{inputenc}
\usepackage{babel}
\usepackage{amsfonts}
\usepackage{amsmath}
\usepackage{amssymb}
\usepackage{amsthm}
\usepackage{calligra,mathrsfs}
\usepackage{graphicx}
\usepackage[all]{xy}
\usepackage{textcomp}
\usepackage{tikz-cd}
\usepackage{enumitem} 
\usepackage{hyperref}
\hypersetup{
      colorlinks = true,
      urlcolor =   blue,
      linkcolor =  blue,
      citecolor = purple
}

\newtheorem{thm}{Theorem}[section]  
\newtheorem{cor}[thm]{Corollary}
\newtheorem{lem}[thm]{Lemma}
\newtheorem{defi}[thm]{Definition}
\newtheorem{prop}[thm]{Proposition}
\newtheorem{es}[thm]{Example}
\newtheorem{rem}[thm]{Remark}
\newtheorem{con}[thm]{Conjecture}

\DeclareMathOperator{\ad}{ad}

\DeclareMathOperator{\Sp}{Sp}
\DeclareMathOperator{\Spa}{Spa}
\DeclareMathOperator{\id}{id}
\DeclareMathOperator{\Spf}{Spf}

\DeclareMathOperator{\Spec}{Spec}
\DeclareMathOperator{\Gal}{Gal}

\DeclareMathOperator{\op}{op}

\DeclareMathOperator{\Bun}{Bun}
\DeclareMathOperator{\Homs}{\mathscr{H}\text{\kern -3pt {\calligra\large om}}\,}
\DeclareMathOperator{\Exts}{\mathscr{E}\text{xt}\,}

\DeclareMathOperator{\sF}{\mathscr{F}}
\newcommand{\sG}{\mathscr{G}}
\newcommand{\sL}{\mathscr{L}}

\DeclareMathOperator{\HT}{HT}
\DeclareMathOperator{\Berk}{Berk}

\title{Constructibility and Reflexivity in non-Archimedean geometry}
\author{Ildar Gaisin, John Welliaveetil }

\begin{document}

\maketitle

\begin{abstract}
We introduce a notion of constructibility 
for étale sheaves with torsion coefficients over 
a suitable class of adic spaces. 
 This notion is related to the classical notion of constructibility for schemes via the nearby cycles functor. 
 We use the work of R. Huber to define an adic Verdier dual and 
 investigate the extent to which we have a 6-functor formalism in this context. 
 Lastly, we attempt to classify those sheaves which are 
 reflexive with respect to the adic Verdier dual.
\end{abstract}

 {\hypersetup {linkcolor = black} 
\tableofcontents
}
\pagebreak
   
 \section{Introduction} 
 
     In the early 1960's, motivated by the question of uniformizing elliptic curves with split multiplicative reduction over a 
     non-Archimedean real valued field, Tate introduced the theory of rigid analytic spaces 
     and a robust formalism within which one could discuss a notion of coherent sheaves. 
     Since the introduction of rigid geometry, there have been other significant theories of non-Archimedean geometry, 
     namely - the theory of formal schemes outlined by Raynaud, the theory of Berkovich spaces and Huber's adic spaces, each of which 
     was developed with a different goal in mind. 
     Berkovich's theory provides one with non-Archimedean analytic spaces 
     with nice topological properties whilst adic spaces were introduced to better understand the étale cohomology of rigid varieties. 
     
     An étale cohomology theory for non-Archimedean spaces seeks to prove analogues of the 
     classical theorems that hold true in the étale cohomology of algebraic varieties such as finiteness results, 
     base change theorems, Poincaré duality, Kunneth formula... 
  Despite the advances made by Berkovich and Huber, we do not yet have a suitable theory of 
  constructible sheaves for non-Archimedean spaces. 
  If we were to imitate the classical definition and require that constructible sheaves be those for which there exists 
  a semi-analytic stratification such that the restriction of the sheaf to each element of the strata is finite locally constant then 
  it is not true that this class of sheaves is stable by pushforwards,  cf. \cite{FM}.  
  In fact it was only recently shown that if $X$ is a compact 
strictly $k$-analytic space then the groups $H^{q}(X,\mathbb{Z}/\ell\mathbb{Z}) \cong H_{c}^{q}(X,\mathbb{Z}/\ell\mathbb{Z})$ 
are of finite dimension (cf. \cite{berk94}, \cite{berk96}, \cite{berk13}) where 
$k$ is an algebraically closed complete non-Archimedean real valued field and $\ell$ is a prime number
 different from the characteristic of the residue field $\widetilde{k}$. 
In the language of adic spaces, if $f: X \rightarrow Y$ is 
smooth and quasi-compact then there is a theorem of 
stability for $R^qf_{!}$ with respect to a certain class of constructible sheaves, cf. \cite{hub96}, and 
in the case that $\dim(Y) \leq 1$, and $f: X \rightarrow Y$ is only quasi-compact, we 
have results of stability by $R^qf_{*}$ (in characteristic 0) and $R^qf_{!}$ for 
another class of sheaves (cf. \cite{hub98a}, \cite{hub98b}, \cite{hub07}).
The focus of this paper is to discuss a notion of constructibility that 
generalizes the constructions of Huber and Berkovich and to study the 
extent to which one has a six functor formalism in this context. We also attempt to relate our notion of constructibility to reflexivity with respect to the adic Verdier dual. More precisely we conjecture that semi-constructible sheaves coincides with reflexive sheaves, while also verifying this conjecture in some particular cases.  This is in the spirit of some upcoming work of P. Scholze.

     
\subsection{Motivation}     
 
     Let $X$ be an adic space that is separated and finite type over $k$. 
     We refer to such spaces as \emph{fine} $k$-adic spaces. 
     Recall that we have an isomorphism of locally ringed topological spaces
     \begin{align} 
       X \simeq  \varprojlim_{\mathfrak{X} \in \mathfrak{B}} \mathfrak{X}_s
     \end{align} 
    where $\mathfrak{B}$ is the cofiltered family of admissible formal models of $X$. 
    
    Observe that for every $\mathfrak{X} \in \mathfrak{B}$, $\mathfrak{X}_s$
    is a variety over $\tilde{k}$. 
    Hence, it seems reasonable to wonder if 
     one could study the étale topos of 
     $X$ using the relatively well understood étale topoi of the varieties $\mathfrak{X}_s$ as 
     $\mathfrak{X}$ varies along the family $\mathfrak{B}$. 
     With this in mind, we restrict the étale site of $X$ and consider only those objects which
     are fine, thus obtaining the fine étale site $X^f_{\text{ét}}$.  
     This allows us to define in a natural way the nearby cycles functor cf. \S \ref{nearby cycles functor} which was studied in \cite{hub96} (in the adic context) as well as 
          \cite{berk94}, \cite{berk96} and \cite{berk13} (in the Berkovich context).
       More precisely, given a formal model $\mathfrak{X}$ of $X$, we have a functor of derived categories
       \begin{align*} 
         R\psi_\mathfrak{X} : \mathcal{D}^b(X^f_{\text{et}},\Lambda) \to \mathcal{D}^b(\mathfrak{X}_{s,\text{et}},\Lambda) 
       \end{align*}    
          where $\Lambda := \mathbb{Z}/\ell\mathbb{Z}$ for some prime $\ell$ different from $\mathrm{char}(\tilde{k})$.
          
   \subsection{Constructible and semi-constructible complexes}    
   
        Our notion of semi-constructibility is motivated by the isomorphism (1) above. We say that 
    a complex $\sF \in \mathcal{D}^b(X^f_{\text{ét}},\Lambda)$ is semi-constructible if 
    after pulling back to any object $U$ of $X^f_{\text{ét}}$, the image of $\sF$ via the nearby cycles functor 
    $R\psi_{\mathfrak{U}}(\sF_{|U})$ is constructible for every formal model $\mathfrak{U}$ of $U$. 
    The class of semi-constructible sheaves is not stable for pullback (cf. Proposition \ref{verdier dual not constructible}). 
    We get around this issue by defining a constructible sheaf to be a semi-constructible sheaf that is stable for 
    pullbacks along morphisms $f : Y \to X$ where $Y$ is a fine $L$-adic space for some non-Archimedean algebraically closed 
    complete field extension $L$ of $k$. The precise definition can be found in 
     \S \ref{constructible sheaves}.

    In \S 3, we show that the class of constructible sheaves described above has reasonable properties. 
    Firstly, the class of constructible sheaves is large enough to include both 
   Huber's notion of constructibility (cf. Remark \ref{huber constructible is constructible}) and those 
   introduced by Berkovich in \cite{berk13}. 
   In Proposition \ref{stability for pushforwards}, we prove that this class is stable for pushforwards and the lower shriek functor. 
   It is natural to ask if the properties of being semi-constructible or constructible are local for the fine étale topology. 
   We answer these questions in the affirmative in Theorem \ref{theorem above} and \S \ref{semi-constructibility is an etale local property}. 
   A powerful tool that we use to prove these results is an inductive construction from 
   \cite{berk94} which we adapt to the adic setting (cf. Lemma \ref{lem:redberkad}).

     Huber's definition of constructibility preserves the spirit of the classical construction in that he requires that his sheaves be finite locally constant along certain stratifications. 
    We provide an alternate characterization of constructibility that shows the definition introduced above is similar in spirit to the classical definition. 
    Firstly, in Proposition \ref{classical constructibles are finite}, using an argument from SGA 4.5, we show that an étale sheaf $\sG$ of $\Lambda$-modules on a
    $k$-variety $Z$ is constructible if and only if $H^0(Z'_{\text{ét}},g^*(\sG))$ is finite where 
    $Z'$ is a $K$-variety \footnote{The field $K$ is an algebraically closed 
     field extension of $k$.} and
    $g : Z' \to Z$ is a morphism of schemes that is the composition of a morphism of $K$-varieties $Z' \to Z_K$ 
    followed by the projection $Z_K \to Z$.   
      Theorem \ref{theorem above} tells us that much like in the case of varieties, the 
    constructible sheaves we are interested in must satisfy a strong finiteness property. 
   More precisely, we show that a sheaf 
   $\sF$ is constructible if and only if for every $q$,
   $H^q(Y^f_{\text{ét}},f^*(\sF))$ is finite, where 
   $Y$ is a fine $L$-adic space \footnote{The field $L$ is a non-Archimedean algebraically closed 
    complete field extension of $k$.} and 
    the morphism $f : Y \to X$ is the composition of a morphism 
    of fine $L$-adic spaces $Y \to X_L$ followed by the projection 
    $X_L \to X$. 
    
      Let $f$ be a morphism between fine $k$-adic spaces. The following table summarises the various
       properties of constructible, semi-constructible and Huber constructible sheaves.

      \begin{center}
    \begin{tabular}{ | p{3.5cm} | c | c | c |}
    \hline
     & Huber constructible & constructible & semi-constructible \\ \hline
    stable under $f^*$ & Yes & Yes & No \\ \hline
    stable under $Rf_*$ & No & Yes & Yes \\ \hline
    stable under $Rf_!$ & No & Yes & Yes \\ \hline
    stable under $f^!$ & No & No & No \\ \hline
    stable under $\otimes^L$ & Yes &  No & No \\ \hline
    stable under $R\Homs$ & Yes & No & No \\ \hline
    can be checked étale locally & Yes & Yes & Yes \\ \hline
    \end{tabular}
\end{center}

Moreover we have inclusions of categories
\[
\left\{ \text{Huber constructible} \right\} \subset \left\{ \text{constructible} \right\} \subset \left\{ \text{semi-constructible} \right\} \subset \left\{ \text{reflexive-sheaves} \right\}
\]
where the first two inclusions are typically strict and the last inclusion is conjectured to be an equality.

\subsection{Counterexamples} 

      The affirmative entries in the table above are justified with proofs and we 
      provide counterexamples for the negative entries. 
      There are three principal counterexamples. 
      
      In \S \ref{huber constructible is not stable for upper shriek}, we provide an example 
      of a morphism $f : X \to Y$ of fine $k$-adic spaces and a sheaf $\sF$ on $Y$ such that
      $f^!(\sF)$ is not Huber constructible. This example was suggested to us by R. Huber. 
       More precisely, we take $X$ to be a singular curve over the field $k$ and 
        then show that $p_{X^{\mathrm{ad}}}^!(\Lambda)$ is not a complex whose cohomology is Huber constructible. 
        On the other hand, $p_{X^{\mathrm{ad}}}^!(\Lambda)$ is a constructible complex.
         
        In \S \ref{counterexample for Verdier dual stability}, we give an example 
        of a sheaf $\sF$ on the adic unit disk $X$ such that 
        the adic Verdier dual $D^{\mathrm{ad}}(\sF)$ is not a constructible complex. 
         This example was suggested to us by P. Scholze. 
         The sheaf $\sF$ is an infinite direct sum of the form 
         $\oplus_{i \in \mathbb{N}} \sF_i$ where for every 
         $i \in \mathbb{N}$, $\sF_i$ is a constructible sheaf. 
        An important fact to verify is that 
        for any morphism $f : Y \to X$ of fine $k$-adic spaces and 
        any $j \in \mathbb{N}$, we have that for all but finitely many 
        $i \in \mathbb{N}$, $H^j(Y,f^*(\sF_i)) = 0$.
        Since the Verdier dual of a semi-constructible sheaf is semi-constructible, 
        this construction gives us
         an étale sheaf on the adic unit disk which has infinite stalk at the origin
          but whose nearby cycles are constructible  
         
         In \S \ref{counterexample for tensor product stability}, we 
         find constructible sheaves $\sF$ and $\sG$ on the adic unit disk $X$ such that 
         $\sF \otimes^L \sG$ is not semi-constructible. This is a variation of 
         the example described above and makes crucial use of 
         results in \cite{hub98a}.

 \subsection{Adic Verdier dual} 
 
     By the results of Huber in \cite{hub96}, 
     we have a functor 
     $Rf_! : \mathcal{D}^b(X^f_{\text{ét}},\Lambda) \to \mathcal{D}(Y^f_{\text{ét}},\Lambda)$ 
     that admits a right adjoint 
     $f^! : \mathcal{D}^b(Y^f_{\text{ét}},\Lambda) \to \mathcal{D}(X^f_{\text{ét}},\Lambda)$.
     As in classical algebraic geometry, given a complex $\sF \in \mathcal{D}^b(X^f_{\text{ét}},\Lambda)$,
     we define 
     \begin{align*}
       D^{\mathrm{ad}}(\sF) := R\Homs(\sF,p_X^!(\Lambda))
     \end{align*} 
      where $p_X : X \to \mathrm{Spa}(k,k^0)$ is the structure morphism. 
     
     Once again, prompted by isomorphism (1) above, we 
     ask the following question. \\
     
     \noindent \emph{Question} : Can the adic Verdier dual  dual of $\sF$ be described in terms of the nearby cycles of $\sF$ 
                       and its pullbacks along étale morphisms ? \\
                       
     In \S 4, we answer this question by proving that $D^{\mathrm{ad}}(\sF)$ can be recovered from the data of the nearby cycles $R\psi_\mathfrak{U}(\sF_{|U})$ where 
     $U \to X$ is an object in $X^f_{\text{ét}}$ and $\mathfrak{U}$ is a formal model of $U$.  
     More precisely, we show in Proposition \ref{compved} that for any étale morphism $U \to X$ of fine $k$-adic spaces,
     \begin{align*} 
     D^{\mathrm{ad}}(\sF)(U) = R\Gamma(\mathfrak{U}_s,R\psi_\mathfrak{U}(\sF_{|U}))      
          \end{align*}  
     where $\mathfrak{U}$ is any formal model of $U$. 
     
      It can be shown that
      the association 
      $U \mapsto R\Gamma(\mathfrak{U}_s,R\psi_\mathfrak{U}(\sF_{|U}))$ is a well defined presheaf on 
      $X^f_{\text{ét}}$ which takes values in the derived category of bounded complexes of $\Lambda$-modules. 
      In fact, using techniques from the theory of $\infty$-categories, one can show without alluding to $f^!$ that this recipe 
        fully defines a complex in $\mathcal{D}^b(X^f_{\text{ét}},\Lambda)$.
     
      From the discussion above, it should not be surprising that we prove in Theorem \ref{big result small proof}, the identity
      \begin{align*} 
       R\psi_\mathfrak{X} \circ D^{\mathrm{ad}} = D \circ R\psi_\mathfrak{X}. 
      \end{align*} 
     The adic Verdier dual does not preserve constructible sheaves, however it does preserve semi-constructibles. 
      As in the classical case, we prove in Lemma \ref{projection formula adic spaces} and Corollary \ref{Corollary 2}, that it behaves nicely with respect to the functors $Rf_*$, $Rf_!$ and 
      $f^!$. An important point to note is that the dualizing complex is a constructible complex. 
       We prove this using an adic version of the classical cohomological descent argument (cf. \S \ref{cohomological descent}) akin to 
       \cite[\S 1.2]{berk13}.

 \subsection{Classification of reflexive sheaves} 
 In some upcoming work, P. Scholze shows that the notion of \emph{reflexive sheaves} turns out to be a suitable framework for the sheaves appearing in Fargues' conjecture, cf. \cite[Conjecture 4.4]{farguesgeo}. More precisely, for $G$ a quasi-split reductive group over $\mathbb{Q}_p$ there is a diamond stack in groupoids $\Bun_G$ whose points are the $G$-torsors over the relative Fargues-Fontaine curve. Scholze proves that there is a fairly simple characterization of reflexive sheaves over $\Bun_G$: namely they are the sheaves whose stalks are admissible representations. The need to work with reflexive sheaves in this setting is highlighted by the fact that the typical sheaves appearing in Fargues' conjecture have infinite dimensional stalks. For this reason the \emph{classical} notion of constructibility is not sufficient.
 
        In \S \ref{sfjxoosdfsdfsd}, we attempt to classify those sheaves on $X^f_{\text{ét}}$ which 
      are reflexive with respect to the adic Verdier dual i.e. the class of sheaves $\sF$ such 
      that the canonical morphism $\sF \to D^{\mathrm{ad}} \circ D^{\mathrm{ad}}(\sF)$ is an isomorphism. 
      As a first step in this direction, we prove in Theorem \ref{reflexive on generic and special} that 
      a sheaf $\sF$ on $X^{f}_{\text{ét}}$ is reflexive if for 
      every $U \to X$ in $X^f_{\text{ét}}$ there exists a formal model $\mathfrak{U}$ of $U$ such that 
    $R\psi_\mathfrak{U}(\sF_{|U})$ is reflexive with respect to 
   the classical Verdier dual on $\mathfrak{U}_{s,\text{ét}}$. 
   We conjecture (cf. Conjecture \ref{reflexivity conjecture}) that an étale sheaf $\sG$ on a variety 
   $Z$ over an algebraically closed field is reflexive if and only if it is constructible. 
   In Corollaries \ref{cor:trsfinindust}, \ref{cor:5.17god} and Propositions \ref{the case of the direct sum}, \ref{prop:5223} we verify particular instances of this conjecture. The idea in each of the cases is to show that a reflexive sheaf has étale cohomology which can be controlled. For instance we make use of the Grothendieck-Ogg-Shafarevich formula. From these calculations, it appears (at least to the authors) that Conjecture \ref{reflexivity conjecture} is related to ramification problems.

   One can 
      use techniques from SGA 4.5 to show that Conjecture \ref{reflexivity conjecture} is true if certain pullbacks of reflexive sheaves remain reflexive.
   
       Finally assuming Conjecture \ref{reflexivity conjecture}, we prove using Theorem \ref{big result small proof} that 
      a sheaf $\sF$ on $X^f_{\text{ét}}$ is reflexive if and only if 
      $\sF$ is semi-constructible.    \\

\noindent \textbf{Acknowledgments}. Ildar Gaisin would like to express his deep 
gratitude to his advisors Jean-Fran\c cois Dat and Laurent Fargues, who 
suggested that he should think about developing a notion of perversity in 
relation to the Langlands Program.
John Welliaveetil would like to thank the Max-Planck Institute for Mathematics where a portion of this work was done. 
 He is also grateful to Imperial college, London for the suitable working conditions and Johannes Nicaise in particular for his
 support and encouragement through this period. 
 The essential ideas of this article are an
 elaboration of comments made to us by Peter Scholze and we are grateful that he shared this with us. Next we want to thank 
 Marco Robalo for spending countless hours answering our questions. We
  would also like to thank Ahmed Abbes, Justin Campbell, Pierre Deligne, Ofer Gabber, Benjamin Hennion, Luc Illusie, Andreas Gross and Damien Lejay for many
   helpful remarks. \\  

\noindent \textit{Notation:} For the remainder of the article, unless otherwise stated, we fix an algebraically closed non-Archimedean real valued field $k$. 

\section{The fine topology on an adic space}

  Our goal in this section is to provide a short introduction to the theory of adic spaces. We then proceed to define the category of 
  \emph{fine} adic spaces which will form the focus of the sections that follow. Our primary references are R. Huber's book 
  \cite{hub96} and P. Scholze's article \cite{perfsapcschp}. 

\subsection{Adic spaces} 
  
 \subsubsection{Huber rings}

  \begin{defi}  
  \emph{A} Huber ring \emph{is a topological ring $A$ that contains an open sub-ring $A_0$ such that there exists 
  a finitely generated ideal $I \subset A_0$ and the family $\{I^n | n \in \mathbb{N}\}$ forms a basis of open neighbourhoods of 
  zero. We call any such ring $A_0$ a} ring of definition of $A$ \emph{while any ideal $I$ that satisfies the property above is referred to as an} ideal of definition. 
 \end{defi} 
  
  
  \begin{es} \label{Huber ring examples} 
   \begin{enumerate}
   \item  \emph{Let $k$ be a non-Archimedean field. The valuation ring $k^0 := \{x \in k | |x| \leq 1\}$ is a ring of definition of $k$ and 
    the ideal generated by any element $\pi \in k^0$ such that $|\pi| < 1$ is an ideal of definition. 
    There is a large and interesting class of Huber rings that are also $k$-algebras, namely the Tate algebras defined over $k$. 
    It can be checked easily that the algebra $k\langle{ }X_1,\ldots,X_n\rangle{ }/I$ is a Huber ring where $I \subset k\langle{ }X_1,\ldots,X_n\rangle{ }$ is an ideal.  
   Such examples of Huber rings where the ideal of definition is generated by a topologically nipotent unit in the ring are called} Tate Huber rings. 
                                                    \end{enumerate} 
  \end{es}

\subsubsection{Affinoid adic spaces}

  \begin{defi} 
    \emph{An} affinoid ring \emph{is a pair $(A,A^+)$ such that $A$ is a Huber ring and $A^+$ is an open integrally closed sub-ring that is contained in the set 
    of power bounded elements, cf. \cite[\S 1.1]{hub96}. The ring $A^+$ is said to be a} ring of integral elements. 
    \emph{A morphism $(A,A^+) \to (B,B^+)$ of affinoid rings is a continuous ring homomorphism $\phi : A \to B$ such that 
    $\phi(A^+) \subset B^+$.}
    \end{defi}
    
     Henceforth, we will assume that all affinoid rings under consideration are \emph{complete}.  
    
    \begin{es} \label{affinoid ring examples}
    \emph{Example \ref{Huber ring examples} can be generalized to an arbitrary base ring that is Huber (leading to the notion of topologically of finite type). 
    Let $A$ be a Huber ring and $M = (M_1,\ldots,M_n)$ be a tuple of finite subsets of $A$ such that 
    for every $i$, $M_i \cdot A \subset A$ is open. 
    Define 
    \begin{align*} 
       A\langle X_1,\ldots,X_n\rangle_M := \{ \Sigma_{v \in \mathbb{N}^n} a_vX^v | \text{for every open set $U$ in $A$, we have that} \\
                                               \text{for all but 
                                                finitely many indices}, a_v \notin M^v \cdot U\}. 
    \end{align*} 
    where if $v = (v_1,\ldots,v_n)$, $X^v := X_1^{v_1} \cdot\ldots \cdot X_n^{v_n}$ and $M^v := M_1^{v_1}\cdot \ldots \cdot M_n^{v_n}$.}
    
    \emph{We endow $A\langle X_1,\ldots,X_n\rangle_M$ with a topology such that the sets}
    $$\{\Sigma_{v \in \mathbb{N}^n} a_vX^v |  a_v \in M^v\cdot U\}$$ ($U$ is an open neighbourhood of $0$ in $A$) 
                                        \emph{ form a basis of open neighbourhoods of $0$ in 
                                          $A\langle X_1,\ldots,X_n\rangle{ }_M$. 
                     It can be checked that $A\langle{ }X_1,\ldots,X_n\rangle{ }_M$ is a Huber ring.}                      
         
        \emph{ Suppose we are given a ring of integral elements $A ^+ \subset A$. We  
        can define a ring of integral elements in 
        $A\langle{ }X_1,\ldots,X_n\rangle{ }_M$ as follows.}
        \emph{Let}  $B := \{\Sigma_{v \in \mathbb{N}^n} a_vX^v |  a_v \in M^v \cdot A^+\}$ \emph{and $C$ denote the integral closure of $B$ 
        in $A\langle{ }X_1,\ldots,X_n\rangle{ }_M$. We then have that $(A\langle{ }X_1,\ldots,X_n\rangle{ }_M, C)$ is an} affinoid ring.
            \end{es}  
    
     Given an affinoid ring $(A,A^+)$ we can associate to it a locally topologically ringed topological space which we denote 
     $\mathrm{Spa}(A,A^+)$.
     
    \begin{defi}(Locally topologically ringed spaces)
      \emph{Let $(V)$ denote the category of triples 
     $(X,\mathcal{O}_X, (v_x | x \in X))$ where $X$ is a topological space, $\mathcal{O}_X$ is a sheaf of complete topological rings 
     and for every $x \in X$, $v_x$ is an equivalence class of valuations 
     \footnote{A brief discussion of valuations is provided in \cite[\S 1.1]{hub96}.}
      on the local ring $\mathcal{O}_{X,x}$. 
      A morphism $(X,\mathcal{O}_X,(v_x | x \in X)) \to (Y, \mathcal{O}_Y,(v_y | y \in Y))$ of objects in $(V)$ is a continuous map 
      $f : X \to Y$ and a morphism of sheaves of topological rings $\phi : \mathcal{O}_Y \to f_*(\mathcal{O}_X)$ such that 
      for every $x \in X$,
      the induced morphism of topological local rings $\phi_x : \mathcal{O}_{Y,f(x)} \to \mathcal{O}_{X,x}$ is compatible with the valuations 
      $v_{f(x)}$ and $v_x$ respectively, i.e. $v_{f(x)} = v_x \circ \phi_x$. For every object of $(V)$, $\mathcal{O}_X^{+}$ denotes the subsheaf of $\mathcal{O}_X$ with $\mathcal{O}_X^{+}(U) := \{ s \in \mathcal{O}_X(U) \text{ }\lvert v_x(s) \leq 1 \text{ for every $x \in U$} \}$ (here $U$ is any open in $X$).} 
      \end{defi}
    
       Let $\mathrm{Spa}(A,A^+)$ be the set of continuous valuations (up to equivalence) on $v : A \to \Gamma \cup \{0\}$ such that 
      $v(a) \leq 1$ for every $a \in A^+$. We equip $\mathrm{Spa}(A,A^+)$ with the topology generated by 
      the sets $\{x \in \mathrm{Spa}(A,A^+) | |a(x)| \leq |b(x)| \neq 0, a,b \in A\}$. 
      
      The space $X := \mathrm{Spa}(A,A^+)$ can be endowed with a structure presheaf $\mathcal{O}_X$ by specifying its values 
      on so-called rational subsets of $X$ and then extending this definition in a natural way 
      to include all open sets in $X$ .
    
    \begin{defi} (Rational subsets)
    \emph{ Let $X = \mathrm{Spa}(A,A^+)$ where $(A,A^+)$ is an affinoid ring. 
      Let $T \subset A$ be a finite set of elements such that $T \cdot A \subset A$ is open. 
     For $s \in A$ let $U(T/s) = \{x \in \mathrm{Spa}(A,A^+) \text{ }| |t(x)| \leq |s(x)| \neq 0, t \in T \}$. 
     A set of this form is referred to as a} rational subset of $X$. 
    \end{defi} 
  
  \begin{prop} 
    Let $(A,A^+)$ be an affinoid ring and satisfies one of the following conditions. 
    \begin{enumerate}
    \item The ring $A$ is discrete. 
    \item  The ring $A$ is finitely generated over a Noetherian ring of definition.
    \item  The ring $A$ is Tate and for every $n \in \mathbb{N}$, $A\langle{ }X_1,\ldots,X_n\rangle{ }$ is Noetherian. 
    \end{enumerate} 
     Then there exists a unique sheaf $\mathcal{O}_X$ on $X := \mathrm{Spa}(A,A^+)$ 
    such that 
     for every rational subset $U(T/s)$, $\mathcal{O}_X(U(T/s)) = A[1/s]^{\wedge}$ where the topology on $A[1/s]$ is 
     such that $A_0[T/s]$ is a ring of definition
     for any ring of definition $A_0$ in $A$ and 
     $I \cdot A_0[T/s]$ is an ideal of definition for any ideal of definition $I \subset A_0$.   
  \end{prop} 
  
  \begin{rem} 
   \emph{On a general open subset $U \subset X$, we have that 
   \begin{align*} 
    \mathcal{O}_X(U) := \varprojlim_{W \subset U, \text{W rational}} \mathcal{O}_X(W)
   \end{align*} }
  \end{rem} 
  
  \begin{defi} 
  \begin{enumerate} 
  \item \emph{An} affinoid adic space \emph{is an object in the category $(V)$ that is isomorphic in $(V)$ to a space of the form 
  $\mathrm{Spa}(A,A^+)$. 
  \item An} adic space \emph{is an object of the category (V) that is locally isomorphic to an affinoid adic space.} 
  \end{enumerate}
\end{defi}

\subsubsection{Morphisms of adic spaces}

\begin{defi} \label{finite type ring morphism}
 \begin{enumerate} 
 \item \emph{A morphism of affinoid rings $f : (A,A^+) \to (B,B^+)$ is of} topologically finite type \emph{if 
 there exists $n \in \mathbb{N}$, a tuple of finite subsets $M = (M_1,\ldots,M_n)$ such that
 $M_i \cdot A \subset A$ is open and a continuous surjective open ring homomorphism 
 $g : A\langle{ }X_1,\ldots,X_n\rangle{ }_M \to B$ (cf. Example \ref{affinoid ring examples}) which satisfies the following properties.}
 \begin{enumerate} 
 \item \emph{Let $h : A \to A\langle{ }X_1,\ldots,X_n\rangle{ }_M$ be the canonical morphism. 
 We have that $f = g \circ h$.}
 \item \emph{In Example \ref{affinoid ring examples}, we defined a ring of integral elements 
 $C \subset A\langle{ }X_1,\ldots,X_n\rangle{ }_M$ such that $h(A^+) \subset C$. We then have that
 $B^+$ is the integral closure of  
 $g(C)$.}
 \end{enumerate} 
  \item \emph{A morphism $f : X \to Y$ of adic spaces is said to be} locally of finite type \emph{if 
 for every $x \in X$, there exists affinoid open neighbourhoods $U$ of $x$ and $V$ of $f(x)$ with $f(U) \subset V$ such that 
 the map $(\mathcal{O}_Y(V),\mathcal{O}_Y^+(V)) \to (\mathcal{O}_X(U),\mathcal{O}_X^+(U))$ is topologically of finite type.} 
 \emph{A morphism of adic spaces is of} finite type \emph{if it is locally of finite type and quasi-compact i.e. the pullback of a quasi-compact open 
 is quasi-compact.} 
 \end{enumerate} 
\end{defi}

We fix a non-trivially valued non-Archimedean field $k$ that is algebraically closed. 

\begin{rem} \label{rem:adkana}
\emph{A $k$-adic space is an adic space $X$ along with a structure map $p_X : X \to \mathrm{Spa}(k,k^0)$. 
A morphism $f : (X,p_X) \to (Y,p_Y)$ of $k$-adic spaces is a morphism of adic spaces 
$f : X \to Y$ such that $p_X = p_Y \circ f$. Let $k-ad$ denote the category of $k$-adic spaces.  
Observe that if $X$ is a $k$-adic space and $x \in X$ then there exists an affinoid neighbourhood $U$ of $x$ such that 
$O_X(U)$ is Tate. Hence, a $k$-adic space is a particular instance of an} analytic adic space. 
\end{rem}

As stated earlier, we would like to define a sub-category of the category of $k$-adic spaces which possess certain favourable properties. 
 Before doing so, we briefly discuss the notion of a separated morphism. 
 
\begin{defi} \label{definition separated}
  \emph{A morphism $f : X \to Y$ of adic spaces which is locally of finite type is said to be} separated \emph{if 
  the image of the diagonal map is closed i.e. if 
  $\Delta : X \to X \times_Y X$ is the diagonal map then
  $\Delta(X)$ is closed in $X \times_Y X$.
 }
\end{defi} 

\begin{rem}  \label{diagonal map is closed}
\emph{If $f : X \to Y$ is separated then the morphism $\Delta : X \to X \times_Y X$ is a homeomorphism of $X$ onto its image. 
Indeed, it suffices to show that the morphism $\Delta$ is closed. 
Let $Z \subset X$ be a closed set
and let $\overline{\Delta(Z)}$ denote its closure in $X \times_Y X$.
Since $f$ is separated 
$\overline{\Delta(Z)}$ is contained in $\Delta(X)$.  
We claim that $\Delta(Z) = \overline{\Delta(Z)}$. 
Suppose there exists $z \in X$ such that 
$\Delta(z) \in \overline{\Delta(Z)} \smallsetminus \Delta(Z)$. 
As $Z$ is closed in $X$, we have that 
there exists an open set $O$ in $X$ that contains $z$ and is disjoint from $Z$.
The set $O \times_Y O \subset X \times_Y X$ is an open neighbourhood of $\Delta(z)$ that 
does not intersect $\Delta(Z)$.
Indeed, $O \times_Y O$ is open since if $p_1$ and $p_2$ denote the projections $X \times_Y X \to X$ then 
$O \times_Y O = p_1^{-1}(O) \cap p_2^{-1}(O)$.  
 Hence $\Delta(z)$ cannot lie in $\overline{\Delta(Z)}$ and this contradicts our 
initial claim.}
\end{rem}

\begin{lem} \label{lem:sepcompsd}
Let $f \colon X \to Y$ and $g \colon Y \to Z$ be locally of finite type morphisms of $k$-adic spaces. Suppose $g \circ f$ is separated. Then so is $f$.
\end{lem}

\begin{proof}
Since $g \circ f$ is separated, it is also quasi-separated. We show first that $f$ is quasi-separated.
This is equivalent to saying that the diagonal morphism $\Delta_f : X \to X \times_Y X$ is quasi-compact. 
To show that $\Delta_f$ is quasi-compact, we show that there exists an affinoid covering $\{U_i\}_i$ of 
$X \times_Y X$ such that $f^{-1}(U_i)$ is quasi-compact for every $i$ cf. \cite[Tag 01K4]{stacks-project}.  

We choose affinoid coverings $\{C_t\}$ of $Z$, $\{B_s\}$ of $Y$ and $\{A_r\}$ of $X$ such that 
for every $r$ there exists $s$ such that the image of $A_r$ is contained in $B_s$ and for 
every $s$, the image of $B_s$ is contained in $C_t$ for some $t$. Furthermore, for every 
$s$, $f^{-1}(B_s)$ is covered by some sub-family of elements in $\{A_r\}$. 
It follows that $X \times_Y X$ is covered by sets of the form 
$A_{r_1} \times_{B_s} A_{r_2}$ which is affinoid by construction. The pre-image of 
$A_{r_1} \times_{B_s} A_{r_2}$ in $X$ is $A_{r_1} \cap A_{r_2}$ which is quasi-compact because 
the map $X \to Z$ is quasi-separated and hence after choosing an appropriate $t$,
 the preimage of $A_{r_1} \times_{C_t} A_{r_2}$ in $X$ which is $A_{r_1} \cap A_{r_2}$ is quasi-compact 

 To complete the proof, we use Huber's \emph{valuative criterion for separatedness}, cf. \cite[Proposition 1.3.7]{hub96}. Let $(x,A)$ be a valuation ring of $X$ and suppose that it possesses two distinct centers $z_1 \not= z_2 \in X$ such that $f(z_1) = f(z_2)$. Then $(g \circ f)(z_1) = (g \circ f)(z_2)$, which contradicts separatedness of $g \circ f$.

\end{proof}

\begin{lem} \label{lem:septaut} \label{lem:adlfttf}
Let $X$ and $Y$ be separated finite type $k$-adic spaces i.e.
the structure morphisms $p_X : X \to \mathrm{Spa}(k,k^0)$ and 
$p_Y : Y \to \mathrm{Spa}(k,k^0)$ are separated and finite type.
If $f : X \to Y$ is a morphism of $k$-adic spaces 
then the morphism $f$ is separated, taut and of finite type \footnote{In the language of \cite{hub96}, this implies that $f$ is locally of ${}^+$weakly finite type.}.
\end{lem}
\begin{proof}
By \cite[Lemma 5.1.3(i)]{hub96} a separated finite type $k$-adic space is taut and by Lemma 5.1.3(iii) in loc.cit. a morphism of separated finite type $k$-adic spaces is taut.  The fact that $f$ is separated follows from Lemma \ref{lem:sepcompsd}.

We show that the morphism $f$ is locally of finite type. 
First note that $f$ comes from a morphism of rigid analytic varieties by \cite[Proposition 4.5(iv)]{hub93}. Working locally on $X$, it suffices to show that a morphism of $k$-Tate rings of the form $f \colon k \langle \overline{T} \rangle/\mathfrak{a} \to k \langle \overline{U} \rangle/\mathfrak{b}$ is topologically of finite type\footnote{Here $k\langle \overline{T}\rangle := k\langle T_1, T_2, \ldots, T_m  \rangle$ for some $m \geq 0$ and $\mathfrak{a} \subset k \langle \overline{T} \rangle$ is an ideal. Similarly $k\langle \overline{U}\rangle := k\langle U_1, U_2, \ldots, U_n  \rangle$ for some $n \geq 0$ and $\mathfrak{b} \subset k \langle \overline{U} \rangle$ is an ideal.}. This follows from Lemma 3.3(iii) in loc.cit. (the point is that the morphism induced by $f$,  $\left( k \langle \overline{T} \rangle/\mathfrak{a} \right) \langle \overline{U} \rangle \to k \langle \overline{U} \rangle/\mathfrak{b}$ is a continuous, surjective and open ring homomorphism. It is open because of the open mapping theorem).

  Finally, the morphism $f$ is quasi-compact by Lemma \ref{lem:fineadqcshi}.
\end{proof}

\begin{lem} \label{lem:fineadqcshi}
Let $f \colon X \to Y$ and $g \colon Y \to Z$ be locally of finite type morphisms of $k$-adic spaces such that $g \circ f$ is quasi-compact and $g$ is quasi-separated. Then $f$ is quasi-compact.
\end{lem}

\begin{proof}
The proof is formal. Observe that $f$ factorises as $X \xrightarrow{(1,f)} X \times_Z Y \to Y$ (the fiber product in question exists by \cite[Proposition 1.2.2]{hub96}), where the second morphism is projection. The second morphism is quasi-compact because it is base change of $g \circ f$ (cf. Corollary 1.2.3(iii) in loc.cit.). The first morphism sits in a cartesian diagram
$$
\begin{tikzcd}[row sep = large, column sep = large]
X \arrow[r, "{(1,f)}"] \arrow[d] &
X \times_Z Y \arrow[d] \\
Y \arrow[r, "\Delta_g"]
&
Y \times_Z Y.
\end{tikzcd}
$$
Since $g$ is quasi-separated, $\Delta_g$ is quasi-compact and again by base change $(1,f)$ is quasi-compact. The composition of quasi-compact morphisms is again quasi-compact and this proves the lemma.
\end{proof}

      As outlined in the introduction, we develop a theory of constructible sheaves via the nearby cycles functor which requires 
      that the adic spaces we deal with admit quasi-compact formal models. Hence, we restrict our attention to 
      a sub-class of adic spaces which we call \emph{fine}. 

\begin{defi} 
\emph{Let} $Fine_{k-ad}$ \emph{denote the category of fine $k$-adic spaces. The objects of $Fine_{k-ad}$
 are separated finite type $k$-adic spaces and the maps between objects are morphisms of $k$-adic spaces.}  
\end{defi}


\begin{lem} \label{lem:morsitfin}
The category $Fine_{k-ad}$ admits fiber products and the canonical functor $Fine_{k-ad} \rightarrow k-ad$ preserves fiber products. 
\end{lem}

\begin{proof}
First, note that by Lemma \ref{lem:adlfttf} and \cite[Proposition 1.2.2(a)]{hub96}, fiber products of fine $k$-adic spaces exists in $k-ad$.
Hence, it suffices to verify that the fiber product of fine $k$-adic spaces in the category of $k$-adic spaces is fine.
 Let $X$, $X'$ and $X''$ be fine $k$-adic spaces sitting in a cartesian diagram
$$
\begin{tikzcd}[row sep = large, column sep = large]
X' \times_X X'' \arrow[r, "\tilde{g}"] \arrow[d, "\tilde{f}"] &
X'' \arrow[d, "f"] \\
X' \arrow[r, "g"]
&
X.
\end{tikzcd}
$$
Since $f$ and $g$ are of finite type (Lemma \ref{lem:adlfttf}), so are $\tilde{f}$ and $\tilde{g}$, cf. Corollary 1.2.3 (i), (iiic) in loc.cit. A composition of finite type morphisms is of finite type and so $X' \times_X X''$ is of finite type over $k$.
 It remains to prove $X' \times_X X''$ is separated. This follows from Lemma 1.10.17(iii) in loc.cit. and Lemma \ref{lem:septaut}.
\end{proof}

\subsubsection{Rigid varieties, Berkovich spaces and adic spaces}
\subsection{Rigid varieties and adic spaces} \label{rigid varieties and adic spaces}

     We have a fully faithful functor, cf. \cite[1.1.11]{hub96}:
		\begin{align*}
		\left\{ \text{rigid-analytic varieties}/k \right\} &\rightarrow \left\{ \text{adic spaces}/k \right\} \\
		X &\mapsto X^{\ad}
		\end{align*}
		sending $\Sp(R)$ to $\Spa(R,R^{+})$ for any affinoid $k$-algebra $(R,R^{+})$ of topologically finite type (tft). It induces an equivalence
	 \begin{align*} 
	 \left\{\text{quasi-compact quasi-separated rigid-analytic varieties}/k \right\} \\ \cong 
	 	  \{\text{quasi-compact quasi-separated adic spaces of finite type}/k \}.
		  \end{align*} 
	There is also an equivalence of categories, cf. \cite[Proposition 8.3.1]{hub96}) :
	\begin{align*}
	\left\{ \text{Hausdorff strictly} \hspace{1mm} k-\text{analytic Berkovich spaces} \right\}  \\
	\cong \left\{\text{taut adic spaces locally of finite type}/k \right\}
	\end{align*}
	sending $\mathcal{M}(R)$ to $\Spa(R,R^+)$ for any affinoid $k$-algebra $(R,R^{+})$ of tft. It induces equivalences
	\begin{align*}
	\left\{ \text{compact Hausdorff strictly} \hspace{1mm} k-\text{analytic Berkovich spaces} \right\} \\
	\cong \left\{\text{quasi-separated adic spaces of finite type}/k \right\}
	\end{align*}
	and 
	\begin{align*}
	\left\{ \text{compact separated strictly} \hspace{1mm} k-\text{analytic Berkovich spaces} \right\} \\
	\cong \left\{\text{separated adic spaces of finite type}/k \right\}
	\end{align*}
	Putting everything together we get the following equivalence of categories where Raynaud's theorem applies:
	\begin{align*}
	\left\{ \text{fine} \hspace{1mm} k- \text{adic spaces} \right\} \\
	\cong \left\{ \text{compact separated strictly} \hspace{1mm} k-\text{analytic Berkovich spaces} \right\}\\
	\cong \left\{\text{separated quasi-compact rigid-analytic varieties}/k \right\}.
	\end{align*}

\subsection{The fine étale site} 

\begin{defi} \label{def:2.1}
\begin{enumerate}
\item \emph{A morphism $(R,R^+) \to (S,S^+)$ of affinoid $k$-algebras is called} finite étale \emph{if $S$ is a finite étale $R$-algebra with the induced topology and $S^+$ is the integral closure of $R^+$ in $S$.}
\item \emph{A morphism $f \colon X \to Y$ of adic spaces over $k$ is called} finite étale \emph{if there is a cover of $Y$ by open affinoids $V \subset Y$ such that the preimage $U = f^{-1}(V)$ is affinoid and the associated morphism of affinoid $k$-algebras
\[
(\mathcal{O}_Y(V), \mathcal{O}_Y^+(V)) \to (\mathcal{O}_X(U), \mathcal{O}_Y^+(U))
\]
is finite étale.}
\item \emph{A morphism $f \colon X \to Y$ of adic spaces over $k$ is called} étale \emph{if for any point $x \in X$ there are open affinoid neighbourhoods $U$ and $V$ of $x$ and $f(x)$ respectively, and a commutative diagram} 
$$
\begin{tikzcd}[row sep = large, column sep = large]
U \arrow[rd, "f|_U"] \arrow[r, "j"] &
W \arrow[d, "p"] \\
&
V  
\end{tikzcd}
$$
\emph{where $j$ is an open embedding and $p$ is finite étale.} 
\end{enumerate}
\end{defi} 

\begin{rem}
\emph{\cite[Definition 1.6.5(i)]{hub96} is equivalent to Definition \ref{def:2.1} for the adic spaces that we will consider and we will use both definitions when convenient, cf. Lemma 2.2.8 in loc.cit.}
\end{rem}

\begin{defi} 
\emph{Let $X$ be a fine $k$-adic space. We define $X^{f}_{\text{ét}}$ to be the site $X_\text{ét}^{f}$
generated by those objects of $X_{\text{ét}}$ which are fine $k$-adic 
 i.e. 
the objects of $X_{\text{ét}}^f$ are étale morphisms $\alpha : Y \to X$ where $Y$ is fine $k$-adic and given $Y \to X$ in $X_{\text{ét}}^f$, we say that 
$\{(Y_i \to X) \to (Y \to X)\}_i$ is a covering if it is a covering in $X_\text{ét}$.
Given a torsion ring $\Lambda$, let $\mathcal{D}^+(X^f_{\text{ét}},\Lambda)$ denote the bounded below derived category of fine étale sheaves 
of $\Lambda$-modules. 
} 
\end{defi}

\subsubsection{A Comparison of sites} 

      In \cite{hub96}, Huber works with the étale site of an adic space and 
      develops a six functor formalism in this setting\footnote{It should be noted that certain hypothesis are required to construct the functors $f_!$ and $f^!$.}.
      We would like to make use of these constructions for fine $k$-adic spaces. 
      
    \begin{prop} \label{comparison theorem}
      Let $X$ be a fine $k$-adic space. There exists a canonical morphism of sites (cf. Lemma \ref{lem:morsitfin})
       \begin{align*} 
        u_X : X_{\text{ét}} \to X^f_{\text{ét}}
       \end{align*}
      such that the induced morphism of topoi 
      $\widetilde{u_X} := (u_X^*,u_{X*}) :  \widetilde{X^f_{\text{ét}}} \to \widetilde{X_{\text{ét}}}$ is an equivalence of categories.  
    \end{prop}   
\begin{proof} 
  The proposition follows from Lemma \ref{lem:2.14} and \cite[Exposé III, Théor\`{e}me 4.1]{sga4tome1}. 
\end{proof} 

\begin{lem} \label{lem:2.14}
Let $Z$ be a $k$-adic space, which is not necessarily fine. Suppose there is an étale morphism $g \colon Z \to X$. Then there exists an étale covering $Z_i \to Z$ such that the $Z_i$ are fine $k$-adic spaces.
\end{lem}

\begin{proof}
An étale morphism is by definition locally of finite type (in fact it is locally of finite presentation). Thus $Z$ is locally of finite type over $k$ and so we can take a covering $g_i \colon Z_i' \to Z$ where the $Z_i'$ are of finite type. 
If the $Z'_i$ were separated over $k$ then we would be done. To complete the proof, we refine each of the $Z'_i$ by affinoid spaces that are separated over $X$. 

Working locally for each point $z_{ij}' \in Z_i'$, there are open affinoid neighbourhoods $Z_{ij}$ and $V_{ij}$ of $z_{ij}'$ and $g(g_i(z_{ij}'))$ respectively, and a commutative diagram
$$
\begin{tikzcd}[row sep = large, column sep = large]
Z_{ij} \arrow[rd] \arrow[r, "j"] &
W \arrow[d, "p"] \\
&
V_{ij}  
\end{tikzcd}
$$
where $j$ is an open embedding and $p$ is finite étale. In particular $g \circ g_i \lvert_{Z_{ij}}$ is separated. Thus the the $Z_{ij}$ are fine $k$-adic spaces and cover $Z$.
\end{proof}

\begin{thm}
   Let $\phi : X \to Y$ be a morphism of fine $k$-adic spaces. There exists a functor
   \begin{align*} 
     R\phi_! : \mathcal{D}^+(X^f_{\text{ét}},\Lambda) \to D^+(Y^f_{\text{ét}},\Lambda)
   \end{align*} 
 which satisfies the following properties. 
 \begin{enumerate} 
 \item Let $\phi_! := R^0\phi_!$. If $\sF \in {Sh}(X^f_{\text{ét}},\Lambda)$ then 
 \begin{align*} 
 \phi_!(\sF)(U) := \{s \in \Gamma(X \times_Y U) | \mathrm{supp}(s) \text{ is proper over } U\}.
 \end{align*} 
 \item Assume the morphism $\phi$ is étale. We have that $\phi_!$ is exact. 
 Furthermore, if 
  $\sF \in {Sh}(X^f_{\text{ét}},\Lambda)$ and $\sG \in Sh(Y^f_{\text{ét}},\Lambda)$ then 
 \begin{align*} 
   \mathrm{Hom}_{ {Sh}(Y^f_{\text{ét}},\Lambda)}(\phi_!(\sF),\sG) =  \mathrm{Hom}_{ {Sh}(X^f_{\text{ét}},\Lambda)}(\sF,\phi^*(\sG)).
 \end{align*} 
 \item If the morphism $\phi$ is partially proper 
 then $R\phi_!$ is the right derived functor of $\phi_!$. 
  \end{enumerate} 
\end{thm} 
\begin{proof}
  This is a direct consequence of \cite[Theorem 5.4.3]{hub96}, Lemma \ref{lem:septaut} and Proposition \ref{comparison theorem}. 
\end{proof} 

\begin{thm} \label{upper shriek}
 Let $\phi : X \to Y$ be a morphism of fine $k$-adic spaces. There exists a functor
   \begin{align*} 
     \phi^! : \mathcal{D}^+(Y^f_{\text{ét}},\Lambda) \to D^+(X^f_{\text{ét}},\Lambda)
   \end{align*} 
such that if $A \in \mathcal{D}^+(X^f_{\text{ét}},\Lambda)$ and 
$B \in D^+(Y^f_{\text{ét}},\Lambda)$ then there is a functorial isomorphism
\begin{align*}
 \mathrm{Hom}_{D^+(Y^f_{\text{ét}},\Lambda)}(R\phi_!(A),B) =  \mathrm{Hom}_{D^+(X^f_{\text{ét}},\Lambda)}(A,\phi^!(B)).
\end{align*} 
In particular, the functor $\phi^!$ is right adjoint to $R\phi_!$.  
\end{thm} 
\begin{proof}
   We claim that $\mathrm{dim.tr}(\phi) < \infty$. Indeed $\mathrm{dim.tr}(\phi) = \mathrm{dim}(\phi) \leq \mathrm{dim}(X)$, where the first equality follows from \cite[Corollary 1.8.7(ii)]{hub96}. But $X$ is in particular quasi-compact and so we can assume $X$ is affinoid. The claim then follow from Lemma 1.8.6 in loc.cit.
    
  The theorem then follows from Theorem 7.1.1 in loc.cit., Lemma \ref{lem:septaut} and Proposition \ref{comparison theorem}.
\end{proof}

\subsection{Formal schemes}

  The theory of formal schemes allows us to relate the étale site of a fine $k$-adic space and 
  the étale sites of certain $\tilde{k}$-varieties (here $\tilde{k}$ is the residue field of $k$). We refer the reader to \cite[\S 2.2]{schweinpadiv} for more details on how to describe the generic fiber of a formal scheme. In what follows, we provide a brief introduction to formal geometry and 
  define the étale site of the formal schemes we will be interested in. We end the subsection by introducing the 
  \emph{nearby cycles functor}, the properties of which we exploit later on. 

\subsubsection{Affine formal schemes} 
  
  We make use of the following notation. Let $k^0 := \{x \in k | |x| \leq 1\}$ and 
  $k^{00} := \{x \in k | |x| < 1\}$ (so that $\tilde{k} = k^0/k^{00}$). The ring $k^0$ is referred to as the ring of \emph{power bounded elements} of 
  $k$ while $k^{00}$ is called the ring of \emph{topologically nilpotent elements} of $k$. 
  We will restrict our attention to formal schemes defined over $k^0$. 
  We fix a topological nilpotent element $\pi$ that is non-zero if the valuation on 
  $k$ is non-trivial and zero if it is trivial. 
  
    As in classical algebraic geometry, a formal scheme over $k^0$ is obtained by glueing together affine formal schemes while an affine formal scheme over $k^0$
    is the formal spectrum of a $k^0$-algebra topologically of finite presentation where the topology is defined by $(\pi)$.
      
\begin{defi} 
\begin{enumerate}
\item  \emph{Let $X_1,\ldots,X_n$ be a set of variables. We define the} ring of restricted power series over $k^0$ \emph{
\begin{align*} 
k^{0}\langle X_1,\ldots,X_n\rangle := \left\{ \sum_{\mathbf{i} := (i_1,\ldots,i_n) \in \mathbb{N}^n} a_{\mathbf{i}} X_1^{i_1}\ldots X_n^{i_n} | |a_{\mathbf{i}}| \mapsto 0 \text{ as } |i| \mapsto \infty \right\}.
\end{align*}  }
 \item \emph{A $k^0$-algebra $A$ is} topologically of finite presentation \emph{if for some $n \in \mathbb{N}$, 
 $A$ is isomorphic to $k^{0}\langle X_1,\ldots,X_n\rangle /I$ where $I$ is a finitely generated ideal in $k^{0}\langle X_1,\ldots,X_n\rangle$.}
   \end{enumerate} 
\end{defi}    

We often simplify notation by setting $\overline{X} := (X_1,\ldots,X_n)$ 
and using $k^0\langle \overline{X}\rangle$ to denote the algebra
$k^{0}\langle X_1,\ldots,X_n \rangle$.  

\begin{rem} \label{projective limit of nilpotent}
\emph{Observe that}
\begin{align*} 
  k^0\langle \bar{X}\rangle = \varprojlim_{m \in \mathbb{Z}_{\geq 1}} k^0[\bar{X}]/ (\pi^m). 
\end{align*} 
\end{rem} 

 \begin{defi}
   \emph{Let $A$ be a $k^0$-algebra which is topologically of finite presentation. We define the} formal spectrum \emph{of $A$ to be
   \begin{align*} 
     \mathrm{Spf}(A) := \varinjlim_m \mathrm{Spec}(A_m)
   \end{align*} 
   where $A_m := A /(\pi^m)$ and the limit is taken in the category of locally ringed spaces.}
 \end{defi} 
 
   Given a $k^0$-algebra $A$ which is topologically of finite presentation, one can think of 
   $\mathrm{Spf}(A)$ as a topological space $\mathrm{Spec}(A \times_{k^0} \tilde{k})$ along 
   with a structure sheaf that contains information about an infinitesimal neighbourhood of $\mathrm{Spec}(A \times_{k^0} \tilde{k})$. 
   We refer to $\mathrm{Spec}(A \times_{k^0} \tilde{k})$ as the special fibre of $\mathrm{Spf}(A)$.  

    \begin{defi}
      A formal scheme over $k^0$ \emph{is a locally ringed space that is locally isomorphic to an affine formal scheme 
      of the form $\mathrm{Spf}(A)$ where $A$ is topologically of finite presentation. 
      A morphism of formal schemes is a morphism of locally ringed spaces. 
      } 
    \end{defi} 
    
    \begin{rem} 
      \emph{Let $\mathfrak{X}$ be a $k^0$-formal scheme. We can associate to $\mathfrak{X}$ a $\tilde{k}$-scheme $\mathfrak{X}_s$ 
      which we call the special fibre of $\mathfrak{X}$ and an adic space $\mathfrak{X}_\eta$ over $k$ which we call 
      the generic fibre of $\mathfrak{X}$. 
       Suppose $\mathfrak{X}$ was of the form $\Spf(A)$ then $\mathfrak{X}_s$ is the scheme
       $\mathrm{Spec}(A \otimes_{k^0} \tilde{k})$ and 
       $\mathfrak{X}_{\eta}$ is the adic space $\mathrm{Spa}(A, A) \times_{\mathrm{Spa}(k^0, k^0)} \mathrm{Spa}(k, k^0)$.
       In general, the special fibre (generic fibre) of a formal scheme $\mathfrak{X}$ is obtained by glueing together the special fibres (generic fibres)
       of a family of affine formal schemes which cover $\mathfrak{X}$. 
       In such a situation, we say that $\mathfrak{X}$ is a formal model of $\mathfrak{X}_\eta$.
       If $\mathfrak{X}$ is separated and quasi-compact over $k^0$ then 
       $\mathfrak{X}_\eta$ is a fine $k$-adic space while $\mathfrak{X}_s$ is a variety over $\tilde{k}$.} 
    \end{rem} 
 
\subsubsection{Étale site of a formal scheme} 
 
 Let $\mathfrak{X}$ denote a formal scheme over $k^0$ with structure sheaf $\mathcal{O}_{\mathfrak{X}}$. Let 
 $\mathfrak{X}_n$ be the scheme $(|\mathfrak{X}|, \mathcal{O}_{\mathfrak{X}}/ \pi^n\mathcal{O}_{\mathfrak{X}})$.

\begin{defi}
 \emph{A morphism of formal schemes over $k^0$, $\phi : \mathfrak{X} \to \mathfrak{Y}$ is étale if for every 
 $n \in \mathbb{N}$, the induced morphism of schemes 
 $\phi_n : \mathfrak{X}_n \to \mathfrak{Y}_n$ is étale.}
\end{defi}

\begin{prop} \label{prop:canfumosishi}
Let $\mathfrak{X}$ be a $k^0$-formal scheme. 
  \begin{enumerate} 
  \item The correspondence $\mathfrak{Y} \mapsto \mathfrak{Y}_s$ induces an equivalence between the category of formal schemes étale over 
  $\mathfrak{X}$ and schemes étale over $\mathfrak{X}_s$. 
  \item Let $f : \mathfrak{Y} \to \mathfrak{X}$ be an étale morphism of formal schemes. We then have that the induced map of generic fibres 
  $f_\eta : \mathfrak{Y}_\eta \to \mathfrak{X}_\eta$ is étale. 
   In particular, we have a morphism of sites 
   $\nu_{X,\mathfrak{X}}: X^f_{\text{ét}} \rightarrow \mathfrak{X}_{s,\text{ét}}$. 
 \end{enumerate} 
\end{prop} 
\begin{proof} 
 See \cite[Lemma 2.1]{berk94} and \cite[Lemma 3.5.1]{hub96}. 
\end{proof} 

 \subsubsection{Nearby cycles functor} \label{nearby cycles functor}

\begin{defi} (Nearby cycles functor) 	 
    \emph{Let $X$ be a fine $k$-adic space and $\mathfrak{X}$ be a formal model of $X$. 
   We have the following functor: }

   \begin{align*} 
   \psi_{\mathfrak{X}} : &  \widetilde{X^f_{\text{et}}} \to \widetilde{\mathfrak{X}_{s,et}} \\
                                        &\mathscr{F} \mapsto \nu_{X,\mathfrak{X}*}(\mathscr{F}).
   \end{align*} 
\end{defi}

\begin{lem} \label{vanishing cycles commutes with lower shriek}
 Let $\Lambda$ be a torsion ring. Let $f : X \to Y$ be a morphism of fine $k$-adic spaces. 
 Let $\mathfrak{X}$ and $\mathfrak{Y}$ be formal models of $X$ and $Y$. We suppose
  that there exists a 
 morphism $\mathfrak{f} : \mathfrak{X} \to \mathfrak{Y}$ such that the induced morphism 
 $\mathfrak{f}_{\eta}$ between the generic fibres coincides with $f$. Let 
 $\mathfrak{f}_s : \mathfrak{X}_s \to \mathfrak{Y}_s$ denote the morphism between the respective 
 special fibres. For $\mathscr{F} \in \mathcal{D}^+(X^f_{\text{ét}},\Lambda)$ one has 
 \begin{align*} 
    R\mathfrak{f}_{s!} \circ R\psi_{\mathfrak{X}}(\mathscr{F}) \xrightarrow{\sim} R\psi_{\mathfrak{Y}} \circ Rf_!(\mathscr{F}). 
 \end{align*}  
\end{lem}

\begin{proof}
The main ideas can be found in \cite[\S 3.5]{hub96}. By Corollary 5.1.13 in
 loc.cit. $f$ has a compacitification $f = g \circ j$ where $j \colon X\to X'$ is an 
 open embedding and $g \colon X' \to Y$ is proper. Note that $X'$ is a fine $k$-adic space. 
 By \cite[Proposition 8.2.16, Lemma 8.4.4]{BO}, there exists formal models $\mathcal{X}_1$ and $\mathcal{X}'$ 
 of $X$ and $X'$ respectively such that we have an admissible blow-up $\mathfrak{b} : \mathcal{X}_1 \to \mathfrak{X}$, 
 a morphism of formal schemes $\mathfrak{f}_1 : \mathfrak{X}_1 \to \mathfrak{Y}$, an open immersion $\mathfrak{j} : \mathfrak{X}_1 \to \mathfrak{X}'$ and 
 a map $\mathfrak{g} : \mathfrak{X}' \to \mathfrak{Y}$. 
 Furthermore, these morphisms are such that $\mathfrak{j}_{\eta} = j$, $\mathfrak{g}_\eta = {g}$, $\mathfrak{f}_1 = \mathfrak{g} \circ \mathfrak{j}$.  
 Here $g$ is the generic fibre of $\mathfrak{g}$ and by \cite[Remark 1.3.18(ii)]{hub96} $\mathfrak{g}$ is proper.


By Corollary 3.5.11(ii) in loc.cit., we see that 
\begin{align*} 
  R\mathfrak{j}_{s!} \circ R\psi_{\mathfrak{X}}(\sF) \xrightarrow{\sim} R\psi_{\mathfrak{X}'} \circ j_! (\sF).
\end{align*} 

We show that 
\[
R\mathfrak{g}_{s*} \circ R\psi_{\mathfrak{X'}}(\mathscr{F}) \xrightarrow{\sim} R\psi_{\mathfrak{Y}} \circ Rg_*(\mathscr{F}).
\] 
By Grothendieck's theorem on the composition of derived functors, it suffices to verify that 
\[
\mathfrak{g}_{s*}(\nu_{\mathfrak{X}'*}(\mathscr{F})) \cong \nu_{\mathfrak{Y}*}(g_*(\mathscr{F})).
\]
 We need only check that $g_s \circ \nu_{\mathfrak{X}'} = \nu_{\mathfrak{Y}} \circ g$ as a morphism of sites .
This reduces to checking that for a formal scheme $\mathfrak{Z} \to \mathfrak{Y}$, 
$(\mathfrak{X}' \times_{\mathfrak{Y}} \mathfrak{Z})_{\eta} = \mathfrak{X}'_{\eta} \times_{\mathfrak{Y}_\eta} \mathfrak{Z}_\eta$. 
This can be further reduced to the case when the formal schemes are all affine, in which case the identity is clear. 
 Since $\mathfrak{f}_1 = \mathfrak{g} \circ \mathfrak{j}$, we have thus shown that 
  \begin{align*} 
  R\mathfrak{f}_{1s!} \circ R\psi_{\mathfrak{X}_1}(\sF) \xrightarrow{\sim} R\psi_{\mathfrak{Y}} \circ Rf_{1!}(\sF)  
  \end{align*} 
   where $f_1 := \mathfrak{f}_{1\eta}$. 
 By construction, $\mathfrak{f}_1 = \mathfrak{f} \circ \mathfrak{b}$. Since, $\mathfrak{b}$ is an admissible blow up, we get that 
   \begin{align*} 
  R\mathfrak{f}_{s!} \circ R\mathfrak{b}_{s!} \circ R\psi_{\mathfrak{X}_1}(\sF) \xrightarrow{\sim} R\psi_{\mathfrak{Y}} \circ Rf_{!}(\sF)  
  \end{align*}
      Since $\mathfrak{b}_s$ is proper and using the isomorphism $R\mathfrak{b}_{s*} \circ R\psi_{\mathfrak{X}_1} \to R\psi_{\mathfrak{X}}$
  we deduce that
    \begin{align*} 
  R\mathfrak{f}_{s!} \circ R\psi_{\mathfrak{X}}(\sF) \xrightarrow{\sim} R\psi_{\mathfrak{Y}} \circ Rf_{!}(\sF). 
  \end{align*}
\end{proof}

\begin{lem} \label{lem:fincohdimnecy}
 Let $X$ be a fine $k$-adic space. 
  Let $\sF \in \mathcal{D}^b(X^f_{\text{ét}},\Lambda)$ and $\mathfrak{X}$ be a formal model of $X$. 
   We then have that $R\psi_\mathfrak{X}(\sF)$ is a bounded complex on $\mathfrak{X}_{s,et}$. 
\end{lem} 	 
\begin{proof} 
    It suffices to show that there exists $d \in \mathbb{N}$ such that
    for every $Y \to \mathfrak{X}_s$ étale, we have that
    $R\Gamma(\mathfrak{U}_s,R\psi_\mathfrak{X}(\sF)) \in \mathcal{D}^{[-d,d]}(\Lambda)$. 
Observe that there exists a formal scheme $\mathfrak{U}$ étale over $\mathfrak{X}$ with special fibre $\mathfrak{U}_s = Y$ and  
$R\Gamma(\mathfrak{U}_s,R\psi_\mathfrak{X}(\sF)) = R\Gamma(\mathfrak{U}_\eta,\sF)$. 
 By \cite[Corollary 2.8.3]{hub96}, one deduces that there exists $d \in \mathbb{N}$ such that 
for every $Y \to \mathfrak{X}_s$ étale, $R\Gamma(Y,R\psi_\mathfrak{X}(\sF)) \in \mathcal{D}^{[-d,d]}(\Lambda)$. 
This completes the proof. 
\end{proof}

\section{Constructible sheaves} 	 
	
		 We fix a prime number $l$ which is coprime to the characteristic of the residue field $\tilde{k}$ of $k$.
		 Let $\Lambda := \mathbb{Z}/ \ell^n \mathbb{Z}$.  
   
   \begin{defi} \label{constructible sheaves} 
     \emph{Let $X$ be a fine $k$-adic space and
      $\mathscr{F}$ be a sheaf of $\Lambda$-modules on the site $X^{f}_{\text{ét}}$.}
    \begin{enumerate}
    \item \emph{We say that $\sF$ is} semi-constructible \emph{if for every 
    $f : U \to X$ in $X^{f}_{\text{ét}}$ and every formal model $\mathfrak{U}$ of $U$, we have that
        $R\psi_{\mathfrak{U}}(f^{*}\sF) \in \mathcal{D}^b_c(\mathfrak{U}_s,\Lambda)$.} 
    \item \emph{The sheaf $\sF$ is} constructible \emph{if for every complete algebraically closed non-Archimedean field extension $L$
 and every morphism of fine $L$-adic spaces $f : Y \to X_L$, we have
 that $f^* \circ p_L^*(\sF)$ is semi-constructible. Here $p_L \colon X_L \to X$ is the projection.} 
    \end{enumerate} 
     \emph{We use} $Con(X^{f}_{\text{ét}},\Lambda)$ \emph{to denote the full sub-category of constructible sheaves 
     and} $sCon(X^{f}_{\text{ét}},\Lambda)$ \emph{to denote the full sub-category of semi-constructible sheaves}
   \end{defi}

   \begin{rem} \label{huber constructible is constructible}
    \emph{It is reasonable to ask if the class of \emph{constructible} sheaves
     defined above contains any sheaves of interest. By \cite[Proposition 2.12]{hub98a}, 
    the class of constructibles defined above extends the class 
    of constructible sheaves introduced by Huber in \cite[Definition 2.7.2]{hub96}. Henceforth, we will refer to these 
    sheaves as Huber constructible.}
   \end{rem}  
   
       It should be noted that there exist semi-constructible sheaves which are not constructible. For instance cf. \S \ref{counterexample for Verdier dual stability}.
    
%

\begin{es} \label{es:coexserrsubcat}
 \emph{Consider the following example of a 
 sub-sheaf of a constructible sheaf that is not constructible. 
 Let $k = \mathbb{C}_p$ and $X = \mathbb{P}^{1,\mathrm{ad}}_{\mathbb{C}_p}$. 
  Let $\Omega := X \smallsetminus \mathbb{P}^1(\mathbb{Q}_p)$. The subspace $\Omega$ is open subspace of
  $X$ and let $j : \Omega \hookrightarrow \mathbb{P}^{1,\mathrm{ad}}_{\mathbb{C}_p}$ denote the open immersion. 
   We set $G := j_{!}j^*\Lambda$. 
  It can be deduced from \cite[Théor\`{e}me 3.1.1]{dat06} that $H^1(X^{f}_{\text{ét}},G)$ is not finite
   since it corresponds to the Steinberg representation, while $H^2(X^{f}_{\text{ét}},G)$ corresponds to a twist of the trivial representation. 
  By Corollary \ref{constructible implies finite global sections}, it follows that 
  $G$ is not constructible. Hence we see that the sub-category
   $Con(X^{f}_{\text{ét}},\Lambda)$ is not necessarily a Serre sub-category of the category of sheaves of $\Lambda$-modules on $X$.}
  \end{es}   
   
    \begin{lem} 
   Let $X$ be a fine $k$-adic space. 
   The category $Con(X^{f}_{\text{ét}},\Lambda)$ is exact. 
  \end{lem}
	\begin{proof} 
	It suffices to show that the category $Con(X^{f}_{\text{ét}},\Lambda)$ is stable for extensions. 
	Let $0 \to \sF \to \sF' \to \sF'' \to 0$ be a short exact sequence of constructible $\Lambda$-étale sheaves on $X$.
	Let $f : Y \to X_L$ be a morphism of fine $L$-adic spaces where $L$ is a non-Archimedean algebraically closed complete field extension of $k$. Let $p_L \colon X_L \to X$ be the projection.
    Let $\mathfrak{Y}$ be a formal model of $Y$. 
	We have an exact sequence $0 \to f^*p_L^*(\sF) \to f^*p_L^*(\sF') \to f^*p_L^*(\sF'') \to 0$ which induces a
	triangle $R\psi_{\mathfrak{Y}}(f^*p_L^*(\sF)) \to R\psi_{\mathfrak{Y}}(f^*p_L^*(\sF')) \to R\psi_{\mathfrak{Y}}(f^*p_L^*(\sF'')) \to \cdot$ 
	in $\mathcal{D}^b(\mathfrak{Y}_s,\Lambda)$. 
    By assumption, $R\psi_{\mathfrak{Y}}(f^*p_L^*(\sF))$ and $R\psi_{\mathfrak{Y}}(f^*p_L^*(\sF''))$ 
    are complexes whose cohomology is constructible. 
    It follows that $R\psi_{\mathfrak{X}}(f^*p_L^*(\sF'))$ is a complex whose cohomology is constructible and 
    hence $F'$ is an object of $Con(X^{f}_{\text{ét}},\Lambda)$. 
    	\end{proof}
	
 \begin{defi} \label{derived category of constructible sheaves}
 \begin{enumerate} 
  \item  \emph{Let $\mathcal{D}^{b}_{sc}(X^f_{\text{ét}},\Lambda)$ 
	 be the full subcategory of $\mathcal{D}^{b}(X^f_{\text{ét}},\Lambda)$ 
	 consisting of those complexes $A \in \mathcal{D}^{b}(X^f_{\text{ét}},\Lambda)$ 
	such that for every étale map $U \to X$ in $X^f_{\text{ét}}$ and every formal model $\mathfrak{U}$ of $U$, 
	$R\psi_{\mathfrak{U}}(A_{|U}) \in D^b_c(\mathfrak{U}_s,\Lambda)$.} 
  \item \emph{Let $\mathcal{D}^{b}_{c}(X^f_{\text{ét}},\Lambda)$ 
	 be the full subcategory of $\mathcal{D}^{b}(X^f_{\text{ét}},\Lambda)$ 
	consisting of those complexes $A \in \mathcal{D}^{b}(X^f_{\text{ét}},\Lambda)$ 
	such that for every complete algebraically closed non-Archimedean field extension $L$
 and every morphism of fine $L$-adic spaces $f : Y \to X_L$, we have
 that $f^* \circ p_L^*(A)$ is semi-constructible. Here $p_L \colon X_L \to X$ is the projection.}
    \end{enumerate} 
      \end{defi}  
  
  Unlike the class of Huber constructible sheaves, the objects of $\mathcal{D}^{b}_{c}(X^f_{\text{ét}},\Lambda)$ and 
  $\mathcal{D}^b_{sc}(X^f_{\text{ét}},\Lambda)$ are stable by pushforwards and lower shriek. 
  
  \begin{lem}  \label{semi-con stability for pushforwards}
      Let $f : X \to Y$ be a morphism of fine
      $k$-adic spaces and let $\mathscr{F} \in \mathcal{D}^b_{sc}(X^f_{\text{ét}},\Lambda)$.
      We have that $Rf_*(\mathscr{F}), Rf_!(\sF) \in \mathcal{D}^b_{sc}(Y^f_{\text{ét}},\Lambda)$.
  \end{lem} 
  \begin{proof}
  By \cite[Theorem 4.3.1]{hub96}, it suffices to show the following. 
   Let $\mathfrak{Y}$ be a formal model of $Y$. 
  We then have that $R\psi_{\mathfrak{Y}}(Rf_*(\sF)) \in \mathcal{D}^b_{c}(\mathfrak{Y}_{s},\Lambda)$.
	By Raynaud's theorem, there exists a
	 formal model $\mathfrak{X}$ of $X$ 
	and a morphism $\tilde{f}: \mathfrak{X} \rightarrow \mathfrak{Y}$ such that 
	 $f = \tilde{f}_{\eta}$. The argument in Lemma
	 \ref{vanishing cycles commutes with lower shriek} shows that
		$$R\psi_{\mathfrak{Y}}(Rf_{*}\mathscr{F}) \cong R\tilde{f}_{s*}(R\psi_{\mathfrak{X}}\mathscr{F}).$$
		As $\mathscr{F} \in D^{b}_{sc}(X^f_{\text{ét}},\Lambda)$, we 
		have that $R\psi_{\mathfrak{X}}(\mathscr{F}) \in \mathcal{D}^b_{c}(\mathfrak{X}_{s},\Lambda)$.  
		It follows that
		\begin{align*}
		 R\tilde{f}_{s*}(R\psi_{\mathfrak{X}}\mathscr{F}) \in \mathcal{D}^b_{c}(\mathfrak{Y}_{s},\Lambda). 
		 \end{align*} 
		 The proof that $Rf_!$ preserves semi-constructible complexes follows the same argument as above, making use of  
		 Theorem 5.4.6 in loc.cit. 
		 \end{proof}

   \begin{prop} \label{stability for pushforwards}
      Let $f : X \to Y$ be a morphism of fine
      $k$-adic spaces
       and let $\mathscr{F} \in\mathcal{D}^b_{c}(X^f_{\text{ét}},\Lambda)$.
      We have that $Rf_*(\mathscr{F}), Rf_!(\sF) \in \mathcal{D}^b_{c}(Y^f_{\text{ét}},\Lambda)$.
   \end{prop}
	
	\begin{proof}
Let $g : Z \to Y_L \xrightarrow{p_L} Y$ be a morphism of fine adic spaces where $Z \to Y_L$ is $L$-adic, $p_L \colon Y_L \to Y$ is the projection morphism and $L$ is a non-Archimedean algebraically closed complete field extension of $k$.
We must show that $g^*(Rf_*(\mathscr{F})) \in \mathcal{D}^b_{sc}(Z^f_{\text{ét}},\Lambda)$. 
Let $f' : X \times_Y Z \to Z$ and $g' : X \times_Y Z \to X$ denote the morphisms obtained by base change. 
By \cite[Lemma 1.1.10 (v), Theorem 4.3.1]{hub96},
we get that 
$g^*(Rf_*(\mathscr{F})) \simeq Rf'_*(g'^*(\mathscr{F}))$. 
As $g'^*(\mathscr{F})$ is constructible on 
$X \times_Y Z$, it suffices to show that 
semi-constructibility is stable for pushforwards.
This is accomplished by Lemma \ref{semi-con stability for pushforwards}. 

		 The proof that $Rf_!$ preserves constructible complexes follows the same argument as above. In place 
		 of the quasi-compact base change theorem, we use  Theorem 5.4.6 in loc.cit. 
		 	\end{proof}

Due to the fact that fine adic spaces are in general highly non-noetherian topological spaces, it means that skyscraper sheaves have no chance of being Huber constructible, as the following example shows. On the other hand they are constructible. 
		 	
 \begin{es}
  \emph{We give an example of a constructible sheaf which is not Huber constructible. Take a closed point of type (1), $i \colon x \hookrightarrow \Spa(k\langle T \rangle, k^0\langle T \rangle)$. Then $i_!\Lambda$ is constructible by Proposition \ref{stability for pushforwards}. However we claim that $x$ is not a globally constructible subset of $X:=\Spa(k\langle T \rangle, k^0\langle T \rangle)$. That is we claim that there does not exist $U , V \subset X$ open and retrocompact in $X$ such that 
  \[
  x = U \cap V^c.
  \] 
  It suffices to show that $X \backslash x$ is not quasi-compact. Suppose it is. Then the image of $X \backslash x$ under the continuous retraction $X \to X^{\Berk}$ is also quasi-compact. But $X^{\Berk} \backslash x$ cannot be compact because $x$ is not open in $X^{\Berk}$ (for the Berkovich topology).}
  \end{es}

\begin{rem}
\emph{It is clear from the definition that constructibility is stable for pullbacks. In general this is not the case for semi-constructibility, cf. \S \ref{counterexample for Verdier dual stability}.}
\end{rem}

\begin{cor} \label{constructible implies finite global sections} 
Let $X$ be a fine $k$-adic space. Let $\sF$ be a sheaf on $X^f_{\text{ét}}$ which takes values in the category of $\Lambda$-modules.
If $\mathscr{F}$ is constructible then for every morphism $g : Z \to X_L \xrightarrow{p_L} X$   where $Z \to X_L$ is $L$-adic, $p_L \colon Y_L \to Y$ is the projection morphism and $L$ is a non-Archimedean algebraically closed complete field extension of $k$,
we have that $H^i(Y,g^*(\sF))$ is finite for every $i \in \mathbb{N}$. 
\end{cor}

\subsection{Properties of constructible sheaves}

     Recall that in the theory of étale cohomology of algebraic varieties, constructibility is an étale local condition.
     
     We show that analogous statements hold true in the context of (semi)-constructible sheaves as defined above on fine adic spaces. 
     
 \subsubsection{An inductive construction}\label{indconsske}
  Let $\mathfrak{X}$ be a separated formal scheme of finite type over $k^0$. Suppose that the canonical morphism $\mathfrak{X} \to \Spf(k^0)$ goes through a morphism 
  $\mathfrak{X} \to \mathfrak{U}^1 := \Spf(k^{0} \langle T \rangle)$. Here $\mathfrak{U}^1_\eta$ is the one-dimensional unit disc. Let $t$ be 
  the Gauss-point of $\mathfrak{U}^1_\eta$. Under the continuous specialization map $\text{sp} \colon \mathfrak{U}^1_\eta \to \mathfrak{U}^1_s$ the point $t$ is the preimage of the generic point of $\mathfrak{U}^1_s$.
  Let $L$ be the completion of the algebraic closure of $k(t)$ where $k(t)$ is the residue class field of the local ring $\mathcal{O}_{\mathfrak{U}^1_\eta,t}$.  
   We set $\mathfrak{X}' := \mathfrak{X} \times_{\mathfrak{U}^1} \Spf(L^0)$. The projection $\lambda \colon \mathfrak{X}' \to \mathfrak{X}$ induces morphisms $\lambda_s \colon \mathfrak{X}'_s \to \mathfrak{X}_s$ and $\lambda_\eta \colon \mathfrak{X}'_\eta \to \mathfrak{X}_\eta$. Let us denote by $\mathscr{F}'$ the pullback $\lambda_\eta^*\mathscr{F}$ where $\mathscr{F}$ is a $\Lambda$-module on $X^f_{\text{ét}}$.
   
   \begin{lem} \label{lem:redberkad}
   In the situation of the preceding paragraph, for any $q \geq 0$, there is a canonical isomorphism
   \[
   \lambda_s^{*}(R^q\psi_{\mathfrak{X}}(\mathscr{F})) \simeq R^q\psi_{\mathfrak{X}'}(\mathscr{F}')^{P}
   \]
   where $P := \mathrm{Gal}(\mathcal{H}(t)^{sep}/\mathcal{H}(t)^{nr})$ and $\mathcal{H}(t)$ is the completion of $k(t)$.
   \end{lem}
     \begin{proof}
     This is analogous to the proof of \cite[Proposition 4.6(ii)]{berk94}. The proof crucially uses Lemma 4.4(i) in loc.cit. and the analogous statement required in our setup is the following:
     
     Let $X$ be a separated finite type $L$-adic space and $\sF$ a sheaf on $X_{\text{ét}}$ which takes values in the category of $\Lambda$-modules (in this situation we do not demand that $L$ be algebraically closed but only complete). Let $L^a$ be an algebraic closure of $L$. We set $X' := X \widehat{\otimes} \widehat{L^a}$. Let $\sF'$ be the inverse image of $\sF$ on $X'$. Then
     \[
     \varinjlim_{K/L} H^0(X \widehat{\otimes} K, \sF) \simeq H^0(X', \sF')
     \]
     where $K$ runs through finite extensions of $L$ in $L^{sep}$.
     
     This statement follows from that fact that $\Gal(L^{sep}/L)$ acts continuously on $H^i(X', \sF')$, cf.  \cite[Proposition 2.6.12]{hub96}.

   \end{proof}

The next result states that constructibility is an
 étale local condition. The key step is to show 
 a converse to Corollary \ref{constructible implies finite global sections}.
	
	 \begin{thm} \label{theorem above}
	  Let $X$ be a fine $k$-adic space and let $\sF$ be a sheaf of $\Lambda$-modules on $X^f_{\text{ét}}$.
 \begin{enumerate}
   \item   $\sF$ is constructible if and only if 
  for every complete algebraically closed non-Archimedean field extension $L$
 and every morphism of fine $L$-adic spaces $f : V \to X_L$, we have
 that
    $R\Gamma(V,f^* \circ p_L^*(\sF))$ has finite cohomology. Here $p_L \colon X_L \to X$ is the projection morphism. 
    \item   Suppose that there exists an étale cover 
   $\{U_i \to X\}$ in $X^f_{\text{ét}}$ such that for every $i$, $\sF_{|U_i}$ is constructible on $U_i$. Then $\sF$ is a constructible.   
    \end{enumerate} 
 \end{thm}   
   \begin{proof}
   \begin{enumerate}
   \item 
   We deduce from Corollary \ref{constructible implies finite global sections} that 
  it suffices to prove 
  the reverse implication: Let $\sF$ be such that
  for every complete algebraically closed non-Archimedean field extension $L$
 and every morphism of fine $L$-adic spaces $f : V \to X_L$, we have
 that $R\Gamma(V,f^* \circ p_L^*(\sF))$ has finite cohomology. Then we show that $\sF$ is constructible. Let $\mathfrak{X}$ is a formal model of $X$. It suffices to show that $R\psi_{\mathfrak{X}}(\sF)$ is constructible.   
  Since constructibility is local for the Zariski topology, we reduce to when $X$ is affinoid and $\mathfrak{X}$ is an affine formal scheme. 
  
  We proceed by induction on $d = \dim (X) < \infty$. If $d = 0$, the site $X^f_{\text{ét}}$ is equivalent to $\Spec(A)_{\text{ét}}$ where $A$ is a finite $k$-algebra, cf. \cite[\S 6.1.2, Corollary 2]{BGR}. Note that $\Spec(A)$ is the disjoint union of a finite number of points and thus we can reduce to the case that it is a single point. Since $\Spec(A)_{\text{ét}}$ is equivalent to $\Spec(A_{\text{red}})_{\text{ét}}$, we can assume that $A$ is a field. In this case $R\psi_{\mathfrak{X}}$ is just the global sections functor, which is finite by assumption.
  
  Suppose 
  that $d \geq 1$ and that the claim is true for formal schemes whose generic fibre has dimension at most $d -1$. 
  We take a closed immersion $\mathfrak{X} \to \mathfrak{U}^N := \Spf(k^{0} \langle T_1, T_2, \ldots, T_N \rangle)$ to 
  the formal affine scheme $\mathfrak{U}^N$. It gives rise to a closed immersion of the affine schemes $\mathfrak{X}_s$ to $\mathfrak{U}^N_s$ over $\tilde{k}$. Let  
  $p_i$ be the $i$th-projection from $\mathfrak{U}^N \to \Spf(k^{0}\langle T_i \rangle)$. 
  As in \ref{indconsske}, we have the morphism 
  $\lambda : \mathfrak{X}' := \mathfrak{X} \times_{\Spf(k^{0}\langle T_i \rangle)} \Spf(L^0) \to \mathfrak{X}$ where
  $L := \widehat{\overline{k(x)}}$ where $x$ is the Gauss point of $\Spa(k\langle T_i\rangle)$.  
  
  By Lemma \ref{lem:redberkad} there is a canonical isomorphism
   \[
   \lambda_s^{*}(R^q\psi_{\mathfrak{X}}(\sF)) \simeq R^q\psi_{\mathfrak{X}'}(\lambda_\eta^*\sF)^P
   \]
  Observe that $\dim (\mathfrak{X}'_\eta) = d-1$. We apply our inductive hypothesis to get 
    $\lambda_s^{*}(R^q\psi_{\mathfrak{X}}(\sF))$ is constructible.  We deduce using  \cite[Lemma 3.5]{SGA4.5} that there exists a constructible sheaf 
 $\mathscr{H}^q \subset R^q\psi_{\mathfrak{X}}(\sF)$ such that the local sections of the quotient 
 $R^q\psi_{\mathfrak{X}}(\sF)/\mathscr{H}^q$ are of finite support.
 We follow the argument in SGA 4.5 to 
    show that $R^q\psi_{\mathfrak{X}}(\sF)$ is constructible
    provided $H^q(X,\sF)$ is finite. 
 We have a spectral sequence 
 \begin{align*} 
     E_2^{p,q} = H^p(\mathfrak{X}_s, R^q\psi_{\mathfrak{X}}(\sF)) \implies H^{p+q}(X,\sF).
 \end{align*} 
     We take the image of these abelian groups in the quotient of the category of 
     abelian groups by the thick sub-category of finite abelian groups. 
     It follows that 
     \begin{align*}
      E_2^{p,q} \sim H^p(\mathfrak{X}_s, R^q\psi_{\mathfrak{X}}(\sF)/\mathscr{H}^q)
     \end{align*} 
     and $H^p(\mathfrak{X}_s, R^q\psi_{\mathfrak{X}}(\sF)/\mathscr{H}^q) \sim 0$ when $p \geq 1$.
     Hence $H^0(\mathfrak{X}_s, R^q\psi_{\mathfrak{X}}(\sF)/\mathscr{H}^q) \sim H^q(X,\sF)$. 
     This completes the proof of part (2). 
     \item  
    To prove (2), 
    we use (1) and deduce that it suffices to show that 
    $H^q(X,\sF)$ is finite for all $q$. This
     must follow from the fact that $\sF$ is locally constructible and the space is fine. More precisely, we can use Cech cohomology associated to a finite refinement of
      the cover 
   $\mathcal{U} := \{U_i \to X\}_i$ to conclude the finiteness statement we are looking for.
     Indeed, recall that we have the following spectral sequence
     \begin{align*} 
     E_2^{p,q} = \check{H}^p(\mathcal{U},\underline{H}^q(\sF)) \implies H^{p+q}(X,\sF)
     \end{align*} 
     where $\underline{H}^q(\sF)$ is the presheaf given by $U \mapsto H^q(U,\sF)$ where 
     $U \to X$ is an element of the fine site. By Lemma \ref{constructible implies finite global sections}, 
     we get that $\underline{H}^q(\sF)$ takes values in the category of finite $\Lambda$-modules. 
     Furthermore, since $X$ is fine, the cover $\mathcal{U}$ can be refined to a finite
      cover and hence we see that 
     $H^{p+q}(X,\sF)$ is finite.
     \end{enumerate} 
   
   \end{proof} 
   
   The classification of constructibility provided above parallels what happens in the theory of étale cohomology of varieties. 
   
   \begin{prop} \label{classical constructibles are finite}
      Let $X$ be an algebraic variety over an algebraically closed field $K$. 
   A sheaf $\sF$ in $\Lambda$-modules with $\ell \not= \text{char}(K)$ on $X_{\text{ét}}$ is constructible if and only if 
  for every algebraically closed extension $L$ of $K$
 and every morphism of $L$-varieties $f : V \to X_L$, we have
 that
    $H^0(V,f^* \circ p_L^*(\sF))$ is finite. Here $p_L \colon X_L \to X$ is the projection morphism. 
   \end{prop} 
\begin{proof}
We need only prove the reverse implication. 
   We may assume that $X$ is affine. 
    We proceed by induction on the dimension of $X$. 
  Let $i : X \hookrightarrow \mathbb{A}^n_K$ be a closed immersion for some $n$ and 
  let $\{p_1,\ldots p_t\}$ be projections from $\mathbb{A}^n_K$ to $\mathbb{A}^1_K$ such 
  that the image of $X$ is dense in $\mathbb{A}^1_K$. 
  Let $\eta$ denote the generic point of $\mathbb{A}^1_K$. 
  We have that $X_{\overline{\eta}}$ is of dimension $\mathrm{dim}(X) - 1$. 
  By our induction hypothesis,
  $\sF_{|X_{\overline{\eta}}}$ is constructible. 
By \cite[Lemma 3.5]{SGA4.5} it follows that there exists a constructible sub-sheaf $\mathscr{H}$ of $\sF$ such that the local sections of
  $\sF/\mathscr{H}$ are of finite support. 
  The finiteness of $H^0(X,\sF)$ implies that 
  $\sF/\mathscr{H}$ is simply the finite direct 
sum of finite skyscraper sheaves, cf. Lemma \ref{classification}. Thus we see that 
$\sF$ is constructible. 
\end{proof} 

%
%

\subsection{Semi-constructibility is an étale local property} \label{semi-constructibility is an etale local property}

 \begin{lem} \label{refining covers}
   Let $X$ and $Y$ be fine $k$-adic spaces. 
   Let $f : Y \to X$ be a surjective étale morphism. 
   There exists a finite affinoid cover $\{U_i\}$ of $X$ such that 
   for every $i$, there exists an affinoid space $V_i \subset Y$ and 
   the map $f_{|V_i} : V_i \to  U_i$ is finite.  
 \end{lem} 
 \begin{proof} 
By \cite[Corollary 1.7.4]{hub96}, $f$ is locally quasi-finite. Since it is quasi-compact, it is also quasi-finite.
  Let $x \in X$ be a maximal point.
  Since $f$ is surjective, there exists a maximal point $y \in Y$ such that 
  $f(y) = x$. 
   By Proposition 1.5.4 in loc.cit., 
  there exists affinoid neighbourhoods
   $V_x$ of $Y$ and $U_x$ of $X$ such that 
   $f(V_x) \subset U_x$ and $f_{|V_x} : V_x \to U_x$ is finite.  
   Hence, there exists a family $\{U_x\}_{x \in X^{\mathrm{an}}}$ of affinoid open subspaces of $X$ such that 
   $\{U_x^{\mathrm{an}}\}$ covers the associated Berkovich space $X^{\mathrm{an}}$. 
   Since $X^{\mathrm{an}}$ is compact, we can replace $\{U_x^{\mathrm{an}}\}$ with a finite sub-cover $\{U^{\mathrm{an}}_i\}_{1 \leq i \leq m}$. 
   It follows that that $\{U_i^{\mathrm{rig}}\}$ is an admissible cover of $X^{\mathrm{rig}}$ and hence 
   we get that 
   $\{U_i\}_{1 \leq i \leq m}$ is a cover of the adic space $X$ which satisfies the assertion of the lemma.    
 \end{proof} 
		
\begin{lem} \label{semi-constructibility descent}
 Let $f : Y \to X$ be a finite surjective étale morphism of fine $k$-adic spaces. 
 Let $\sF$ be a sheaf of $\Lambda$-modules of $X^f_{\text{ét}}$ such that 
 $\sF_{|Y}$ is semi-constructible sheaf on $Y^f_{\text{ét}}$. 
 Then $\sF$ is semi-constructible on $X^f_{\text{ét}}$. 
\end{lem} 		
\begin{proof}
It suffices to show that if 
$\mathfrak{X}$ be a formal model of $X$,
   $R\psi_{\mathfrak{X}}(\sF) \in \mathcal{D}^b_c(\mathfrak{X}_s,\Lambda)$. 
   We can replace $Y$ by the Galois closure of $Y$ over $X$ and hence assume that 
 $f$ is finite étale and Galois.
   Let $\mathfrak{Y}$ be a formal model of $Y$ such that the map 
   $f$ extends to a map $\mathfrak{f} : \mathfrak{Y} \to \mathfrak{X}$.  
  Let $G := \mathrm{Aut}(Y/X)$. 
 The sheaf $f_*f^*(\sF)$ is a $\Lambda[G]$-module on $X$ and 
 we have the following isomorphism. 
 \begin{align} 
  \sF \simeq R\Gamma^G_X(f_*f^*(\sF))
 \end{align} 
  where $\Gamma_X^G(-)$ is the functor that sends a $\Lambda[G]$-module $\sG$ to 
 the sheafification of $\Lambda$-sub module of $G$-invariants which is defined by $U \mapsto \sG(U)^G$. 

     We can extend the construction of the nearby cycles functor 
   to obtain a functor 
   $\psi_{\mathfrak{X},G} : \mathrm{Sh}(X,\Lambda[G]) \to \mathrm{Sh}(\mathfrak{X}_s,\Lambda[G])$ 
   whose derived functors we  denote 
   \begin{align*}
   R\psi_{\mathfrak{X},G} : \mathcal{D}^b_c(X,\Lambda[G]) \to \mathcal{D}^b_c(\mathfrak{X}_s,\Lambda[G]).
   \end{align*}  
  
   We claim the following isomorphism. 
  \begin{align*} 
   R\psi_{\mathfrak{X}} \circ R\Gamma^G_X(-) \simeq R\Gamma_{\mathfrak{X}_s}^G(-) \circ R\psi_{\mathfrak{X},G}. 
   \end{align*} 
     The functor $\Gamma^G_X(-)$ is right adjoint to the exact functor that sends a 
     $\Lambda$-module $M$ to the $\Lambda[G]$-module $M$ where $G$ acts trivially. Hence, it takes injectives to injectives. 
     We thus have that 
     $R\psi_{\mathfrak{X}} \circ R\Gamma^G_X(-) \simeq R (\psi_{\mathfrak{X}} \circ \Gamma^G_X(-))$. 
     Likewise, since $\psi_{\mathfrak{X},G}$ is nothing but a pushforward morphism from the 
     sheaves of $\Lambda[G]$-modules on $X^f_{\text{ét}}$ to étale sheaves on $\mathfrak{X}_s$, we see that 
     it preserves injectives and hence 
     $R\Gamma^G_{\mathfrak{X}_s}(-) \circ R\psi_{\mathfrak{X},G} \simeq R (\Gamma^G_{\mathfrak{X}_s}(-) \circ \psi_{\mathfrak{X},G})$.
     Hence to verify the claim, we need to check 
     that 
     \begin{align*} 
     \Gamma^G_{\mathfrak{X}_s}(-) \circ \psi_{\mathfrak{X},G} \simeq \psi_{\mathfrak{X}} \circ \Gamma_X^G(-)
     \end{align*} 
      which is clear from the definitions.
      
        Let $\sG := f_*f^*(\sF)$. By 
        the claim above and (2), 
       \begin{align*} 
        R\psi_{\mathfrak{X}}(\sF) \simeq R\Gamma^G_{\mathfrak{X}_s}(R\psi_{\mathfrak{X},G}(\sG)).  
       \end{align*} 
     However, $R\Gamma^G_{\mathfrak{X}_s}(R\psi_{\mathfrak{X},G}(\sG)) \in \mathcal{D}^b_c(\mathfrak{X}_s,\Lambda)$ because 
     $R\psi_{\mathfrak{X},G}(\sG) \simeq R\mathfrak{f}_{s*}R\psi_{\mathfrak{Y},G}(f^*(\sF))$ and $f^*(\sF)$ is semi-constructible on $Y$.  
\end{proof} 

\begin{prop}
  Let $X$ be a fine $k$-adic space and $\{U_i \to X\}$ be an étale cover of $X$ by fine $k$-adic spaces. 
  Let $\sF$ be a sheaf of $\Lambda$-modules on $X^f_{\text{ét}}$ such that for every $i$,
  $\sF_{|U_i}$ is semi-constructible. Then 
  $\sF$ is semi-constructible.  
\end{prop} 
\begin{proof} 
 Let $\mathfrak{X}$ be a formal model of $X$. It suffices to show that 
   $R\psi_{\mathfrak{X}}(\sF) \in \mathcal{D}^b_c(X,\Lambda)$ to conclude a proof. 
 Given $U_i \to X$ an element of the cover, we can 
 apply Lemma \ref{refining covers} to the map $\{U_i \to \mathrm{im}(U_i)\}$ and hence
  suppose that the cover $\{U_i \to X\}$ is composed of affinoids such that for 
 every $i$, there exists an open affinoid subspace $V_i \subset X$ and the morphism $U_i \to V_i$ is finite étale. 
 By Lemma \ref{semi-constructibility descent}, $\sF_{|V_i}$ is semi-constructible. 
By \cite[Lemma 4.4]{BL93}, there exists an admissible blow up $b : \mathfrak{X}' \to \mathfrak{X}$ and 
an open formal cover $\{\mathfrak{V}_i\}$ of $\mathfrak{X}'$ such that 
$\mathfrak{V}_{i,\eta} \simeq V_i$.  
We claim that $R\psi_{\mathfrak{X}'}(\sF) \in \mathcal{D}^b_c(\mathfrak{X}'_s,\Lambda)$. 
Indeed for every $i$,
\begin{align*} 
   R\psi_{\mathfrak{X}'}(\sF)_{|\mathfrak{V_i}_s} \simeq R\psi_{\mathfrak{V}_i}(\sF_{|V_i}).
\end{align*} 
 The claim follows since $\sF_{|V_i}$ is semi-constructible and 
 by choice of $\mathfrak{X}'$, $\{\mathfrak{V}_i\}_i$ is a cover. 
 We can conclude the proof since $Rb_{s*} \circ R\psi_{\mathfrak{X}'} \simeq R\psi_{\mathfrak{X}}$. 
\end{proof}

    \section{Adic Verdier dual} 

        As stated in the introduction, Proposition \ref{compved} shows that the adic Verdier dual is compatible with the nearby cycles functor 
        and the classical Verdier dual on varieties.  

      Recall from Theorem \ref{upper shriek} that
    when given a morphism $\phi : X \to Y$ of fine $k$-adic spaces, Huber defined 
    a functor     
    \begin{align*} 
     \phi^! : \mathcal{D}^+(Y^f_{\text{ét}},\Lambda) \to D^+(X^f_{\text{ét}},\Lambda). 
   \end{align*} 
    
     \begin{defi}   
 \emph{Let $X$ be a fine $k$-adic space and $p_X : X \to \mathrm{Spa}(k)$ denote the structure map.  
 Let $\sF \in \mathcal{D}^b(X^f_{\text{ét}},\Lambda)$. We use $D^{\mathrm{ad}}(\sF)$ to denote the} 
 adic Verdier dual of $\sF$ \emph{and define it to be the complex $R\Homs(\sF,p_X^!(\Lambda)) \in \mathcal{D}^b(X^f_{\text{ét}},\Lambda)$}. 
    \end{defi}

    \begin{prop} \label{compved}
Let $\mathscr{F} \in \mathcal{D}^b(X^f_{\text{ét}}, \Lambda)$ and 
let $f : U \to X$ be an element of the fine étale site of $X$. 
Let $\mathfrak{U}$ be any formal model of $U$. 
Then we have that  
\begin{align*} 
  R\Gamma(U,D^{\mathrm{ad}}(\sF)) = R\Gamma(\mathfrak{U}_s, D(R\psi_{\mathfrak{U}}(f^*\mathscr{F}))) 
  \end{align*} 
where $p_U : U \to k$ is the structure morphism. In particular $R\Gamma(\mathfrak{U}_s, D(R\psi_{\mathfrak{U}}(f^*\mathscr{F})))$ is independent of the choice of formal model $\mathfrak{U}$.
\end{prop}

\begin{proof}
This is a direct calculation. 
 Let $\mathfrak{h} : \mathfrak{U} \to \mathrm{Spf}(k^0)$ and $h_s := \mathfrak{h}_s : \mathfrak{U}_s \to \mathrm{Spec}(\tilde{k})$. 
 We use $h = \mathfrak{h}_\eta$ to denote the generic fibre of $\mathfrak{h}$, so that 
 $h = p_X \circ f$.  
 Evaluating $D^{\mathrm{ad}}(\mathscr{F})$ for an étale morphism $f: U \to X$, we obtain
\begin{align*}
D^{\mathrm{ad}}(\mathscr{F})(U) 
& \overset{(1)}{=}  R\Homs(\mathscr{F}, p_X^{!}(\Lambda))(U) \\
& \overset{(2)}{=} R\Gamma(U,f^{*}R\Homs(\mathscr{F},p_X^{!}\Lambda)) \\
& \overset{(3)}{=} R\Gamma(U,f^{!}R\Homs(\mathscr{F},p_X^{!}\Lambda)) \\
& \overset{(4)}{=} R\Gamma(U,R\Homs(f^*\mathscr{F},h^{!}\Lambda)) \\
& \overset{(5)}{=} R\Homs(Rh_{!}f^*\mathscr{F},\Lambda)\\
& \overset{(6)}{=} R\Homs(R\psi_{\mathrm{Spf}(k^0)}(Rh_{!}f^*\mathscr{F}),\Lambda) \\
& \overset{(7)}{=} R\Homs(Rh_{s,!}R\psi_{\mathfrak{U}}(f^*\mathscr{F}),\Lambda) \\
&\overset{(8)}{=} R\Gamma(\mathfrak{U}_s,R\Homs(R\psi_{\mathfrak{U}}(f^*\mathscr{F}),K_{\mathfrak{U}_s})) \\
&\overset{(9)}{=}  R\Gamma(\mathfrak{U}_s, D(R\psi_{\mathfrak{U}}(f^*\mathscr{F}))) \\
\end{align*}
where (1) follows from the definition of the adic Verdier dual, 
(3) follows from the fact that $f$ is étale,
(4) follows from Lemma \ref{projection formula adic spaces}, 
(5) follows from the adjointness of $Rh_{!}$ and $h^{!}$, 
(6) follows from triviality of nearby cycles on $\Spa(k)$,
(7) follows from Lemma \ref{vanishing cycles commutes with lower shriek},
(8) follows from the adjointness of $Rh_{s,!}$ and $h_s^!$
(here $K_{\mathfrak{U}_s} = h_s^{!}\Lambda$ is the dualizing sheaf of $h_s: \mathfrak{U}_s \to \widetilde{k}$) and 
(9) follows from the definition of the classical Verdier dual
\end{proof}

\begin{rem} 
 \emph{The calculation in Proposition \ref{compved} can be used to construct the adic Verdier dual $D^{\mathrm{ad}}$ 
   using only the nearby cycles functor $R\psi_{\mathfrak{U}}$ and the classical Verdier dual.}  
 \end{rem}

 \begin{thm} \label{big result small proof}
    Let $X$ be a fine $k$-adic space.   
      Let $\mathfrak{X}$ be a formal model of $X$. 
      The adic Verdier dual satisfies the following properties. 
     \begin{enumerate}  
          \item $R\psi_{\mathfrak{X}} \circ D^{\mathrm{ad}} = D \circ R\psi_{\mathfrak{X}}$ where $D$ is the (classical) Verdier dual
      in $\mathcal{D}^b(\mathfrak{X}_s,\Lambda)$. 
          \item $R\psi_\mathfrak{X} \circ D^{\mathrm{ad}} \circ D^{\mathrm{ad}} = D \circ D \circ R\psi_\mathfrak{X}$. In particular, 
          if $\sF \in \mathcal{D}^b_{sc}(X^f_{\text{ét}},\Lambda)$ then 
          \begin{align*}
          R\psi_\mathfrak{X} \circ D^{\mathrm{ad}} \circ D^{\mathrm{ad}} (\sF) = R\psi_\mathfrak{X}(\sF).
          \end{align*}
          \item If $\mathscr{F} \in \mathcal{D}^b_{sc}(X^f_{\text{ét}},\Lambda)$ then $D^{\mathrm{ad}} \circ D^{\mathrm{ad}}(\sF)  = \sF$. 
      \end{enumerate}  
  \end{thm} 
  \begin{proof} 
  \begin{enumerate} 
  \item 
   Let $\mathfrak{X}_s$ denote the special fibre of the formal scheme $\mathfrak{X}$. 
  Observe that it suffices to show that for every étale morphism
  $p_s : Y_s \to \mathfrak{X}_s$ and every $\Lambda$-sheaf $\mathscr{F}$ on $X_{\text{ét}}$, we have a natural isomorphism  
  \begin{align*} 
    R\Gamma(Y_s, R\psi_\mathfrak{X}(D^{\mathrm{ad}}(\mathscr{F}) )) \simeq R\Gamma(Y_s,D(R\psi_\mathfrak{X}(\mathscr{F}) )). 
  \end{align*}  
  
       We have a natural isomorphism of sites $\mathfrak{X}_{\text{s,ét}} \simeq \mathfrak{X}_{\text{ét}}$. 
       It follows that there exists an étale morphism of formal schemes 
       $\mathfrak{p} : \mathfrak{Y} \to \mathfrak{X}$ such that the induced morphism 
       between the respective special fibres coincides with $p_s$. Let 
      $p : Y \to X$ be the morphism $\mathfrak{p}_\eta$ between the respective 
      generic fibres. Since $\mathfrak{p}$ is étale, we get that $p$ is étale. 
  
  Observe that 
  \begin{align*} 
   R\Gamma(Y_s, R\psi_\mathfrak{X}(D^{\mathrm{ad}}(\mathscr{F}) )) 
   &\overset{(1)}{\simeq} R\Gamma(Y_s, p_s^*R\psi_\mathfrak{X}(D^{\mathrm{ad}}(\mathscr{F}) )) \\
   &\overset{(2)}{\simeq}  R\Gamma(Y_s, R\psi_\mathfrak{Y}(p^*D^{\mathrm{ad}}(\mathscr{F}) )) \\ 
   &\overset{(3)}{\simeq}  R\Gamma(Y,p^*D^{\mathrm{ad}}(\mathscr{F})) \\
   &\overset{(4)}{\simeq} R\Gamma(Y, D^{\mathrm{ad}}(\mathscr{F}) ) \\
   &\overset{(5)}{\simeq} R\Gamma(Y_s, D(R\psi_{\mathfrak{Y}}(p^*\mathscr{F}))) \\
   & \overset{(6)}{\simeq} R\Gamma(Y_s, p_s^*D(R\psi_{\mathfrak{X}}(\mathscr{F}) )) \\
  &\overset{(7)}{\simeq} R\Gamma(Y_s, D(R\psi_{\mathfrak{X}}(\mathscr{F}) )). 
   \end{align*} 
   
   The quasi-isomorphism (2) follows from the fact that $p_s$ is étale, (3) is because of the composition of derived functors, 
   (5) is implied by Proposition \ref{compved}, (6) is because $p_s$ is étale. 
   

  \item  Part (2) follows by manipulating the equalities from part (1). 
  By (1), we get that $R\psi_\mathfrak{X} \circ D^{\mathrm{ad}} \circ D^{\mathrm{ad}} = 
  D \circ R\psi_\mathfrak{X} \circ D^{\mathrm{ad}}$. 
  Using (1) again, implies that 
  $D \circ R\psi_\mathfrak{X} \circ D^{\mathrm{ad}} = D \circ D \circ R\psi_{\mathfrak{X}}$. 
  
\item  We verify the equality by evaluating the given expressions at an object
$f : U \to X$ of the étale site of $X$. 
 We claim that  
 \begin{align*} 
 R\Gamma(U,D^{\mathrm{ad}} \circ D^{\mathrm{ad}}(\sF)) \simeq R\Gamma(U,\sF). 
 \end{align*} 
    Indeed, by part (2) and the fact that the classical Verdier dual is an involution on $\mathcal{D}^b_c(\mathfrak{U}_s,\Lambda)$, we get that if $\mathfrak{U}$ is a formal model of $\mathfrak{U}$ then 
   \begin{align*} 
 R\psi_{\mathfrak{U}} \circ D^{\mathrm{ad}} \circ D^{\mathrm{ad}}(\sF_{|U}) \simeq R\psi_\mathfrak{U}(\sF_{|U}). 
 \end{align*} 
  Hence we get that 
    \begin{align*} 
R\Gamma(\mathfrak{U}_s,R\psi_{\mathfrak{U}} \circ D^{\mathrm{ad}} \circ D^{\mathrm{ad}}(\sF_{|U})) \simeq 
R\Gamma(\mathfrak{U}_s,R\psi_\mathfrak{U}(\sF_{|U})). 
 \end{align*} 
  Since for a complex $\mathscr{G}$ on $U_{\text{ét}}$, we have that 
  $R\Gamma(U,\mathscr{G}) \simeq R\Gamma(\mathfrak{U}_s,R\psi_\mathfrak{U}(\mathscr{G}))$, cf. \cite[Corollary 4.5]{berk94}. 
  This implies the claim. 
 \end{enumerate}
  \end{proof}  

\begin{rem}
  \emph{One of the important features of the theory of étale cohomology for varieties is the 
  fact that when restricted to the class of constructible étale sheaves, the 
  Verdier dual interacts nicely with the functors $Rf_!$, $f^*$, $f_*$ and $f^!$.  
  We deduce as simple consequences of Theorem \ref{big result small proof},
   similar results for the adic Verdier dual and the class of
    semi-constructible complexes of sheaves on a fine $k$-adic space.}
\end{rem} 
  
  \begin{lem} \label{projection formula adic spaces} 
  Let $f : X \to Y$ be a morphism of fine $k$-adic spaces.
  Let $A,B \in D^b(Y^f_{\text{ét}},\Lambda)$ and $\sF \in \mathcal{D}^b(X^f_{\text{ét}},\Lambda)$. 
  \begin{enumerate} 
  \item  $f^!R\Homs(A,B) = R\Homs(f^*(A),f^!(B))$. 
  \item  $f^!D^{\mathrm{ad}}(A) = D^{\mathrm{ad}}(f^*(A))$.
  \item  $D^{\mathrm{ad}}(Rf_{!}(\mathscr{F})) = Rf_{*}(D^{\mathrm{ad}}(\mathscr{F}))$
  \end{enumerate} 
  \end{lem} 
  \begin{proof} 
   Let $C \in D^b(X^f_{\text{ét}},\Lambda)$. 
   By \cite[Theorem 5.5.9(ii)]{hub96}, we have the following projection formula. 
   \begin{align*} 
    Rf_!(C) \otimes^L A \simeq Rf_!(C \otimes^L f^*(A)). 
    \end{align*} 
    Using the formula above, the proof is the same as the proof of Lemma \ref{projection formula varieties}. 
    Part (2) is a direct consequence of (1). 
    Part (3) follows from the adjointness of $Rf_!$ and $Rf_*$. 
  \end{proof}


\begin{cor} \label{Corollary 2}
  Let $X$ be a fine $k$-adic space and let $\sF,\sG \in \mathcal{D}^b_{sc}(X^f_{\text{ét}},\Lambda)$.
  Let $f : X \to Y$ be a morphism of fine $k$-adic spaces. 
  Then we have that 
  \begin{enumerate} 
    \item $D^{\mathrm{ad}}(\sF) \in \mathcal{D}^b_{sc}(X^f_{\text{ét}},\Lambda)$. 
  \item  $D^{\mathrm{ad}}(Rf_{*}(\mathscr{F})) = Rf_{!}(D^{\mathrm{ad}}(\mathscr{F}))$.
\item $f^!(\sF) = D^{\mathrm{ad}} \circ f^* \circ D^{\mathrm{ad}} (\sF)$.  
\item $R\mathscr{H}om(\sF,\sG) = D^{\mathrm{ad}}(\sF \otimes^L D^{\mathrm{ad}}(\sG))$. 
  \end{enumerate}  
\end{cor} 
\begin{proof} 
 \begin{enumerate} 
 \item This follows from part (1) of Theorem \ref{big result small proof}. 
 \item Substituting $\sF$ with $D^{\mathrm{ad}}(\sF)$ in Lemma \ref{projection formula adic spaces}, we get 
 \begin{align*} 
 D^{\mathrm{ad}}Rf_{!}(D^{\mathrm{ad}}(\mathscr{F})) &= Rf_{*}D^{\mathrm{ad}}D^{\mathrm{ad}}(\mathscr{F})) \\
                                                                                       &= Rf_*(\sF)  
 \end{align*} 
        where the second equation follows from part (3) in Theorem \ref{big result small proof} .
        Applying $D^{\mathrm{ad}}$ to both sides of the equation above and noting that
        $Rf_{!}(D^{\mathrm{ad}}(\mathscr{F})) \in \mathcal{D}^b_{sc}(X^f_{\text{ét}},\Lambda)$ by part (1) and Lemma \ref{semi-con stability for pushforwards} completes the proof of the identity. 
 \item Let $\sG := D^{\mathrm{ad}}(\sF)$.
     Part (1) implies that $\sG$ is semi-constructible. 
  By Lemma \ref{projection formula adic spaces}, we have that 
 \begin{align*} 
     f^! \circ D^{\mathrm{ad}}(\sG) &= D^{\mathrm{ad}} \circ f^* (\sG).
   \end{align*} 
   The result follows from the fact that semi-constructible sheaves are reflexive.
  \item This follows by setting $\sG = D^{\mathrm{ad}} \circ D^{\mathrm{ad}}(\sG)$ and the adjointness of $R\mathscr{H}om$ and $\otimes^L$. 
 \end{enumerate} 
\end{proof}

\subsection{Huber constructible is not stable for upper shriek} \label{huber constructible is not stable for upper shriek}
 
        We provide an example of a morphism $p_X : X \to \mathrm{Spa}(k)$ of fine adic spaces such that 
        $p_X^!(\Lambda)$ is a complex whose cohomology is not Huber constructible. In the next section 
        however (cf. \S \ref{stabbyupshrei}) we show that the dualizing complex is constructible. 
        Using Lemma \ref{suitable singular curve}, we
        choose a suitable singular curve $C$ such that 
        $H^{-2}(p_{C^{\mathrm{ad}}}^!(\Lambda))$ is trivial with respect to a stratification of the form $\{U^{\mathrm{ad}},\{x\}\}$ where 
        $U$ is a Zariski open subset of $C$ and $x$ is a $k$-point.  
        Since a closed point cannot be a Huber constructible set, we obtain the counterexample we are looking for. 
        The idea for the above construction was suggested to us by R. Huber. 
        
      \begin{lem} \label{suitable singular curve}
  Let $C$ be a one-dimensional proper scheme over an algebraically closed field $k$ and assume that $\ell$ is prime to $\textrm{char}(k)$. Setting $\Lambda = \mathbb{Z}/ \ell\mathbb{Z}$ and $p_C \colon C \to k$, the structure morphism, the dualizing complex sits in an distinguished triangle
  \[
  i_*\Lambda \to p_C^!\Lambda \to Rj_*\Lambda[2] \to \cdot
  \]
  where $j \colon U \to C$ is the open complement of the singular set $i \colon Z \to C$.
  \end{lem}  
  \begin{proof}
This follows from the proof in \cite[Theorem 3.1]{den}. 
  \end{proof}
  
  \begin{lem} \label{dual of the restriction sequence}
   Let $X$ be a fine $k$-adic space. 
    Let $\sF \in \mathcal{D}^b_{sc}(X^f_{\text{ét}},\Lambda)$. 
    Let $i : Z \hookrightarrow X$ be a closed adic subspace.  
  Let $j : U :=  X \smallsetminus Z \hookrightarrow X$ be the inclusion of the complement. 
    We have the following distinguished triangle. 
\begin{align*} 
   i_*i^!(\sF) \to \sF \to Rj_*j^*(\sF) \to \cdot
\end{align*}  
  \end{lem} 
  \begin{proof} 
   We have the following triangle
   \begin{align*} 
     Rj_!j^*(D^{\mathrm{ad}}(\sF)) \to D^{\mathrm{ad}}(\sF) \to i_*i^*(D(\sF)) \to  \cdot
      \end{align*} 
    Note that $D^{\mathrm{ad}}(\sF)$ is semi-constructible
    and hence reflexive.
    We apply $D^{\mathrm{ad}}$ to the triangle above to obtain  
       \begin{align*} 
    D^{\mathrm{ad}} i_*i^*(D^{\mathrm{ad}}(\sF)) \to \sF \to  D^{\mathrm{ad}}Rj_!j^*(D^{\mathrm{ad}}(\sF))  \to \cdot 
      \end{align*} 
      By Corollary \ref{Corollary 2}, we get 
      \begin{align*} 
       D^{\mathrm{ad}} i_*i^*(D^{\mathrm{ad}}(\sF)) \simeq i_*i^!(\sF). 
      \end{align*} 
      
     For the next part note that $j \colon U \to X$ is partially proper. We have that 
    \begin{align*} 
     D^{\mathrm{ad}}Rj_!j^*(D^{\mathrm{ad}}(\sF)) &\overset{(i)}{\simeq}  Rj_*D^{\mathrm{ad}}j^*D^{\mathrm{ad}}(\sF) \\
      &\overset{(ii)}{\simeq}  Rj_*j^*D^{\mathrm{ad}}D^{\mathrm{ad}}(\sF) \\ 
       &\overset{(iii)}{\simeq}Rj_*j^*\sF. 
      \end{align*} 
        Here we see that the quasi-isomorphism (i) follows from the adjointness of $j_!$ and $j^*$. The quasi-isomorphism (ii)
       can be deduced from the projection formula \cite[Lemma 5.5.9(ii)]{hub96} using the same steps as in the proof of   
      Lemma \ref{projection formula varieties}. Lastly, the quasi-isomorphism (iii) 
      is obtained from the reflexivity of the sheaf $\sF$. 
      
             
  \end{proof}

  \begin{prop} \label{comparison result} 
  Let $V$ be an algebraic variety over $k$ and $V^{\mathrm{ad}}$ denote its 
  adification. Let $\mu_V$
 denote the morphism of sites
  $V^{\mathrm{ad}}_{\text{ét}} \to V_{\text{ét}}$.
  The following isomorphism holds true.
  \begin{align*} 
    \mu_V^*p_V^!(\Lambda) \simeq p_{V^{\mathrm{ad}}}^!(\Lambda).
  \end{align*} 
  \end{prop}
  \begin{proof}
  We proceed by induction on the dimension $\mathrm{dim}(V)$ of the variety $V$. 
      Let $j : U \to V$ be a Zariski dense open smooth sub-variety of $V$ and 
      let $i : Z := V \smallsetminus U \to V$ be the inclusion of the complement of $U$ 
      into $V$. 
      We add the super script $\mathrm{ad}$ to $i$ and $j$ to denote the respective adifications. 
     It is well known from the classical étale cohomology theory for varieties that we have 
        \begin{align*} 
      i_*i^{!}p_V^!(\Lambda) \to p_V^!(\Lambda) \to Rj_*j^{*}p_V^!(\Lambda) \to \cdot
          \end{align*} 
          Since $U$ is smooth, $j^{*}p_V^!(\Lambda) \simeq \oplus_t \Lambda[2d_t]$ where $d_t$ runs over the dimension of the irreducible components of $U$.
        By \cite[Theorem 3.8.1]{hub96},
        $\mu_V^*i_*p_Z^!(\Lambda) = i^{\mathrm{ad}}_*\mu_Z^*p_Z^!(\Lambda)$ and 
          $\mu_V^*Rj_*j^*(\oplus_t \Lambda[2d_t]) = Rj^{\mathrm{ad}}_*(\oplus_t \Lambda[2d_t])$.
        Our induction hypothesis implies that 
        $i^{\mathrm{ad}}_*\mu_Z^*i^!(\Lambda) = i_*^{\mathrm{ad}}p_{Z^{\mathrm{ad}}}^!(\Lambda)$. 
        Hence applying $\mu_V^*$ to the triangle above gives
        \begin{align*} 
          i_*^{\mathrm{ad}}p_{Z^{\mathrm{ad}}}^!(\Lambda) \to \mu_V^*p_V^!(\Lambda) \to Rj^{\mathrm{ad}}_* (\oplus_t \Lambda [2d_t]) \to \cdot 
        \end{align*} 

      By Lemma \ref{dual of the restriction sequence} and the fact that $p_{U^{\mathrm{ad}}}^!(\Lambda) = \oplus_t \Lambda[2d_t]$, we get the following triangle
      \begin{align*} 
      i^{\mathrm{ad}}_*p_{Z^{\mathrm{ad}}}^!(\Lambda) \to p_{V^{\mathrm{ad}}}^!(\Lambda) \to Rj^{\mathrm{ad}}_*( \oplus_t \Lambda[2d_t]) \to \cdot
          \end{align*} 
         Since the maps above are natural, 
         we get that 
         $p_{V^{\mathrm{ad}}}^!(\Lambda) \simeq  \mu_V^*p_V^!(\Lambda)$.
          
   \end{proof} 
          
 \begin{prop} 
 There exists a 
   singular curve $C$ over $k$ such that $H^{-2}(p_{C^{\mathrm{ad}}}^!(\Lambda))$ is not Huber constructible. 
 \end{prop} 
 \begin{proof} 
   Let $C$ be the singular curve in $\mathbb{P}^2_k$ given by the homogenous 
   equation $\langle Y^2 = X^2(X-Z) \rangle$.
    Observe that $C$ has a singularity at the point $x := (0:0:1)$.    
    Let $f : \tilde{C} \to C$ denote the normalization of $C$. 
    Since the morphism $f$ is finite, $f_*$ is exact. 
    A quick calculation shows that the preimage of $x$ in $\tilde{C}$ has two points. 
    Let $U := C \smallsetminus \{x\}$, $j : U \hookrightarrow C$ 
    and $j' : U \hookrightarrow \tilde{C}$ 
   be the associated open immersions.
   We have that 
   $Rj_* = f_* Rj'_*$. 
   Let $y$ be a preimage of $x$ in $\tilde{C}$. 
   By \cite[Exposé VIII, Théor\`{e}me 5.2]{sga4tome2}, we have that $(j'_*(\Lambda))_{\bar{y}} = H^0(K(y),\Lambda)$ where 
   $K(y)$ is the quotient field of the strict henselization of $\mathcal{O}_{\tilde{C},y}$. 
   Hence we get that $[j_*(\Lambda)]_{\bar{x}} = \Lambda \oplus \Lambda$. 
   This must mean in particular that $j_*(\Lambda)$ cannot be locally constant and 
   is only trivial along the stratification of $C$ given by $\{U,x\}$.      

  By Lemma \ref{suitable singular curve} and 
  Lemma \ref{comparison result}, we see that $H^{-2}(p_{C^{\mathrm{ad}}}^!(\Lambda))$ cannot be locally constant and is only trivial along the 
   stratification of $C^{\mathrm{ad}}$ given by $\{U^{\mathrm{ad}},x\}$. 
   However, $\{x\}$ is not a constructible subset of $C^{\mathrm{ad}}$. Hence 
   $H^{-2}(p_{C^{\mathrm{ad}}}^!(\Lambda))$ is not Huber constructible.  
  \end{proof} 
  
   \subsection{Stability by upper shriek} \label{stabbyupshrei}
  
In general, upper shriek does not preserve constructibility nor semi-constructibility, cf. Example \ref{es:nonstabupshr}. However in this section we prove that for a morphism of fine $k$-adic spaces $f \colon X \to Y$, $f^!$ of a Huber constructible sheaf is constructible. In particular the dualizing complex (which appears in Verdier duality) is constructible. The main ingredients are Gabber's weak uniformization theorem \cite[Exp. VII, Theorem 1.1]{TraGabbunilo} and Deligne's cohomological descent theory (cf. \cite[Exposé Vbis]{sga4tome2}  and \cite{Del74THdeHIII}) adapted to the setting of adic spaces.
  
  First we need a factorization lemma.
  
  \begin{lem} \label{lem:factmorfinek}
  Let $f \colon X \to Y$ be a morphism of fine $k$-adic spaces. Then locally (for the analytic topology) on $X$, $f$ factors into $X \xrightarrow{i} Z \xrightarrow{g} Y$ where $Z$ is a fine $k$-adic space, $i$ is a closed embedding and $g$ is a smooth morphism of pure dimension $d$. 
  \end{lem}
  
  \begin{proof}
    Let $x \in X$ and $y := f(x)$. By definition of a morphism of finite type, there exists 
    open affinoid neighbourhoods $U := \mathrm{Spa}(B,B^+) \subset X$ and $V := \mathrm{Spa}(A,A^+) \subset Y$ such that 
    $f$ restricts to a morphism $U \to V$ which is topologically of finite type. 
    By definition, $f_{|U}$ is induced by a morphism $A \to B$ that factors 
    through a surjective open continuous map $g : A\langle X_1,\ldots,X_n\rangle_M \to B$  where $M$ is as in Definition \ref{finite type ring morphism}.  
    Hence, we can identify $U$ with a closed adic sub-space of $\mathrm{Spa}(\langle A\langle X_1,\ldots,X_n\rangle_M,C)$ where $C$ is
    as in Definition \ref{finite type ring morphism}.
    By \cite[Corollary 1.6.10]{hub96}, 
    the morphism $$\mathrm{Spa}(\langle A\langle X_1,\ldots,X_n\rangle_M,C) \to V$$ is smooth.     
   
  \end{proof}

\begin{lem} \label{uppershbaschgrfi}
Let $K/k$ be an extension of algebraically closed non-Archimedean fields. Let $X$ be a fine $k$-adic space and $X' := X \times_k K$. Then from the cartesian diagram
  $$
\begin{tikzcd}[row sep = large, column sep = large]
X' \arrow[d, "g'"] \arrow[r, "p_{X'}"] &
K \arrow[d, "g"] \\
X \arrow[r, "p_X"]
&
k  
\end{tikzcd}
$$
the canonical morphism (coming from adjunction)
\[
g'^*p_X^!(\Lambda) \to p_{X'}^!(\Lambda)
\]
is an isomorphism.
\end{lem}

\begin{proof}
By base change for compact support (cf. \cite[Theorem 5.5.9(i)]{hub96}), we have a canonical morphism (given by the counit map)
\[
p_{X'!}g'^{*}p_X^!(\Lambda) \simeq g^*p_{X!}p_X^!(\Lambda) \to g^*(\Lambda)
\]
and by adjunction this gives a morphism $g'^*p_X^!(\Lambda) \to p_{X'}^!(\Lambda)$. Proving it is an isomorphism is local for the analytic topology on $X'$, so by Lemma \ref{lem:factmorfinek}, we can assume that $p_X$ factorises as $X \xrightarrow{i} Y \xrightarrow{p_Y} k$ where $Y$ is a fine $k$-adic space, $i$ is a closed embedding and $p_Y$ is a smooth morphism of pure dimension. Let $j \colon Z \to Y$ be the complement of $i$, so that we have a distinguised triangle
\[
i_*i^!(\Lambda) \to \Lambda \to Rj_*(\Lambda) \to \cdot
\]
It suffices to prove that $i^!(\Lambda)$ is compatible with change of base field, or equivalently $Rj_*(\Lambda)$ is compatible with change of base field. This follows from the equivalent statement for Berkovich spaces (cf. the proof of \cite[Theorem 7.6.1]{berk}) and \cite[Theorem 8.3.5]{hub96}.
\end{proof}


\subsubsection{Cohomological descent} \label{cohomological descent}
We need a base change result for upper shriek.

\begin{lem} \label{lem:uppshreibaschange1}
Let
  $$
\begin{tikzcd}[row sep = large, column sep = large]
X' \arrow[d, "g'"] \arrow[r, "f'"] &
Y' \arrow[d, "g"] \\
X \arrow[r, "f"]
&
Y  
\end{tikzcd}
$$
be a cartesian diagram of fine $k$-adic spaces. Then for $\mathscr{F} \in D^+(Y'_{\text{ét}}, \Lambda)$ there is an isomorphism of functors
\[
f^! \circ Rg_*(\mathscr{F}) \simeq Rg'_* \circ f'^!(\mathscr{F}).
\]
\end{lem}

\begin{proof}
This follows from \cite[Theorem 5.4.6]{hub96} and adjunction.
\end{proof}

Since we will be dealing with hypercoverings, what we will really need is a simplicial version of Lemma \ref{lem:uppshreibaschange1}. So let us recall and apply some definitions and facts from cohomological descent theory to fine $k$-adic spaces. Something similar has been done by Berkovich in \cite[\S 1.2]{berk13} and we refer the reader there for more details. A simplicial object of $Fine_{k-ad}$ is a contravariant functor $\Delta \to Fine_{k-ad}$, where $\Delta$ is the category whose objects are the sets $[n] = \{0,1, \ldots, n\}$, $n \geq 0$ and morphisms are nondecreasing maps. Such an object is denoted by $X_\bullet = (X_n)_{n \geq 0}$, where $X_n$ is the image of $[n ]$ and $X_\bullet(f)$ denotes the morphism $X_m \to X_n$ that corresponds to a morphism $f \colon [n] \to [m]$. One can then make sense of the étale site $X_{\bullet \text{ét}} $ and the étale topos $X_{\bullet \text{ét}}^\sim$.

Let $S$ be a fine $k$-adic space. This defines a constant simplicial object $S_\bullet$ which corresponds to the functor on $\Delta$ that takes the constant value $S$. By an augmentation of a simplicial object $X_\bullet$ to $S$, we mean a morphism
\[
a = (a_n)_{n \geq 0} \colon X_\bullet \to S_\bullet
\] 
which is briefly denoted by $a \colon X_\bullet \to S$. If $\mathscr{F}$ is an étale sheaf on $S$, then
\[
a^*(\mathscr{F}) = (a_n^*\mathscr{F})_{n \geq 0}
\]
is an étale sheaf on $X_\bullet$. The functor $\mathscr{F} \mapsto a^*\mathscr{F}$ has a right adjoint $a_*$ defined by
\[
a_*(\mathscr{F}^\bullet) := \ker(a_{0*}(\mathscr{F}^0) \rightrightarrows a_{1*}(\mathscr{F}^1))
\]
where the two arrows are induced by the two morphisms $[0] \rightrightarrows [1]$. Finally for a morphism of simplicial fine $k$-adic spaces
\[
\varphi = (\varphi_n)_{n \geq 0} \colon Y_\bullet \to X_\bullet
\]
we can define $\varphi^!$ and $R\varphi_!$ term by term.

\begin{lem} \label{lem:uppshreibaschange}
Let $X$ and $Y$ be fine $k$-adic spaces, and $X'_\bullet$ and $Y'_\bullet$ simplicial fine $k$-adic spaces. Suppose
  $$
\begin{tikzcd}[row sep = large, column sep = large]
X'_\bullet \arrow[d, "g'"] \arrow[r, "f'"] &
Y'_\bullet \arrow[d, "g"] \\
X \arrow[r, "f"]
&
Y  
\end{tikzcd}
$$
is a cartesian diagram, where the vertical arrows are augmented simplicial $k$-adic spaces. Then for $\mathscr{F} \in D^+(Y'_{\bullet \text{ét}}, \Lambda)$ there is an isomorphism of functors
\[
f^! \circ Rg_*(\mathscr{F}) \simeq Rg'_* \circ f'^!(\mathscr{F}).
\]
\end{lem}

\begin{proof}
By adjunction it suffices to prove a simplicial version of \cite[Theorem 5.4.6]{hub96}: That is for $\mathscr{G} \in D^+(X_{ \text{ét}}, \Lambda)$ an isomorphism $g^* \circ Rf_!(\mathscr{G}) \simeq f'_! \circ g'^*(\mathscr{G})$. This can be checked in each separate degree, where one has an honest cartesian diagram of fine $k$-adic spaces.
\end{proof}

We now turn our attention to specific hypercoverings. We need the following analogue of \cite[Theorem 1.3.1]{berk13}.

\begin{lem} \label{lem:hypcovtru}
Let $X$ be a fine $k$-adic space. Then there exists a 
surjective morphism $\coprod_{i \in I} Y_i \to X$ with 
each $Y_i$ of the form $\mathcal{Y}_{\eta}^\wedge$, where $\mathcal{Y}$ is 
an affine scheme finitely presented over $k^0$.
\end{lem}

\begin{proof}
This is a direct consequence of taking the adification of \cite[Theorem 1.3.1]{berk13}. The only thing to check is that the morphism in loc.cit. remains surjective after adification. It suffices to check this for $X = \Spa (A)$. The morphism in loc.cit. is then coming from the analytification of a covering in the alteration topology of $\Spec (A)$. Thus it is either a proper surjective morphism (cf. \cite[Proposition 2.6.9]{berk}) or an étale morphism (cf. \cite[Proposition 3.3.11]{berk}). In the former situation the adification of a proper surjective morphism is again surjective (cf. the proof of \cite[Proposition 8.3.4]{hub96}), and in the latter situation the adification is again surjective by Proposition 8.3.4 in loc.cit.
\end{proof}

\begin{cor} \label{cor:spehypsmooad}
Let $X$ be a fine $k$-adic space. Then $X$ admits a hypercovering in the étale topology of universal $\Lambda$-cohomological descent of the form $a \colon Y_{\bullet} \to X$, in which all the $Y_n$ are disjoint unions of affinoids in the analytifications of smooth affine schemes over $k$.
\end{cor}

\begin{proof}
This follows by taking the the coskeleton of the morphism in Lemma \ref{lem:hypcovtru}, $\text{cosk}_0(\coprod_{i \in I} Y_i/X) \to X$, the generalizing base change theorem (\cite[Theorem 4.1.1(c)]{hub96}) and \cite[Exposé Vbis, Proposition 3.2.4]{sga4tome2}.
\end{proof}

  \begin{prop} \label{prop:huberstableuppeshrei}
  Let $X$ be a fine $k$-adic space with structure morphism $p_X \colon X \to k$. Then the dualizing complex $p_X^!\Lambda$ is constructible.
  \end{prop}
  
  \begin{proof}
  Denote by $\omega_X := p_X^!\Lambda$. Since the anaytic Verdier dual preserves semi-constructibility, cf. Theorem \ref{big result small proof} and the constant sheaf is semi-constructible, it follows that $D^{\mathrm{ad}}(\Lambda) = \omega_X$ is semi-constructible. Let $f \colon Y \to X$ be a morphism of fine $k$-adic spaces. By Lemma \ref{uppershbaschgrfi} it suffices to show that $f^*\omega_X$ is semi-constructible. Now $f^!\Lambda = D^{\mathrm{ad}} \circ f^* \circ D^{\mathrm{ad}}(\Lambda) = D^{\mathrm{ad}}(f^*\omega_X)$, so it suffices to show $f^!\Lambda$ is semi-constructible. Note first that if $X$ is smooth, then this is true since up to a shift and twist $f^!(\Lambda) = p_Y^!(\Lambda)$ and now one applies the result for $p_Y^!(\Lambda)$. In the general case, choose a hypercovering $a \colon Z_{\bullet} \to X$ in which all $Z_n$ are disjoint unions of affinoids in the analytifications of smooth affine schemes over $k$ (cf. Corollary \ref{cor:spehypsmooad}). Consider the cartesian diagram
    $$
\begin{tikzcd}[row sep = large, column sep = large]
Z'_{\bullet} \arrow[d, "a'"] \arrow[r, "f'"] &
Z_{\bullet} \arrow[d, "a"] \\
Y \arrow[r, "f"]
&
X. 
\end{tikzcd}
$$

Now by cohomological descent we have $\Lambda \simeq Ra_{*}(\Lambda)$ and so by Lemma \ref{lem:uppshreibaschange} we have $f^!(\Lambda) \simeq f^!(Ra_{*}(\Lambda)) \simeq Ra'_{*}(f'^!(\Lambda))$. Since the $Z_n$ are smooth, by the previous paragraph we have that $f'^!(\Lambda)$ is semi-constructible. Finally semi-constructibility is preserved via pushforward and so $Ra'_{*}(f'^!(\Lambda))$ is semi-constructible. It follows that $f^{!}(\Lambda)$ is semi-constructible and this proves the result.

  \end{proof}
  
  \begin{cor} \label{cor:uppeshreconstshea}
  Let $f \colon Y \to X$ be a morphism between fine $k$-adic spaces. Then $f^!(\Lambda)$ is constructible.
  \end{cor}
  
  \begin{proof}
   Constructibility of $f^!(\Lambda)$ is local for the étale topology on $Y$, cf. Theorem \ref{theorem above}. Thus by Lemma \ref{lem:factmorfinek}, we can assume that $f$ factors into $Y \xrightarrow{i} Z \xrightarrow{g} X$ 
  where $Z$ is a fine $k$-adic space, $i$ is a closed embedding and $g$ is a smooth morphism of pure dimension $d$. Poincaré duality gives $g^!\Lambda = g^*\Lambda[2d](d)$, which is again constructible. So it suffices to prove the proposition for $f = i$. If $Z$ is smooth, then this follows from Proposition \ref{prop:huberstableuppeshrei}. In the general case, choose a hypercovering $a \colon U_{\bullet} \to Z$ in which all $U_n$ are disjoint unions of affinoids in the analytifications of smooth affine schemes over $k$ (cf. Corollary \ref{cor:spehypsmooad}). Consider the cartesian diagram
    $$
\begin{tikzcd}[row sep = large, column sep = large]
U'_{\bullet} \arrow[d, "a'"] \arrow[r, "i'"] &
U_{\bullet} \arrow[d, "a"] \\
Y \arrow[r, "i"]
&
Z. 
\end{tikzcd}
$$
By cohomological descent one has $\Lambda \simeq Ra_*a^*(\Lambda)$ and so by Lemma \ref{lem:uppshreibaschange},  $i^!(\Lambda) \simeq i^!(Ra_*a^*(\Lambda)) \simeq Ra'_*i'^!(a^*(\Lambda))$. Now $\Lambda = a^*(\Lambda)$ is again constant and so by the previous paragraph $i'^!(a^*(\Lambda))$ is constructible. Finally constructibility is preserved via pushforward and so $Ra'_*i'^!(a^*(\mathscr{F}))$ is constructible. This means $i^!(\Lambda)$ is constructible as promised.
  \end{proof}
  
  \begin{cor} \label{cor:uppershrilocsisconst}
  Let $f \colon Y \to X$ be a morphism between fine $k$-adic spaces. Let $\mathscr{F}$ be a locally constant sheaf of finite type on $X_{\text{ét}}$ Then $f^!(\mathscr{F})$ is constructible.
  \end{cor}
  
  \begin{proof}
  Let $h_i \colon X_i \to X$ be an étale covering, such that the restriction $\mathscr{F}\lvert_{X_i}$ is the constant sheaf associated with a finitely generated $\Lambda$-module. Let $h'_i \colon Y_i = X_i \times_X Y \to Y$ be the corresponding étale covering of $Y$ and $f_i \colon Y_i \to X_i$ the base change of $f$. Then $h_i'^*f^!(\mathscr{F}) = f_i^!(\mathscr{F}\lvert_{X_i})$ and by Corollary \ref{cor:uppeshreconstshea}, $f_i^!(\mathscr{F}\lvert_{X_i})$ is constructible. Hence $f^!(\mathscr{F})$ is étale locally constructible and hence constructible by Theorem \ref{theorem above}.
  \end{proof}
  
  \begin{lem} \label{lem:fujwe}
  Let $X$ be a fine $k$-adic space and $\mathscr{F}$ a locally constant sheaf of finite type on $X_{\text{ét}}$. Then $D^{\mathrm{ad}}(\mathscr{F})$ is constructible.
  \end{lem}
  
  \begin{proof}
  Let $f \colon Y \to X$ be a morphism of fine $k$-adic spaces. We show that $f^*D^{\mathrm{ad}}(\mathscr{F})$ is semi-constructible. Note however $f^!(\mathscr{F}) = D^{\mathrm{ad}} \circ f^*D^{\mathrm{ad}}(\mathscr{F})$ and the result follows from Corollary \ref{cor:uppershrilocsisconst}. Similarly for an extension $L/k$ of algebraically closed non-Archimedean complete fields $g^!(\mathscr{F}) = D^{\mathrm{ad}} \circ g^*D^{\mathrm{ad}}(\mathscr{F})$ where $g \colon X_L \to X$ is the projection map. 
  \end{proof}
  
  \begin{lem} \label{lem:4.13}
  Let $X$ be a fine $k$-adic space and $j \colon U \to X$ be an étale morphism of fine $k$-adic spaces. Then $D^{\mathrm{ad}}(j_!\mathscr{F})$ is constructible where $\mathscr{F}$ is a locally constant sheaf of $\Lambda$-modules of finite type on $U_{\text{ét}}$. 
  \end{lem}
  
  \begin{proof}
  Since $D^{\mathrm{ad}}(j_!\mathscr{F}) = Rj_*D^{\mathrm{ad}}(\mathscr{F})$, the result follows from Lemma \ref{lem:fujwe} and stability of constructibility by pushforwards.
  \end{proof}
  
  \begin{lem} \label{lem:4.14}
  Let $f \colon Y \to X$ and $j \colon U \to X$ be morphisms between fine $k$-adic spaces with $j$ étale. Then $f^!(j_!\mathscr{F})$ is constructible where $\mathscr{F}$ is a locally constant sheaf of $\Lambda$-modules of finite type on $U_{\text{ét}}$. 
  \end{lem}
  
  \begin{proof}
Let $g \colon Z \to Y$ be a morphism of fine $k$-adic spaces. We have a commutative diagram 
    $$
\begin{tikzcd}[row sep = large, column sep = large]
U'' \arrow[d, "g'"] \arrow[r, "j''"] &
Z \arrow[d, "g"] \\
U' \arrow[d, "f'"] \arrow[r, "j'"] &
Y \arrow[d, "f"] \\
U \arrow[r, "j"]
&
X 
\end{tikzcd}
$$ 
where the squares are cartesian. It suffices to show that $g^*f^!(j_!\mathscr{F})$ is semi-constructible or equivalently $D^{\mathrm{ad}}(g^*f^!(j_!\mathscr{F})) = g^!D^{\mathrm{ad}}(f^!(j_!\mathscr{F}))$ is semi-constructible. We compute
\begin{align*}
g^!D^{\mathrm{ad}}(f^!(j_!\mathscr{F})) &\overset{(i)}{=} g^!D^{\mathrm{ad}}(D^{\mathrm{ad}} \circ f^* \circ D^{\mathrm{ad}}(j_!\mathscr{F})) \\
&\overset{(ii)}{=} g^!f^*D^{\mathrm{ad}}(j_!\mathscr{F}) \\
&\overset{(iii)}{=} g^!f^*Rj_*D^{\mathrm{ad}}(\mathscr{F}) \\
&\overset{(iv)}{=} g^!Rj'_*f'^*D^{\mathrm{ad}}(\mathscr{F}) \\
&\overset{(v)}{=} Rj''_*g'^!f'^*D^{\mathrm{ad}}(\mathscr{F}) \\
&\overset{(vi)}{=} Rj''_*g'^!D^{\mathrm{ad}}(f'^!\mathscr{F}) \\
&\overset{(vii)}{=} Rj''_*D^{\mathrm{ad}}(g'^*f'^!\mathscr{F})
\end{align*}
where (i) follows from $f^! = D^{\mathrm{ad}} \circ f^* \circ D^{\mathrm{ad}}$ for semi-constructible sheaves, (ii) follows from Lemma \ref{lem:4.13} which says that $D^{\mathrm{ad}}(j_!\mathscr{F})$ is constructible and hence so is $f^* \circ D^{\mathrm{ad}}(j_!\mathscr{F})$ (and in particular reflexive), (iii) follows from adjunction, (iv) follows from generalizing base change (cf. \cite[Theorem 4.1.1(c)]{hub96}), (v) follows from Lemma \ref{lem:uppshreibaschange}, (vi) follows from the constructibility of $f'^!\mathscr{F}$ (cf. Corollary \ref{cor:uppershrilocsisconst}) and (vii) follows from adjunction. Now $g'^*f'^!\mathscr{F}$ is constructible and so $D^{\mathrm{ad}}(g'^*f'^!\mathscr{F})$ is semi-constructible. Since semi-constructibility is stable via pushforward it follows that $Rj''_*D^{\mathrm{ad}}(g'^*f'^!\mathscr{F})$ is semi-construcible. 
  \end{proof}
  
    \begin{prop}
  Let $f \colon Y \to X$ be a morphism between fine $k$-adic spaces. Suppose $\mathscr{F}$ is a Huber constructible $\Lambda$-module on $X_{\text{ét}}$. Then $f^!(\mathscr{F})$ is constructible. 
  \end{prop}
  
  \begin{proof}
  By \cite[Lemma 2.7.9]{hub96}, every Huber constructible sheaf on $X_{\text{ét}}$ is a compact object. Thus the category of Huber constructible sheaves is a full subcategory of the smallest category which is closed under retracts and contains every object of the form $j_!\mathscr{F}$ where $j \colon U \to X$ is an étale morphism between fine $k$-adic spaces and $\mathscr{F}$ is a locally constant sheaf of $\Lambda$-modules of finite type on $U_{\text{ét}}$ (cf. the proof of \cite[Proposition 4.2.2]{lurgaitsinf}). We conclude by Lemma \ref{lem:4.14}.

  \end{proof}
  
  \subsection{Constructibility is not stable for the adic Verdier dual} \label{counterexample for Verdier dual stability}
  
        In this section we provide an example of a constructible sheaf whose Verdier dual is not constructible. 
        The example was
        suggested to us by P. Scholze while Johannes Nicaise helped us with the proof.

       \subsubsection{Construction} 
       
            Let $k$ be of equi-characteristic zero and
            let $X$ be the adic closed unit disk. Let 
            $r_n \in |k^*|$ be a decreasing sequence of real numbers that tend to zero. 
            For every $n$, let $X_n$ denote the adic closed ball around $0$ of radius $r_n$
            and let $h_n$ denote the open immersion $X_n \hookrightarrow X$. 
            We set $U_n := X_n \smallsetminus \{0\}$ and use $j_n$ to denote the 
            open embedding $U_n \hookrightarrow X_n$. 
            Let $\sF_n := h_{n*}j_{n!}(\Lambda)$ and $\sF = \oplus_{n \in \mathbb{N}} \sF_n$. 
            
            We need a lemma which allows us to calculate the cohomology of closed adic subspaces via tubular neighbourhoods. 
        
       \begin{lem} \label{global cohomology same as fibre cohomology}
        Let $r > 0$ be an element of $|k^*|$ and $X(r)$ denote the closed adic ball around $0$ of radius $r$. 
       Let $Y$ be a fine $k$-adic space and $f : Y \to X(r)$. There exists $0< s < r$ with $s \in |k^*|$ such that 
       for every $0 < s' \leq s$ with $s' \in |k^*|$, 
      the restriction map 
      \begin{align*} 
       R\Gamma(Y \times_{X(r)} X(s'),\Lambda) \to R\Gamma(f^{-1}(0),\Lambda) 
      \end{align*} 
      is an isomorphism.
       \end{lem}
    \begin{proof} 
       This follows from \cite[Theorem 3.6(a)]{hub98a} by taking the space $Y$ in place of Huber's choice of $X$,
       $f^{-1}(0)$ instead of $Z$ and noting that $\Lambda$ is oc-quasi-constructible. 
     \end{proof}

    \begin{prop} \label{verdier dual not constructible}
      The sheaf $\sF$ on $X$ is constructible. However, the stalk of its Verdier dual at the origin i.e. 
      $[D^{\mathrm{ad}}(\sF)]_{\overline{0}}$ is not in $\mathcal{D}^b_c(\Lambda)$. 
    \end{prop}
    \begin{proof} 
     We begin by showing that $\mathscr{F}$ is constructible. For $L/k$ an extension of algebraically closed non-Archimedean fields (extending the absolute value of $k$) we have a diagram (where each square is cartesian)
$$
\begin{tikzcd}[row sep = large, column sep = large]
U_{n,L} \arrow[d, "p_{n,L}^U"] \arrow[hookrightarrow, r, "j_{n,L}"] &
X_{n,L} \arrow[d, "p_{n,L}"] \arrow[hookrightarrow, r, "h_{n,L}"] &
X_L \arrow[d, "p_L"] \\
U_n \arrow[hookrightarrow, r, "j_n"] &
X_n \arrow[hookrightarrow, r, "h_n"] &
X.
\end{tikzcd}
$$ 
We compute
\begin{align*}
p_L^*\mathscr{F}_n &= p_L^*h_{n*}j_{n!}(\Lambda) \\
&\overset{(i)}{=} h_{n,L*}p_{n,L}^*j_{n!}(\Lambda) \\
&\overset{(ii)}{=} h_{n,L*}j_{n,L!}(\Lambda)
\end{align*}
where (i) follows from \cite[Theorem 4.1.1(c)]{hub96} and (ii) follows from \cite[Proposition 5.2.2(iv)]{hub96}. Hence by Theorem \ref{theorem above}, it suffices to prove the following statement. 
      If $f : Y \to X$ is a morphism of fine $k$-adic spaces then
     for every $i \in \mathbb{Z}$, $H^i(Y,f^*(\sF))$ is finite. 
     We make use of the notation from the lemma above.  
     For $0 < r < 1$ and $r \in |k^*|$, let $Y(r)$ denote the 
     By Lemma \ref{global cohomology same as fibre cohomology}, 
     there exists $0 < s < 1$ with $s \in |k^*|$ such that for every $0 < s' \leq s$ with $s' \in |k^*|$,
     the natural map $R\Gamma(Y(s),\Lambda) \to R\Gamma(Y_0,\Lambda)$ is an 
     isomorphism.  
      
       We claim that for all $n \in \mathbb{N}$ such that 
       $r_n \leq s$, 
       $R\Gamma(Y,f^*(\sF_n)) = 0$.   
       By \cite[Proposition 5.2.2(iv)]{hub96}, $f_{|Y(r_n)}^*j_{n!}(\Lambda) \simeq j'_{n!}(\Lambda)$ 
       where $j'_n$ is the open embedding $Y(r_n) \smallsetminus Y_0 \hookrightarrow Y(r_n)$.
        
       Lemma \ref{vanishing lower star} applied to the triangle
               \begin{align}
          Rh_{n*}j_{n!}(\Lambda) \to Rh_{n*}\Lambda \to Rh_{n*}i_{n*}(\Lambda) \to \cdot 
        \end{align}
         where $i_n$ denotes the embedding $\{0\} \hookrightarrow X_n$, implies that 
       $\sF_n  = Rh_{n*}j_{n!}(\Lambda)$.  
       By Theorem 4.3.1 in loc.cit. and the isomorphism above, 
      $f^*(\sF_n) = f^*Rh_{n*}j_{n!}(\Lambda) \simeq Rh'_{n*} j'_{n!}(\Lambda)$ 
      where $h'_{n}$ is the embedding $Y(r_n) \hookrightarrow Y$. 
      Hence, 
     \begin{align} 
     R\Gamma(Y,f^*(\sF_n)) \simeq R\Gamma(Y(r_n),j'_{n!}(\Lambda)). 
    \end{align} 

      Observe that we have the following exact sequence on $Y(r_n)$
      \begin{align*} 
        0 \to j'_{n!}(\Lambda) \to \Lambda  \to i_*\Lambda_{|Y_{0}} \to 0 
           \end{align*} 
           where $i : Y_{0} \hookrightarrow Y(r_n)$. This
      gives us the triangle
    \begin{align*} 
     R\Gamma(Y(r_n),j'_{n!}(\Lambda)) \to R\Gamma(Y(r_n),\Lambda) \to R\Gamma(Y_{0},\Lambda) \to \cdot
    \end{align*} 
     By Lemma \ref{global cohomology same as fibre cohomology} and (4) above, $R\Gamma(Y,f^*(\sF_n)) \simeq 0$. 
     We have thus verified the claim. 
     
     The claim above implies that 
     \begin{align*} 
     H^i(Y,f^*(\sF)) =& \oplus_{n \in \mathbb{N}} H^i(Y,\sF_n) \\
                            =&   \oplus_{n | r_n > s} H^i(Y,\sF_n).
                            \end{align*} 
     Hence, $\sF$ is constructible. 
     
       By Lemma \ref{Verdier dual computation}, the Verdier dual of $\sF$ is the complex $R\prod_n h_{n!}Rj_{n*}(\Lambda[2])$. 
        Finally by Lemma \ref{lem:infinstal} the stalk of $D^{\mathrm{ad}}(\sF)$ at $0$ is infinite and it is therefore not constructible.           
    \end{proof}
    
    \begin{lem} \label{Verdier dual computation}
      The Verdier dual of $\sF$ is the complex (up to a twist) $R\prod_n h_{n!}Rj_{n*}(\Lambda[2])$. 
        \end{lem}
      \begin{proof} 
        By definition, $\sF = \oplus_n \sF_n$. It follows that 
        $D^{\mathrm{ad}}(\sF) = R\prod_n D^{\mathrm{ad}}(\sF_n)$. 
        Let $i_n$ denote the embedding $\{0\} \hookrightarrow X_n$. 
        We claim that $j_{n!}(\Lambda)$ is semi-constructible. 
        Indeed, consider the following distinguished triangle 
        \begin{align}
          j_{n!}(\Lambda) \to \Lambda \to i_{n*}(\Lambda) \to \cdot 
        \end{align}  
       Since $\Lambda$ and $i_{n*}(\Lambda)$ are semi-constructible, we get that 
       $j_{n!}(\Lambda)$ is semi-constructible as well. 
       It follows that $Rh_{n*}(j_{n!}(\Lambda))$ is also semi-constructible. 
       Applying $Rh_{n*}$ to the triangle above and invoking Lemma 
       \ref{vanishing lower star}, we deduce that 
       $Rh_{n*}(j_{n!}(\Lambda)) \simeq h_{n*}(j_{n!}(\Lambda))$. 
       By Corollary \ref{Corollary 2}, $D^{\mathrm{ad}}(\sF_n) = h_{n!}D^{\mathrm{ad}}(j_{n!}(\Lambda))$.
      
        To calculate $D^{\mathrm{ad}}(j_{n!}(\Lambda))$, consider
        the following isomorphisms, 
       \begin{align*} 
       D^{\mathrm{ad}}(j_{n!}(\Lambda)) &\simeq R\mathscr{H}om(j_{n!}(\Lambda),p_{X_n}^!(\Lambda)) \\  
                                   &\simeq Rj_{n*}R\mathscr{H}om(\Lambda,p_{U_n}^!(\Lambda)) \\ 
                                   &\simeq Rj_{n*}p_{U_n}^!(\Lambda)
                                   \end{align*} 
                                    where the second isomorphism follows by adjointness. 
        Since $U_n$ is smooth, we get that 
        $p_{U_n}^! \simeq p_{U_n}^*[2]$ (ignoring twists). 
        Hence we see that 
        $D^{\mathrm{ad}}(\sF_n) = h_{n!}Rj_{n*}(\Lambda[2])$. 
       
      \end{proof} 
    
    \begin{lem}  \label{lem:infinstal}
      Let $A$ be the complex $[R\prod_n h_{n!}Rj_{n*}(\Lambda)]_{\overline{0}}$. 
       We have that $H^0(A)$ is infinite.     
    \end{lem} 
    \begin{proof}  Since we assumed the field $k$ to be of equi-characteristic zero, we have that the system of closed disks centred 
      at $0$ with radii in 
      $|k^*|$ form a fundamental system of étale neighbourhoods of the origin. Indeed,
    let $Y \to X$ be a fine étale neighbourhood of the origin. 
     By Lemma \ref{refining covers}, there exists an open affinoid neighbourhood $V$ of the origin 
     and an affinoid open subset $U$ of $Y$ such that $U \to V$ is finite. Since the 
     family of closed disks containing the origin forms a fundamental system of open neighbourhoods
     in the adic topology, we replace $V$ with a closed disk that is contained in it. By \cite[Theorem 6.3.2]{berk},
     and since the field $k$ is of equi-characteristic zero, we get that 
     $U \to V$ is the identity map.  
     This shows that the family of adic closed disks around the origin are terminal amongst all étale neighbourhoods of the origin. 
     Hence the claim. 
     
           By definition and since $R\prod_n$ commutes with taking global sections,
     \begin{align*} 
     A = \varinjlim_{0 < s < 1 | s \in |k^*|} R\prod_n R\Gamma(X(s), h_{n!}Rj_{n*}(\Lambda)). 
     \end{align*} 
   For every $n \in \mathbb{N}$, $R\Gamma(X(s), h_{n!}Rj_{n*}(\Lambda))$ has finite cohomology (this is because $D^{\mathrm{ad}}(\sF_n)$ is semi-constructible). Hence 
   by the Mittag-Leffler condition applied to the spectral sequence
   \[
   E_2^{p,q} = R^p\prod_nH^{q}( R\Gamma(X(s), h_{n!}Rj_{n*}(\Lambda))) \Rightarrow H^{p+q}\left( R\prod_n R\Gamma(X(s), h_{n!}Rj_{n*}(\Lambda))\right)
   \]
   we see that 
   $$H^{0}\left(R\prod_n R\Gamma(X(s), h_{n!}Rj_{n*}(\Lambda)) \right) = \prod_n H^{0}(R\Gamma(X(s), h_{n!}Rj_{n*}(\Lambda))).$$ 
   Now by the spectral sequence
   \[
   E_2^{p,q} = R^p\Gamma(X(s), h_{n!}R^{q}j_{n*}(\Lambda)) \Rightarrow H^{p+q}( R\Gamma(X(s), h_{n!}Rj_{n*}(\Lambda)))
   \]
   we see that
   \begin{align*} 
      H^0(A) &= \varinjlim_{0 < s < 1 | s \in |k^*|} \prod_n H^0 R\Gamma(X(s), h_{n!}Rj_{n*}(\Lambda)) \\
                  &= \varinjlim_{0 < s < 1 | s \in |k^*|} \prod_n H^0(X(s), h_{n!}j_{n*}(\Lambda)). 
      \end{align*}
      Now for $s < r_n$, $H^0(X(s), h_{n!}j_{n*}(\Lambda)) = \Lambda$ and so we see that $H^0(A)$ is infinite.

       \end{proof}

    \begin{lem}\label{vanishing lower star}
      Let $h : X(r) \hookrightarrow X$ be the open immersion from the closed adic disc around zero of radius $r$ to
      the unit disc. Then $R^ih_*(\Lambda) = 0$ if $i \geq 1$.     
    \end{lem}
    \begin{proof} 
       Let $h^{\mathrm{an}}$ denote the associated map between the respective Berkovich spaces. 
     By \cite[Proposition 8.3.5]{hub96}, $Rh_*(\Lambda) = \mu_X^*Rh^{\mathrm{an}}_*(\Lambda)$ where 
     $\mu_X : X_{\text{ét}} \to X^{\mathrm{an}}_{\text{ét}}$.
    However, $Rh^{\mathrm{an}}_* = h^{\mathrm{an}}_*$ because $h^{\mathrm{an}}_*$ is an exact functor. 
    Indeed, firstly observe that for an étale sheaf $\sG$ on $X(r)^{\mathrm{an}}$, $[h^{\mathrm{an}}_*(\sG)]_y = 0$ where 
    $y \in X^{\mathrm{an}} \smallsetminus X(r)^{\mathrm{an}}$. This
     follows from the fact that $X^{\mathrm{an}} \smallsetminus X(r)^{\mathrm{an}}$ 
     is an open analytic domain of $X^{\mathrm{an}}$ and hence an object of the site 
     $X^{\mathrm{an}}_{\text{ét}}$. The exactness of $h^{\mathrm{an}}_*$ can be deduced from this fact. 
        \end{proof}
        
       \begin{es}[\textbf{Non-stability of $R\mathscr{H}om(-,-)$}]\label{es:nonstarhm}
         \emph{Observe that $D^{\mathrm{ad}}(\sF)$ is a semi-constructible complex that is not constructible. In particular $R\mathscr{H}om(-,-)$ is not stable for constructiblility. 
        Moreover we show that $R\mathscr{H}om(A,B)$ is not stable for semi-constructible complexes $A$ and $B$. 
  Indeed, let $x \in X$ be the origin. The geometric stalk $i_{x}^*(D^{\mathrm{ad}}(\sF))$ is not a complex with finite cohomology. Let  
   $A := D^{\mathrm{ad}}(\sF)$, $B = i_{x*}(\Lambda)$. Then $R\Gamma(X,R\mathscr{H}om(A,B)) =    
   RHom(i_x^*A,\Lambda)$ which is not a complex with finite cohomology by our choice of $x$ and since $\Lambda$ is an injective $\Lambda$-module on $\Spa(k)$.}
   \end{es}
   
   \begin{es}[\textbf{Non-stability of upper shriek}] \label{es:nonstabupshr}
   \emph{Furthermore by Corollary \ref{Corollary 2}(3), and
   $$i_x^!(\sF) = D^{\mathrm{ad}}i_x^*(D^{\mathrm{ad}}(\sF))$$ and hence $i_x^!(\sF)$ is neither constructible nor semi-constructible.}
  \end{es}
  
  \begin{es}[\textbf{Non-stability of $- \otimes^L -$}]
   \emph{It can be deduced from Corollary \ref{Corollary 2}(4) that $- \otimes^L -$ is not stable for semi-constructible sheaves. Indeed $R\mathscr{H}om(A,B) = D^{\mathrm{ad}}(A \otimes^L D^{\mathrm{ad}}B)$ is not semi-constructible by Example \ref{es:nonstarhm}. Thus $A \otimes^L D^{\mathrm{ad}}B$ is not semi-constructible.} 
     \end{es}

  \subsubsection{Non-stability of $\otimes^L$} \label{counterexample for tensor product stability}
     
     We provide an example of constructible sheaves $\sF_1$ and $\sF_2$ on a fine $k$-adic space such that  
     $\sF_1 \otimes^L \sF_2$ is not constructible. 
     We recall the notation introduced in \S 4.3. 
     
      Let $k$ be of equi-characteristic zero and
            let $X$ be the adic closed unit disk. Let 
            $r_n \in |k^*|$ be a decreasing sequence of real numbers that tend to zero. 
            For every $n$, let $X_n$ denote the adic closed ball around $0$ of radius $r_n$
            and let $h_n$ denote the open immersion $X_n \hookrightarrow X$. 
            We set $U_n := X_n \smallsetminus \{0\}$ and use $j_n$ to denote the 
            open embedding $U_n \hookrightarrow X_n$. 
            
     For every $n \in \mathbb{N}$, let 
     $u_n : V_n \hookrightarrow X_n$ denote an adic closed ball that is defined over $k$ and does not contain $0$. 
     Let $v_n := h_n \circ u_n : V_n \hookrightarrow X$. 
     Let $a_n, b_n$ be distinct $k$-points in $V_n$. 
     Let $W_{1n} := V_n \smallsetminus \{a_n\}$ and $W_{2n} :=  V_n \smallsetminus \{b_n\}$. 
     We use $w_{1n}$ and $w_{2n}$ to denote the open immersions 
     $W_{1n} \hookrightarrow V_n$ and $W_{2n} \hookrightarrow V_n$. 
     Let $\sF_1 := \bigoplus_{n \in \mathbb{N}} v_{n*}w_{1n!}(\Lambda)$
     and $\sF_2 := \bigoplus_{n \in \mathbb{N}} v_{n*}w_{2n!}(\Lambda)$. 
     For every $n \in \mathbb{N}$, let
     $\sF_{1n} := v_{n*}w_{1n!}(\Lambda)$ and $\sF_{2n} := v_{n*}w_{2n!}(\Lambda)$. 
      
      We begin by showing that the sheaf $\sF_1$ is constructible.
      
     \begin{lem} \label{global cohomology same as fibre cohomology 2}
        Given $r \in \mathbb{R}$, let $X(r)$ denote the closed unit adic ball around $0$ of radius $r$ with $r \in |k^*|$.
       Let $Y$ be a fine $k$-adic space and $f : Y \to X(r)$ be a morphism of $k$-adic spaces. There exists $0< s < r$ 
       such that 
       for every $0 < s' \leq s$
       and every $i \in \mathbb{N}$, 
       $R^if_*(\Lambda)$ is a locally constant sheaf of finite type when restricted to  
       $X(s') \smallsetminus \{0\}$.  
           \end{lem}
           \begin{proof}
                     By Proposition \ref{semi-con stability for pushforwards}, $Rf_*(\Lambda) \in \mathcal{D}^b_{sc}(X(r)^f_{\text{ét}},\Lambda)$. In particular,  
          there exists $d \in \mathbb{N}$ such that for every $d' \geq d$, 
          $R^{d'}f_*(\Lambda) = 0$. 
          Hence it suffices to show that if 
          $i \in D := \{1,\ldots,d\}$ 
          then there exists a closed disk $U$ around $0$ such that 
          $R^if_*(\Lambda)_{|U \smallsetminus \{0\}}$ is locally constant of finite type.  
          
          Now note that the set of adic closed disks centred at $0$ forms
         a fundamental system of open neighbourhoods of $0$ 
        for the adic topology. This follows from the explicit description of compact 
        neighbourhoods for the Berkovich analytic topology of the closed disk as
         implied by \cite[Proposition 1.6]{BR10}. Although the reference doesn't concern adic spaces, one 
         can still apply it since it only suffices to describe the affinoid open subspaces of $X$ and these correspond to
         Berkovich affinoid domains of $X$ via
           the equivalences in 
         \S \ref{rigid varieties and adic spaces}.
          

               By \cite[Theorem 0.1]{hub98a}, the sheaf
               $R^if_*(\Lambda)$ is oc-quasi-constructible. By the proof of \cite[\S 1.3 (iv)]{hub98a} we see that there exists an open subset $U$ containing 0 such that $R^if_*(\Lambda)_{|U \smallsetminus \{0\}}$ is locally constant of finite type. The result now follows by the previous paragraph.

                    \end{proof} 
   
        \begin{lem} \label{global cohomology same as fibre cohomology 3}
       Let $Y$ be a fine $k$-adic space and $f : Y \to X$ be a morphism of $k$-adic spaces.
        There exists $n_0 \in \mathbb{N}$ such that for every $n \geq n_0$, 
      the restriction map 
      \begin{align*} 
       R\Gamma(Y \times_{X} V_n,\Lambda) \to R\Gamma(f^{-1}(a_n),\Lambda) 
      \end{align*} 
      is an isomorphism.
       \end{lem}
    \begin{proof} 
       By Lemma \ref{global cohomology same as fibre cohomology 2},
       there exists 
       $n_0 \in \mathbb{N}$ such that 
       for every $i \in \mathbb{N}$ and $n \geq n_0$  
       $R^if_*(\Lambda)_{|V_n}$ is locally constant. 
       By \cite[Theorem 6.3.2]{berk}, 
       we get that  
        $R^if_*(\Lambda)_{|V_n}$ is constant.
        Observe that since $V_n \to X$ is étale, if 
        $f'$ denotes the base change $Y \times_X V_n \to V_n$ then 
        $R^if_*(\Lambda)_{|V_n} = R^if'_*(\Lambda)$. 
         
         We have the following spectral sequence
         \begin{align*} 
           H^p(V_n,R^qf'_*(\Lambda)) \Rightarrow H^{p+q}(Y \times_X V_n,\Lambda). 
         \end{align*} 
         Since $R^qf'_*(\Lambda)$ is constant and $V_n$ is a closed $k$-adic disk, we know that $H^p(V_n,R^qf'_*(\Lambda)) = 0$ for all $p \geq 1$. 
           Hence, $H^0(V_n,R^qf'_*(\Lambda)) = H^q(Y \times_X V_n,\Lambda)$. 
           By the quasi-compact base change theorem (cf. \cite[Theorem 4.3.1]{hub96}), 
           $[R^qf'_*(\Lambda)]_{\overline{a_n}} = H^q(f^{-1}(a_n),\Lambda)$. 
           Since $R^qf'_*(\Lambda)$ is constant and $V_n$ is connected, we get that 
           $[R^qf'_*(\Lambda)]_{\overline{a_n}} = H^0(V_n,R^qf'_*(\Lambda))$. Thus, 
           \begin{align*} 
           H^q(Y \times_X V_n,\Lambda) = H^q(f^{-1}(a_n),\Lambda).
           \end{align*}  
                   \end{proof}

  
      \begin{prop} 
      The sheaf $\sF_1$ on $X$ is constructible. 
    \end{prop}
    \begin{proof} 
    For $L/k$ an extension of algebraically closed non-Archimedean fields (extending the absolute value of $k$) we have a diagram (where each square is cartesian)
$$
\begin{tikzcd}[row sep = large, column sep = large]
W_{1n,L} \arrow[d, "p_{n,L}^W"] \arrow[hookrightarrow, r, "w_{1n,L}"] &
V_{n,L} \arrow[d, "p_{n,L}"] \arrow[hookrightarrow, r, "v_{n,L}"] &
X_L \arrow[d, "p_L"] \\
W_{1n} \arrow[hookrightarrow, r, "w_{1n}"] &
V_n \arrow[hookrightarrow, r, "v_n"] &
X.
\end{tikzcd}
$$ 
We compute
\begin{align*}
p_L^*\mathscr{F}_{1n} &= p_L^*v_{n*}w_{1n!}(\Lambda) \\
&\overset{(i)}{=} v_{n,L*}p_{n,L}^*w_{1n!}(\Lambda) \\
&\overset{(ii)}{=} v_{n,L*}w_{1n,L!}(\Lambda)
\end{align*}
where (i) follows from \cite[Theorem 4.1.1(c)]{hub96} and (ii) follows from \cite[Proposition 5.2.2(iv)]{hub96}. Hence by Theorem \ref{theorem above}, it suffices to prove the following statement. 
      If $f : Y \to X$ is a morphism of fine $k$-adic spaces then
     for every $i \in \mathbb{Z}$, $H^i(Y,f^*(\sF_1))$ is finite. 
     By Lemma \ref{global cohomology same as fibre cohomology 3}, there exists 
     $n_0 \in \mathbb{N}$ such that for every $n \geq n_0$, 
      the restriction map 
      \begin{align*} 
       R\Gamma(Y \times_{X} V_n,\Lambda) \to R\Gamma(f^{-1}(a_n),\Lambda) 
      \end{align*} 
      is an isomorphism.

       We claim that for all $n \in \mathbb{N}$ such that 
       $n \geq n_0$, 
       $R\Gamma(Y,f^*(\sF_{1n})) = 0$.   
      Let $Y_{V_n} := Y \times_X V_n$ and $Y_{a_n}$ be the fibre over $a_n$. 
       By \cite[Proposition 5.2.2(iv)]{hub96}, $f_{|Y_{V_n}}^*w_{1n!}(\Lambda) \simeq w'_{1n!}(\Lambda)$ 
       where $w'_{1n}$ is the open embedding $Y_{V_n} \smallsetminus Y_{a_n} \hookrightarrow Y_{V_n}$.
        
       Lemma \ref{vanishing lower star} 
          implies that 
       $\sF_{1n}  = Rv_{n*}w_{1n!}(\Lambda)$.  
       By Theorem 4.3.1 in loc.cit. and the isomorphism above, 
      $f^*(\sF_{1n}) = f^*Rv_{n*}w_{1n!}(\Lambda) \simeq Rv'_{n*} w'_{1n!}(\Lambda)$ 
      where $v'_{n}$ is the embedding $Y_{V_n} \hookrightarrow Y$. 
      Hence, 
     \begin{align} 
     R\Gamma(Y,f^*(\sF_{1n})) \simeq R\Gamma(Y_{V_n},w'_{1n!}(\Lambda)). 
    \end{align} 

      Observe that we have the following exact sequence on $Y_{V_n}$
      \begin{align*} 
        0 \to w'_{1n!}(\Lambda) \to \Lambda  \to i_*\Lambda_{|Y_{a_n}} \to 0 
           \end{align*} 
           where $i : Y_{a_n} \hookrightarrow Y_{V_n}$. This
      gives us the triangle
    \begin{align*} 
     R\Gamma(Y_{V_n},w'_{1n!}(\Lambda)) \to R\Gamma(Y_{V_n},\Lambda) \to R\Gamma(Y_{a_n},\Lambda) \to \cdot
    \end{align*} 
     By Lemma \ref{global cohomology same as fibre cohomology 3} and (6) above, $R\Gamma(Y,f^*(\sF_{1n})) \simeq 0$. 
     We have thus verified the claim. 
     
     The claim above implies that 
     \begin{align*} 
     H^i(Y,f^*(\sF_1)) =& \oplus_{n \in \mathbb{N}} H^i(Y,f^*\sF_{1n}) \\
                            =&   \oplus_{n | n \leq n_0} H^i(Y,f^*\sF_{1n}).
                            \end{align*} 
     Hence, $\sF_1$ is constructible since for every $n \in \mathbb{N}$, $\sF_{1n}$ is constructible and hence 
     $H^i(Y,f^*\sF_{1n})$ is finite for all $i \in \mathbb{N}$.  
             \end{proof}
             
    We have thus shown that both $\sF_1$ and $\sF_2$ are constructible. We now show that 
    $\sF_1 \otimes^L \sF_2$ is not constructible.
             
\begin{prop} 
  The complex $\sF_1 \otimes^L \sF_2$ does not 
  belong to $\mathcal{D}^b_c(X^f_{\text{ét}},\Lambda)$.  
  \end{prop} 
  \begin{proof} 
   We deduce from the definitions of $\sF_1$ and $\sF_2$ that 
   \begin{align*} 
    \sF_1 \otimes^L \sF_2 \simeq \oplus_{n \in \mathbb{N}} \oplus_{m \in \mathbb{N}} (\sF_{1n} \otimes^L \sF_{2m}). 
   \end{align*} 
      Hence for every $i \in \mathbb{N}$, 
      \begin{align*} 
    \mathbb{H}^i(X,\sF_1 \otimes^L \sF_2) \simeq \oplus_{n} \oplus_m \mathbb{H}^i(X,\sF_{1n} \otimes^L \sF_{2m}).  
      \end{align*} 
      where $\mathbb{H}^i(X,-) = H^iR\Gamma(X, -)$ is the hyper-cohomology. 
      By Lemma \ref{the first cohomology does not vanish}, we get that 
      for all $n$, $\mathbb{H}^i(X,\sF_{1n} \otimes^L \sF_{2n})$ is not zero. Hence
      $\mathbb{H}^1(X,\sF_1 \otimes^L \sF_2)$ is infinite which implies that $\sF_1 \otimes^L \sF_2$ is not even
      semi-constructible.     
  \end{proof} 

   Let $n \in \mathbb{N}$. To simplify notation, we write $\sG := \sF_{1n} \otimes^L \sF_{2n}$. 

\begin{lem} \label{the first cohomology does not vanish}
  We have that $\mathbb{H}^1(X,\sG) \neq 0$. 
\end{lem} 
\begin{proof} 
   By definition, 
   \begin{align*} 
   \sG := v_{n*}w_{1n!}\Lambda \otimes^L v_{n*}w_{2n!}\Lambda. 
   \end{align*} 
   Let $\overline{V_n} = (X,\overline{|V_n|})$ be the closed 
   pseudo-adic subspace \cite[Definition 1.10.8 (i)]{hub96} where 
   $\overline{|V_n|}$ is the topological space which is the closure of $|V_n|$ in $|X|$. 
   Let $\bar{v}_n : \overline{V_n} \hookrightarrow X$ denote the associated closed embedding. 
    Observe that $\mathrm{supp}(\sG) \subset \overline{V_n}$. 
    By \cite[Lemma 1.10.17 (i)]{hub96}, $\bar{v}_n$ is proper and hence $\bar{v}_{n*} = \bar{v}_{n!}$. 
    It follows that $\bar{v}_{n*}$ is exact. Furthermore, one checks at stalks that 
     we have an isomorphism 
    $\sG \simeq \bar{v}_{n*}\sG_{|\overline{V_n}}$. 
    In particular, it suffices to show that $\mathbb{H}^1(\overline{V_n},\sG_{|\overline{V_n}})$ is not zero. 
    
    Observe that $p_n : V_n \hookrightarrow \overline{V_n}$ is an open embedding whose complement 
    is the pseudo adic space $x := (X, |\overline{V_n}| \smallsetminus |V_n|)$.   
    Let $i_{x} : \{x\} \hookrightarrow \overline{V_n}$ be the closed embedding of pseudo-adic spaces. 
    Consider the following 
    exact sequence associated to the 
    open embedding $V_n \hookrightarrow \overline{V_n}$ : 
    \begin{align*} 
      p_{n!}\sG_{|V_n} \to \sG_{|\overline{V_n}} \to i_{x*}i_x^*(\sG_{|\overline{V_n}}) \to \cdot 
        \end{align*}     
      Observe that the open embedding $v_n : V_n \hookrightarrow X$   
      factors through the closed embedding $\overline{V_n} \hookrightarrow X$ via the map $p_n$. 
      Hence, for every $m \in \mathbb{N}$, 
      $\mathbb{H}^m(\overline{V_n},p_{n!}\sG_{|V_n}) \simeq \mathbb{H}^m(X,v_{n!}\sG_{|V_n})$. 
     Applying Lemmas \ref{intermediate calculation I} and \ref{intermediate calculation II},
     gives the following long exact sequence :
     \begin{align*} 
      0 \to 0 \to \mathbb{H}^0(\overline{V_n},\sG_{|\overline{V_n}}) \to \Lambda \to \Lambda \oplus \Lambda \to \mathbb{H}^1(\overline{V_n},\sG_{|\overline{V_n}}) \to \ldots 
     \end{align*} 
         Since the map $\Lambda \to \Lambda \oplus \Lambda$ cannot be surjective, we get that 
         $\mathbb{H}^1(\overline{V_n},\sG_{|\overline{V_n}})$ cannot be zero. 
\end{proof} 
   
 \begin{lem} \label{intermediate calculation I}
  We have the following isomorphisms. 
  \begin{enumerate} 
  \item $\mathbb{H}^0(X,v_{n!}(\sG_{|V_n})) \simeq 0.$
  \item $\mathbb{H}^1(X,v_{n!}(\sG_{|V_n})) \simeq \Lambda \oplus \Lambda$.
  \item $\mathbb{H}^2(X,v_{n!}(\sG_{|V_n})) \simeq \Lambda$. 
  \end{enumerate}
 \end{lem}   
 \begin{proof} 
 Recall that the derived tensor product commutes with pullbacks i.e. the identity $g^*(A \otimes^L B) \simeq g^*A \otimes^L g^*B$. 
   One deduces from this that $\sG_{|V_n \smallsetminus \{a_n,b_n\}}$ is the constant sheaf $\Lambda$ 
   and
   $\sG_{\overline{a_n}} = \sG_{\overline{b_n}} = 0$. Hence we see that 
    $v_{n!}(\sG_{|V_n}) \simeq v'_{n!}(\Lambda)$ where $v'_n$ is the open
    immersion $V_n \smallsetminus \{a_n,b_n\} \hookrightarrow X$.
    On $V_n$, we have the following exact sequence
    \begin{align*} 
     0 \to v''_{n!}(\Lambda) \to \Lambda \to i_{a_n*}(\Lambda) \oplus i_{b_n*}(\Lambda) \to 0 
    \end{align*} 
  where $v''_n : V_n \smallsetminus \{a_n,b_n\} \hookrightarrow V_n$ and $i_{a_n} : \{a_n\} \hookrightarrow V_n$.  
  Composing with $v_{n!}$ gives 
  \begin{align*} 
       0 \to v'_{n!}(\Lambda) \to v_{n!}(\Lambda) \to i_{a_n*}(\Lambda) \oplus i_{b_n*}(\Lambda) \to 0 
  \end{align*} 
    where we abuse notation and use $i_{a_n}$ to denote the closed embedding $a_n \hookrightarrow X$ as well. 
   Applying the global sections functor gives a long exact sequence in cohomology. Hence, it suffices 
   to determine $H^j(X,v_{n!}(\Lambda))$ for all $j \in \mathbb{N}$. 
   Since $\overline{V_n}$ embeds into $X$, we see that $H^j(X,v_{n!}(\Lambda)) \simeq H^j_c(V_n,\Lambda)$. 
   By \cite[Example 0.4.6]{hub96}, 
   $H_c^j(V_n,\Lambda)$ is $0$ if 
       $j \neq 2$ and equal to $\Lambda$ when $j = 2$. 
      
        \end{proof} 
 
  \begin{lem} \label{intermediate calculation II}
   Recall the notation introduced in Lemma \ref{the first cohomology does not vanish}. 
   Let $x$ be the pseudo-adic space $(X, \overline{|V_n|} \smallsetminus |V_n|)$ and
    let $i_{x} : \{x\} \hookrightarrow \overline{V_n}$ be the closed embedding of pseudo-adic spaces. 
  We have the following isomorphism 
  $$\mathbb{H}^0(x,i_x^*(\sG_{|\overline{V_n}})) \simeq \Lambda.$$
 \end{lem}   
    \begin{proof} 
    Since pullbacks commute with derived tensor products, we get that 
    $$i_x^*(\sG_{|\overline{V_n}}) \simeq i_x^*v_{n*}w_{1n!}\Lambda \otimes^L i_x^*v_{n*}w_{2n!}\Lambda.$$
   
       Observe that the complement of $\{a_n\}$ in 
       $\overline{V_n}$ is an étale open neighbourhood of $x$. 
       Let us call this open neighbourhood $O$ and $j_O : O \hookrightarrow \overline{V_n}$ be the associated open embedding.  
       The morphism $i_x : x \hookrightarrow \overline{V_n}$ factors through $j_O$ and the inclusion $x \hookrightarrow O$ which we abuse notation and 
       denote by $i_x$ as well.
       Hence 
       $i_x^*p_{n*}w_{1n!}\Lambda \simeq i_x^*j_O^*p_{n*}w_{1n!}\Lambda$. 
       By smooth base change \cite[Theorem 4.1.1]{hub96}, 
        $i_x^*j_O^*p_{n*}w_{1n!}\Lambda \simeq i_x^*p'_{n*}\Lambda$ where $p'_n : V_n \times_{\overline{V_n}} O \to O$. 
       We now show 
       $i_x^*p'_{n*}\Lambda = i_x^*p_{n*}(\Lambda)$. 
       Indeed, by identical arguments as above 
       \begin{align*} 
       i_x^*p_{n*}(\Lambda) &\simeq i_{x}^*j_O^*p_{n*}(\Lambda) \\
                      &\simeq i_{x}^*p'_{n*}(\Lambda).
                      \end{align*}
       Observe that $i_x^*p_{n*}(\Lambda) = i_x^*v_{n*}(\Lambda)$. This is because 
       $v_n$ can be seen as the composition of $p_n : V_n \to \overline{V_n}$ and 
       $\overline{V_n} \hookrightarrow X$.  
        Hence we see that 
        \begin{align}
        i_x^*(\sG_{|\overline{V_n}}) \simeq i_x^*v_{n*}\Lambda \otimes^L i_x^*v_{n*}\Lambda
        \end{align}
                
         We claim $i_x^*v_{n*}\Lambda \simeq \Lambda$.
         We first show that $H^0(x,i_x^*v_{n*}\Lambda) = \Lambda$.        
         Consider the following exact sequence on $\overline{V_n}$ : 
        $$0 \to p_{n!}\Lambda \to p_{n*}\Lambda \to i_{x*}i_x^*p_{n*}\Lambda \to 0$$
        where $p_n : V_n \hookrightarrow \overline{V_n}$. 
        Applying the derived functor $R\Gamma(\overline{V_n},-)$ 
        gives a long exact sequence in cohomology. Note 
        that $H^m(\overline{V}_n,p_{n!}(\Lambda)) = H^m_c(V_n,\Lambda)$. 
        In Lemma \ref{intermediate calculation I}, we saw that 
        $H^m_c(V_n,\Lambda) = 0$ if $m \neq 2$ and is equal to $\Lambda$
        if $m = 2$. 
        Since ${V_n}$ is connected, we
        get that $H^0(\overline{V_n},p_{n*}\Lambda) = H^0(V_n,\Lambda) = \Lambda$. The long exact sequence
        proves that $H^0(x, i_x^*v_{n*}\Lambda) = \Lambda$. 
        
             By \cite[Theorem 8.3.5]{hub96}, the sheaf 
             $v_{n*}(\Lambda)$ is the pullback of a sheaf on the étale site of the
              Berkovich space $X^{\mathrm{an}}$ (the Berkovich analytic unit disk).
             Hence $v_{n*}(\Lambda)$ is overconvergent. 
             Let $\overline{x}$ be a geometric point over $x$ and 
             let $y$ be the unique generalization of $x$. 
             By definition of overconvergence, we have a bijection 
             $[v_{n*}(\Lambda)]_{\overline{x}} \to [v_{n*}(\Lambda)]_{\overline{y}}$ where 
             $\overline{y}$ is a geometric point over $y$ that generalizes $\overline{x}$. 
             Since $y \in V_n$, we get that $[v_{n*}(\Lambda)]_{\overline{y}} = \Lambda$ 
              and hence $[v_{n*}(\Lambda)]_{\overline{x}} = \Lambda$ as a set. 
              By Proposition 2.3.10 in loc.cit., 
              the topos associated to the site $\{x\}_{\text{ét}}$ can be identified with 
              the topos associated to $\mathrm{Spec}(k(x)^h)_{\text{ét}}$ where 
              $k(x)^h$ denotes the henselization of the field $k(x)$ with respect to $k(x)^+$. 
              Hence, we can identify sheaves on 
              $\{x\}_{\text{ét}}$ with $G := \mathrm{Gal}(k(x)^{h,\mathrm{sep}}/k(x)^h)$-modules 
              where the identification takes a sheaf on $\{x\}_{\text{ét}}$ and 
              sends it to the stalk at $\overline{x}$.
              Furthermore, we 
              see that for a sheaf $\sF$ on $x$, $H^0(x,\sF) = [\sF]_{\overline{x}}^G$. 
               Since $H^0(x,i_x^*v_{n*}(\Lambda)) = \Lambda$ and 
               $[v_{n*}(\Lambda)]_{\overline{x}} = \Lambda$,
               we get that $G$ acts trivially on $[v_{n*}(\Lambda)]_{\overline{x}}$. 
               This proves the claim and by equation (7), 
               we see that 
               $i_{x}^*(\sG_{|\overline{V_n}}) \simeq \Lambda$. 
               \end{proof} 
    

\section{Reflexivity} \label{sfjxoosdfsdfsd}

         Recall from \cite[Exposé XVIII, Théor\`{e}me 3.1.4]{sga4tome3} that if
  $g : Z' \to Z$ is a morphism of varieties over an algebraically closed field $L$ then there exists a functor 
  $g^! : \mathcal{D}^b(Z_{\text{ét}},\Lambda) \to \mathcal{D}^b(Z'_{\text{ét}},\Lambda)$ which is right adjoint to 
  the functor $Rg_! : \mathcal{D}^b(Z'_{\text{ét}},\Lambda) \to \mathcal{D}^b(Z_{\text{ét}},\Lambda)$. 
  Given an object $\sF \in \mathcal{D}^b(Z_{\text{ét}},\Lambda)$, 
  the Verdier dual of $\sF$ is defined to be 
   $D(\sF) :=  R\Homs(\sF,g^!(\Lambda))$ where $g : Z \to \mathrm{Spec}(L)$ is the structure map. 
   
   \begin{defi}
   \emph{Let $L$ be an algebraically closed field and let $Z$ be an $L$-variety. 
    An object $\sF \in \mathcal{D}^b(Z_{\text{ét}},\Lambda)$ is said to be} reflexive \emph{if the canonical map
    \begin{align*} 
      \sF  \to  D \circ D(\sF)
    \end{align*}    
    is an isomorphism. 
    Likewise, if $X$ is a fine $k$-adic space and 
    $\sF \in \mathcal{D}^b(X^f_{\text{ét}},\Lambda)$ is said to be} reflexive \emph{if the canonical map
    \begin{align*} 
      \sF  \to  D^{\mathrm{ad}} \circ D^{\mathrm{ad}}(\sF)
    \end{align*}
    is an isomorphism.} 
\end{defi}
  
     Our results from \S 4 let us connect reflexive sheaves on the fine étale site of a fine $k$-adic space 
   $X$ and reflexive sheaves on the special fibre of any of its formal models via the nearby cycles functor. 
   The precise statement is the following. 
   
\begin{thm} \label{reflexive on generic and special}
 Let $X$ be a fine $k$-adic space. Let $\sF$ be a complex in $\mathcal{D}^b(X^f_{\text{ét}},\Lambda)$. 
 \begin{enumerate}
 \item If $\sF$ is reflexive then for any 
 étale morphism $U \to X$ of fine $k$-adic spaces and any formal model 
 $\mathfrak{U}$ of $U$, 
 $R\psi_{\mathfrak{U}}(\sF_{|U})$ is reflexive. 
 \item Suppose for every 
 étale morphism $U \to X$ of fine $k$-adic spaces, there exists a formal model 
 $\mathfrak{U}$ of $U$ such that 
 $R\psi_\mathfrak{U}(\sF_{|U})$ is reflexive then
 $\sF$ is reflexive. 
 \end{enumerate} 
\end{thm}  
  \begin{proof}
    We begin by proving part (1). 
    Suppose $\sF$ is reflexive and $U \to X$ is an étale morphism of fine $k$-adic spaces. 
    Let $\mathfrak{U}$ be a formal model of $U$. 
    We must verify that $R\psi_{\mathfrak{U}}(\sF_{|U})$ is reflexive on $\mathfrak{U}_{s,\text{ét}}$ i.e.
    $$R\psi_{\mathfrak{U}}(\sF_{|U}) \simeq D \circ D \circ R\psi_{\mathfrak{U}}(\sF_{|U}).$$
    By part (2) in Theorem \ref{big result small proof}, 
    we see that 
    $$D \circ D \circ R\psi_{\mathfrak{U}}(\sF_{|U}) \simeq R\psi_{\mathfrak{U}} \circ D^{\mathrm{ad}} \circ D^{\mathrm{ad}}(\sF_{|U}).$$
   By part (2) of Lemma \ref{projection formula adic spaces}, 
     $$D^{\mathrm{ad}} \circ D^{\mathrm{ad}}(\sF_{|U}) \simeq [D^{\mathrm{ad}} \circ D^{\mathrm{ad}}(\sF)]_{|U}.$$
  Since $\sF$ is reflexive, we get that 
  $\sF \simeq D^{\mathrm{ad}} \circ D^{\mathrm{ad}}(\sF)$ which completes the proof of (1). 
  
   We now prove part (2).
   We must show that $\sF$ is reflexive.
    It suffices to prove that if   
   $U \to X$ is an étale morphism of fine $k$-adic spaces then 
   $$R\Gamma(U,\sF) \simeq R\Gamma(U, D^{\mathrm{ad}} \circ D^{\mathrm{ad}}(\sF)).$$
   Let $\mathfrak{U}$ be a formal model of $U$ such that $R\psi_{\mathfrak{U}}(\sF_{|U})$ is 
   reflexive. 
   Observe that we have the following chain of isomorphisms. 
   \begin{align*} 
  R\Gamma(U,D^{\mathrm{ad}} \circ D^{\mathrm{ad}}(\sF)) &\overset{(i)}{\simeq} R\Gamma(U,D^{\mathrm{ad}} \circ D^{\mathrm{ad}}(\sF_{|U})) \\
   &\overset{(ii)}{\simeq} R\Gamma(\mathfrak{U}_s,R\psi_{\mathfrak{U}} \circ D^{\mathrm{ad}} \circ D^{\mathrm{ad}}(\sF_{|U})) \\
   &\overset{(iii)}{\simeq} R\Gamma(\mathfrak{U}_s, D \circ D \circ R\psi_{\mathfrak{U}}(\sF_{|U})) \\
   &\overset{(iv)}{\simeq} R\Gamma(\mathfrak{U}_s, R\psi_{\mathfrak{U}}(\sF_{|U})) \\ 
   &\overset{(v)}{\simeq} R\Gamma(U,\sF_{|U}). 
   \end{align*} 
    The isomorphism (i) follows from part (2) of Lemma \ref{projection formula adic spaces}, (ii) is because of the composition of derived functors, 
    (iii) is a direct consequence of part (2) in Theorem \ref{big result small proof}, (iv) is because by assumption 
    $R\psi_{\mathfrak{U}}(\sF_{|U})$ is reflexive and finally (v) is simply by the composition of derived functors.
    \end{proof}

  Our goal is to classify the reflexive objects in $\mathcal{D}^b(X^f_{\text{ét}},\Lambda)$ where 
  $X$ is fine $k$-adic. Theorem \ref{reflexive on generic and special} suggests that 
  it suffices to classify reflexive objects on the special fibre of a model of $X$.   
  To this end, we propose the following conjecture concerning 
  reflexive objects on $\mathcal{D}^b(Z_{\text{ét}}, \Lambda)$ where 
  $Z$ is a variety over an algebraically closed field. 
   
\begin{con} \label{reflexivity conjecture} \label{classcons1}
Let $L$ be an algebraically closed field and let 
$Z$ be an $L$-variety.  
An object $\sF \in \mathcal{D}^b(Z_{\text{ét}}, \Lambda)$ is reflexive if and only if its cohomology sheaves are constructible i.e.
$\sF \in \mathcal{D}^b_c(Z_{\text{ét}},\Lambda)$. 
\end{con} 

\begin{rem}
We will see from the proof of Lemma \ref{classcons}, that Conjecture \ref{reflexivity conjecture} is true if it is true for $Z = \mathbb{A}^n_L$.
\end{rem}

  \begin{thm} \label{biduality theorem}
  Let $X$ be a fine $k$-adic space and 
  let $p_X : X \to k$ denote the structure morphism. 
  Let $\sF \in \mathcal{D}^b(X^f_{\text{ét}}, \Lambda)$.
  If $\sF$ is semi-constructible then it is 
  reflexive. The converse is true if we suppose that 
   Conjecture \ref{reflexivity conjecture} is true. 
   \end{thm}    
\begin{proof} 
  Let $\sF$ be a semi-constructible complex. Theorem \ref{big result small proof}(3) shows that $\sF$ is reflexive. 
 Suppose $\sF \in \mathcal{D}^b(X^f_{\text{ét}},\Lambda)$ is reflexive. 
   To show that $\sF$ is semi-constructible, we must show that for every étale
   $f : U \to X$ and every formal model 
   $\mathfrak{U}$ of $U$, $R\psi_{\mathfrak{U}}(f^{*}\sF) \in \mathcal{D}^b_c(\mathfrak{U}_s,\Lambda)$.
	The following computation shows that $f^{*}\sF$ is reflexive:
	\begin{align*}
	R\Homs(R\mathrm{Hom}(f^{*}\sF, p_{U}^{!}(\Lambda)),p_{U}^{!}(\Lambda)) &\overset{(i)}{=} R\Homs(f^{!}R\Homs(\sF, p_{X}^{!}(\Lambda)),p_{U}^{!}(\Lambda)) \\
	&\overset{(ii)}{=} R\Homs(f^{*}R\Homs(\sF, p_{X}^{!}(\Lambda)),p_{U}^{!}(\Lambda)) \\
	&\overset{(iii)}= f^{!}R\Homs(R\Homs(\sF, p_{X}^{!}(\Lambda)),p_{X}^{!}(\Lambda)) \\
	&\overset{(iv)}= f^{*}R\Homs(R\Homs(\sF, p_{X}^{!}(\Lambda)),p_{X}^{!}(\Lambda)) \\
	&= f^{*}\sF
	\end{align*}
	where (i) and (iii) follow from Lemma \ref{projection formula adic spaces}(2); and (ii) and (iv) follow from the fact that $f$ is étale.
	
	By Theorem \ref{reflexive on generic and special}, $R\psi_{\mathfrak{U}}(f^{*}\sF)$ is reflexive.
        The given hypothesis 
	implies $R\psi_{\mathfrak{U}}(f^{*}\sF)$ is constructible and hence $\sF$ is semi-constructible as promised.   
  \end{proof} 
 
 \subsection{Reflexivity for varieties} 
 
  For the rest of this section we fix a field $k$ which is algebraically closed. Let $\ell$ denote a prime number different from the characteristic of $k$ and set
 $\Lambda := \mathbb{Z}/\ell\mathbb{Z}$ for some $n \in \mathbb{N}$. 
 
   Our goal is to study Conjecture \ref{classcons1} which provides a classification of the class of reflexive étale sheaves 
 on a variety $X$ over $k$. In what follows, we show that the conjecture holds true at least in certain specific instances. 
  
   \begin{lem} \label{classcons}
 The following statements are equivalent. 
\begin{enumerate} 
\item Let $X$ be a $k$-variety. An étale sheaf $\sF$ of $\Lambda$-modules on $X$ is reflexive if and only if it is constructible. 
\item  
\begin{enumerate}   
\item   Let $X$ be a $k$-variety. Let $\sF$ be a reflexive étale sheaf of $\Lambda$-modules. 
                         Suppose we have a $k$-morphism $f : X \to \mathbb{A}^1_k$. Let $\eta$ denote the 
                         generic point of $\mathbb{A}^1_k$ and $X_{\bar{\eta}} := X \times_{\mathbb{A}^1_k} \bar{\eta}$.  
                         Then the sheaf $f_{\eta}^*(\sF)$ is reflexive on $X_{\bar{\eta}}$. \\ 
\item Let $X$ be a $k$-variety. Let  $\sF$ be a reflexive étale sheaf
          of $\Lambda$-modules on $X$. 
                         Suppose $\sF$ is the direct sum of skyscraper sheaves. Then $\sF$ 
                         is constructible. 
   \end{enumerate} 
   \end{enumerate} 
    \end{lem} 
 \begin{proof} 
 It suffices to verify (a) + (b) implies (1). 
For an open immersion $f: U \to X$ we have that
\begin{align*}
	R\Homs(R\Homs(f^{*}\sF, p_{U}^{!}(\Lambda)),p_{U}^{!}(\Lambda)) &\overset{(i)}{=} R\Homs(f^{!}R\Homs(\sF, p_{X}^{!}(\Lambda)),p_{U}^{!}(\Lambda)) \\
	&\overset{(ii)}{=} R\Homs(f^{*}R\Homs(\sF, p_{X}^{!}(\Lambda)),p_{U}^{!}(\Lambda)) \\
	&\overset{(iii)}= f^{!}R\Homs(R\Homs(\sF, p_{X}^{!}(\Lambda)),p_{X}^{!}(\Lambda)) \\
	&\overset{(iv)}= f^{*}R\Homs(R\Homs(\sF, p_{X}^{!}(\Lambda)),p_{X}^{!}(\Lambda)) \\
	\end{align*}
	where (i) and (iii) follow from Lemma \ref{projection formula varieties}(2); and (ii) and (iv) follow from the fact that $f$ is étale.
	Thus we see that the property of being reflexive is local for the Zariski topology on $X$. Hence we can assume that $X$ is a connected affine $k$-variety. 
	We proceed by induction on $\mathrm{dim}(X)$. 
	 If $\mathrm{dim}(X) = 0$, the site $X_{\text{ét}}$ 
	 is equivalent to $\Spec(A)_{\text{ét}}$ where $A$ is a finite $k$-algebra. 
	 Note that $\Spec(A)$ is the disjoint union of a finite number of points each of which is open 
	 in $X$ for the Zariski topology. 
	 Since reflexivity is preserved for pullbacks along open immersions, we can reduce to the case that $X$ is a single point. Since 
	 $\Spec(A)_{\text{ét}}$ is equivalent to $\Spec(A_{\text{red}})_{\text{ét}}$, we can assume that $A$ is a field. 
	 The case $X = \mathrm{Spec}(k)$ is well known. 
	 
Let us assume the following. For every algebraically closed field extension $K$ of $k$ and 
every $K$-variety $Y$ of dimension strictly less than $\mathrm{dim}(X)$, 
a reflexive sheaf $\mathscr{G}$ on $Y$ is constructible. 

  By construction, there exists an immersion $h : X \to \mathbb{A}_k^{n}$ for some $n \in \mathbb{N}$. 
  Let $p_i : \mathbb{A}^{N}_k \to \mathbb{A}_k^1$ be the projection onto the $i$-th coordinate
  and $p'_i := p_i \circ h$. Let $\eta$ denote the generic point of 
  $\mathbb{A}^1_k$. 
  By (a), we have that $\sF_{| X \times_{\mathbb{A}^1_k} \overline{k(\eta)}}$ is reflexive.
	
    Then by our induction hypothesis, we have 
  that $\sF_{| X \times_{\mathbb{A}^1} \overline{k(\eta)}}$ is constructible. 
  It follows from \cite[Lemma 3.5]{SGA4.5}, that there exists 
  a constructible sheaf $\mathscr{G} \subset \sF$ such that the local sections of 
  $\mathscr{H} := \sF/\mathscr{G}$ have finite support. 
  By Lemma \ref{classification}, $\mathscr{H}$ is the direct sum of skyscrapers sheaves. 
  Since $\mathscr{G}$ is constructible, it is reflexive. It follows that $\mathscr{H}$ is reflexive.
 By (b),
 $\mathscr{H}$ is constructible, from which it can be deduced that 
 $\sF$ is constructible. 
 \end{proof} 
 
\begin{rem} 
  \emph{In the section that follows, we provide a proof of Statement (b), cf. Corollary \ref{cor:trsfinindust}} 
\end{rem}
  
  \begin{lem} \label{projection formula varieties} 
  Let $f : X \to Y$ be a morphism of $k$-varieties. Let 
  $p_X : X \to \mathrm{Spec}(k)$ and $p_Y : Y \to \mathrm{Spec}(k)$ be the structure 
  maps of $X$ and $Y$. 
  Let $A,B \in \mathcal{D}^b(Y_{\text{ét}},\Lambda)$.
  \begin{enumerate} 
  \item  $f^!R\Homs(A,B) = R\Homs(f^*(A),f^!(B))$. 
  \item  $f^!R\Homs(A,p_Y^!(\Lambda)) = R\Homs(f^*(A),p_X^!(\Lambda))$. 
  \end{enumerate} 
     \end{lem} 
  \begin{proof} 
 \begin{enumerate}
 \item   It suffices to show that for any $C \in \mathcal{D}^b(X_{\text{ét}},\Lambda)$.
 \begin{align*} 
  \mathrm{Hom}(C,f^!R\Homs(A,B)) = \mathrm{Hom}(C,R\Homs(f^*(A),f^!(B))). 
 \end{align*} 
   Let $C \in D^b(X_{\text{ét}},\Lambda)$. 
  \begin{align*} 
    \mathrm{Hom}(C,f^!R\Homs(A,B)) &= \mathrm{Hom}(Rf_!(C), R\Homs(A,B)) \\ 
                                         &= \mathrm{Hom}(Rf_!(C) \otimes^L A, B). \\ 
  \end{align*}       
  By \cite[Exposé XVII, Proposition 5.2.9]{sga4tome3}, we have 
    \begin{align*}
    \mathrm{Hom}(Rf_!(C) \otimes^L A,B) &= \mathrm{Hom}(Rf_!(C \otimes^L f^*(A)),B) \\ 
                                                 &= \mathrm{Hom}(C,R\Homs(f^*(A),f^!(B))).
  \end{align*} 
  \item  Part (2) follows directly from part (1). 
  \end{enumerate} 
    \end{proof}

\subsection{Calculation for skyscraper sheaves}

 In this section we prove Statement (B) from Lemma \ref{classcons}. 
  For the rest of this section we write  
$$\sF = \bigoplus_{x \in X(k)} i_{x*}M_x$$
where $X(k)$ is the set of closed points of $X$, $i_x \colon x \hookrightarrow X$ is the inclusion and 
$M_x$ is a $\Lambda$-vector space. We begin with a general lemma.

\begin{lem} \label{lem:dire}
Suppose $\sG = \oplus_{i \in I} \sG_i$ is a reflexive sheaf on $X$, where $I$ is any indexing set. Then $\sG_i$ is also reflexive.
\end{lem}

\begin{proof}
This is a purely category theoretic result. For $\mathscr{F} \in \mathcal{D}^b(X_{\text{ét}},\Lambda)$,
 we denote by $\delta_{\mathscr{F}} \colon \mathscr{F} \to DD\mathscr{F}$ the canonical morphism. Since
  we are working in a triangulated category, if $\delta_{\mathscr{F}}$ is a monomorphism and an
   epimorphism then it is an isomorphism. Let $i \colon \mathscr{G}_i \to \mathscr{G}$ and
    $\pi \colon \mathscr{G} \to \mathscr{G}_i$ be the canonical morphisms. Note that 
    $\pi \circ i = \id_{\mathscr{G}_i}$. Thus $i$ is a monomorphism and $\pi$ is an epimorphism.

We have the commutative diagram
$$
\begin{tikzcd} [row sep = large, column sep = large] 
\mathscr{G}_i \arrow[r, "\delta_{\mathscr{G}_i}"] \arrow[d, "i"] &
DD\mathscr{G}_i  \arrow[d, "DDi"] \\
\mathscr{G} \arrow[r, "\delta_{\mathscr{G}}"] &
DD\mathscr{G}  
\end{tikzcd}
$$ 
where by assumption $\delta_{\mathscr{G}}$ is an isomorphism. Thus $\delta_{\mathscr{G}_i}$ is a monomorphism. Similarly we have the commutative diagram
$$
\begin{tikzcd} [row sep = large, column sep = large] 
\mathscr{G} \arrow[r, "\delta_{\mathscr{G}}"] \arrow[d, "\pi"] &
DD\mathscr{G}  \arrow[d, "DD\pi"] \\
\mathscr{G}_i \arrow[r, "\delta_{\mathscr{G}_i}"] &
DD\mathscr{G}_i  
\end{tikzcd}
$$ 
where $DD\pi$ satisfies $DD\pi \circ DDi = \id_{DD\mathscr{G}_i}$ and so $DD\pi$ is an epimorphism. Hence $\delta_{\mathscr{G}_i}$ is an epimorphism. We have shown that $\delta_{\mathscr{G}_i}$ is both a monomorphism and an epimorphism and hence it is an isomorphism.

\end{proof}



\begin{lem} \label{lem:vanext}
For any étale sheaf $\sG$ of $\Lambda$-modules, 
$\Exts^{i}(\sG, p_X^{!}\Lambda) = 0$ for $i >0$. 
\end{lem}

\begin{proof}
It is enough to prove it for group Ext, that is $\mathrm{Ext}^i(\sG, p_X^{!}\Lambda) = 0$. In 
this case, the result follows by duality i.e. $\mathrm{Ext}^i(\sG|_U, p_U^{!}\Lambda) = \mathrm{Hom}(H^{-i}_c(U,\sG|_U), \Lambda)$. 
\end{proof}


\begin{lem} \label{lem:skrefc}
Suppose $k$ is countable. Recall that we assumed $\sF := \bigoplus_{x \in X(k)} i_{x*}M_x$. We will suppose in addition that for every $x \in X(k)$, 
$M_x$ is finite dimensional. 
We have that $R\Homs(\sF, p_X^!\Lambda) = \prod_{x \in X(k)} i_{x*}M_x^{\wedge}$ where $M_x^{\wedge} := \mathrm{Hom}_{\Lambda}(M_x, \Lambda)$ is the dual vector space.  
\end{lem} 

\begin{proof}
Denote by $p_x = p_X \circ i_x$. We compute
\begin{align*}
R\Homs(\sF, p_X^!\Lambda) &= R\Homs(\bigoplus_{x \in X(k)} i_{x*}M_x, p_X^!\Lambda)\\
&= R\prod_{x \in X(k)} R\Homs(i_{x*}M_x, p_X^!\Lambda) \\
&= R\prod_{x \in X(k)} i_{x*}R\Homs(M_x, p_x^!\Lambda) \\
&= R\prod_{x \in X(k)} i_{x*}M_x^{\wedge}.
\end{align*}
Since $k$ is countable, $X(k)$ is a countable set.
The Mittag-Leffler criterion, cf. \cite[Tag 0940, Lemma 21.22.5]{stacks-project} now implies that 
$R\prod_{x \in X(k)} i_{x*}M_x^{\wedge} = \prod_{x \in X(k)} i_{x*}M_x^{\wedge}$.

\end{proof}

\begin{prop} \label{reflexive skyscraper implies constructible}
Let $X$ be a variety over the algebraically closed countable field $k$. 
Let $\sF = \bigoplus_{x  \in X(k)} i_{x*}M_x$ be an étale sheaf on $X$.   
If $\sF$ is reflexive then $\sF$ is constructible. 
\end{prop}

\begin{proof}
By Lemma \ref{lem:dire}, for every $x \in X(k)$, $M_x$ is a finite dimensional $\Lambda$-vector space. By Lemma \ref{lem:skrefc}, we have 
\[
R\Homs(\sF, p_X^!\Lambda) = \prod_{x \in X(k)} i_{x*}M_x^{\wedge}
\]
(where $M_x^{\wedge}$ is the dual vector space) and since $\sF$ is reflexive the canonical morphism
\begin{equation} \label{eq:clethisup}
\sF \to R\Homs \left(\prod_{x \in X(k)} i_{x*}M_x^{\wedge}, p_X^!\Lambda \right)
\end{equation}
is an isomorphism. Taking $H^0$ of \eqref{eq:clethisup} gives an isomorphism
\[
\sF \to \Homs \left(\prod_{x \in X(k)} i_{x*}M_x^{\wedge}, p_X^!\Lambda \right).
\]
Let $\mathscr{H}$ be the quotient of the canonical morphism
\[
\oplus_{x \in X(k)} i_{x*}M_x^{\wedge} \hookrightarrow \prod_{x \in X(k)} i_{x*}M_x^{\wedge}.
\]
By taking $R\Homs (-, p_X^{!}\Lambda)$, the long exact sequence associated to 
\[
0 \to \oplus_{x \in X(k)} i_{x*}M_x^{\wedge} \hookrightarrow \prod_{x \in X(k)} i_{x*}M_x^{\wedge} \to \mathscr{H} \to 0
\]
is of the form
\[
\Homs\left(\prod_{x \in X(k)} i_{x*}M_x^{\wedge}, p_X^!\Lambda\right) \to \Homs(\oplus_{x \in X(k)} i_{x*}M_x^{\wedge}, p_X^!\Lambda) \to \Exts^{1}(\mathscr{H}, p_X^{!}\Lambda) \to \cdots
\]
Noting that $\Exts^{1}(\mathscr{H}, p_X^{!}\Lambda) = 0$ (cf. Lemma \ref{lem:vanext}) shows that the morphism 
$$\Homs\left(\prod_{x \in X(k)} i_{x*}M_x^{\wedge}, p_X^!\Lambda\right) \to \Homs(\oplus_{x \in X(k)} i_{x*}M_x^{\wedge}, p_X^!\Lambda)$$ 
is surjective.

Summarizing we have a commutative diagram of sheaves
$$
\begin{tikzcd} [row sep = large, column sep = large] 
\oplus_{x \in X(k)} i_{x*}M_x \arrow[r, "\sim"] \arrow[d] &
\Homs(\prod_{x \in X(k)} i_{x*}M_x^{\wedge}, p_X^!\Lambda)  \arrow[d, twoheadrightarrow] \\
\prod_{x \in X(k)} i_{x*}M_x \arrow[r, "\sim"] &
\Homs(\oplus_{x \in X(k)} i_{x*}M_x^{\wedge}, p_X^!\Lambda)  
\end{tikzcd}
$$ 
where the bottom arrow is an isomorphism by Lemma \ref{lem:skrefc}. It follows that the left hand vertical arrow is surjective and so there are only finitely many $x \in X(k)$ for which $M_x \neq 0$. Thus $\sF$ is constructible.
\end{proof}

\begin{rem} \label{rem:inficounsh}
Both Lemma \ref{lem:skrefc} and Proposition \ref{reflexive skyscraper implies constructible} hold true under the more general assumption (without assuming $k$ is countable) that $I := \{ x \in X(k) \text{ } \lvert M_x \not= 0 \}$ is countable.
\end{rem}

We now come to the main result of this section.

\begin{cor} \label{cor:trsfinindust}
Let $X$ be a variety over the algebraically closed field $k$. 
Let $\sF = \bigoplus_{x  \in X(k)} i_{x*}M_x$ be an étale sheaf on $X$.   
If $\sF$ is reflexive then $\sF$ is constructible.
\end{cor}

\begin{proof}
We set $I := \{ x \in X(k) \text{ } \lvert M_x \not= 0 \}$. Suppose for the sake of contradiction that $I$ is infinite. Then $I$ contains an infinite countable subset $J \subset I$ and by Lemma \ref{lem:dire} we see that $\mathscr{G} := \bigoplus_{x  \in J} i_{x*}M_x$ is reflexive. This contradicts Remark \ref{rem:inficounsh}.
\end{proof}




We close this section with a classification result.

\begin{lem} \label{classification} 
If $\sG$ is an étale sheaf of $\Lambda$-modules on $X$ whose local sections have finite
 support  (cf. \cite[Tag 04FQ, Definition 53.31.3]{stacks-project}), then $\sG$ is a direct sum of skyscraper sheaves supported on closed points. 
\end{lem}

\begin{proof}
We begin by making the following claim. 
Since the sheaf $\sG$ has all of its local sections having finite support, the support of $\sG$ is only on closed points.
Indeed, suppose there exists $x \in X$ such that 
$\sG_{\overline{x}} \neq 0$ and $x$ is not a closed point. 
It follows by definition that there exists an étale neighbourhood $f : U \to X$ of $\overline{x}$ 
and $s \in \sG(U)$ such that the image of $s$ in $\sG_{\overline{x}}$ is not zero. 
Observe that $\mathrm{supp}(s) \subset U$ is a Zariski closed subset. 
Let $x'$ be an element in $U$ such that $x' \mapsto x$. 
It follows by assumption on $x$ that $\overline{\{x'\}} \subset \mathrm{supp}(s)$ 
where $\overline{\{x'\}}$ is the Zariski closure of $x'$. 
Observe that $\overline{\{x'\}}$ is an algebraic variety whose dimension is strictly greater than zero. 
Since the field $k$ is algebraically closed, it is in particular infinite. Hence, $\overline{\{x'\}}$ contains an infinite number of points.
This implies that $\mathrm{supp}(s)$ contains an infinite number of points. 

 Consider the morphism
\begin{align*}
\sG \to \bigoplus_{x \in X(k)} i_{x*}\sG_{\overline{x}}
\end{align*}
which for an étale morphism $V \to X$, $s \in \sG(V)$ gets sent to $(s_{\overline{x}})_{x \in X(k)}$. This is a well-defined morphism because $s_{\overline{x}}$ is zero except for a finite number of points $x$. It is clearly an isomorphism (checking at the level of stalks for closed points, noting that taking stalks at closed points commutes with colimits). 
\end{proof}

\subsection{Calculation for direct sum of constructible sheaves}

    Let $X$ be a smooth affine curve over $k$ and suppose that $k$ is of positive characteristic (we will return to the characteristic zero case later, the proof differing to the positive characteristic case). We show that Conjecture \ref{reflexivity conjecture} above is true in a
    particular case. 

In this section we assume that $\sF$ is a direct sum of local systems. For the rest of this section we write
\[
\sF = \bigoplus_{i \in \mathbb{N}} \sL_i
\]
where $\sL_i$ are local systems corresponding to (finite) dimensional $\mathbb{F}_l$-representations $V_i$ of $\pi_1(X)$.
 We assume that the index set is the positive integers $\mathbb{N}$. We denote by $\sL_i^\wedge := \Homs(\sL_i, \Lambda)$. The spectral sequence
\[
\Exts^i(R^{-j}\prod_i \sL_i^\wedge, \Lambda) \implies \Exts^{i+j}(R\prod \sL_i^{\wedge}, \Lambda) 
\]
degenerates at the 2nd page, cf. \cite[Tag 07A9]{stacks-project}. The possible non-zero terms on the 2nd page are

$$
\begin{tikzcd} [row sep = large, column sep = large] 
\Homs(\prod_i \sL_i^\wedge,\Lambda) \arrow[rrd, "\alpha"] &
\Exts^1(\prod_i \sL_i^\wedge,\Lambda) \arrow[rrd] &
\Exts^2(\prod_i \sL_i^\wedge,\Lambda) &
0 \\
\Homs(R^1\prod \sL_i^\wedge,\Lambda) &
\Exts^1(R^1\prod \sL_i^\wedge,\Lambda) &
\Exts^2(R^1\prod \sL_i^\wedge,\Lambda) &
0 
\end{tikzcd}
$$
Assuming $\sF$ is reflexive, we get that 
$$\Exts^1(\prod_i \sL_i^\wedge,\Lambda) = \Exts^2(\prod_i \sL_i^\wedge,\Lambda) = 0$$ and
\begin{equation} \label{eq:derpro}
\Homs(R^1\prod \sL_i^\wedge,\Lambda) = 0.
\end{equation}
Moreover the morphism $\alpha$ must be surjective. Finally we have a short exact sequence
$$0 \to \Exts^1(R^1\prod \sL_i^\wedge,\Lambda) \to \bigoplus_i \sL_i \to \ker \alpha \to 0.$$

\begin{prop} \label{the case of the direct sum}  
Let $X$ be a smooth affine curve. For every $i \in \mathbb{N}$, let $\sL_i$ be a finite local system on $X$ such that 
$\oplus_i \sL_i = \prod_i \sL_i$. We suppose in addition that 
$\oplus_i \sL^\wedge_i = \prod_i \sL^\wedge_i$. Suppose, $\oplus_i \sL_i$ is a reflexive étale sheaf. Then we must have that for all but finitely 
many $i \in \mathbb{N}$, $\sL_i = 0$. 
In particular, $\oplus_i \sL_i$ is constructible. 
\end{prop} 
\begin{proof} 
 Our discussion concerning the spectral sequence above shows that 
 the reflexivity of $\oplus_i \sL_i$ implies that 
\begin{equation} \label{eq:stuasshisd}
\Exts^1(\prod_i \sL_i^\wedge,\Lambda) = 0.
\end{equation} 
  We claim that $\Exts^1(\prod_i \sL_i^\wedge,\Lambda) = R^1\prod_i \sL_i$. Indeed, 
 \begin{align*} 
   \Exts^1(\prod_i \sL_i^\wedge,\Lambda) = H^1(R\Homs(\prod_i \sL_i^{\wedge},\Lambda)). 
 \end{align*} 
     By assumption, $\prod_i \sL_i^\wedge = \oplus_i \sL_i^{\wedge}$ and hence we see that 
    $R\Homs(\prod_i \sL_i^{\wedge},\Lambda) = R\prod_i \sL_i$. Therefore $\Exts^1(\prod_i \sL_i^\wedge,\Lambda) = R^1\prod_i \sL_i$.  
   Equation \eqref{eq:stuasshisd} implies that 
   $R^1\prod_i \sL_i = 0$. By 
   \cite[Tag 0940, Lemma 21.22.2]{stacks-project}, we see that 
   for every $U \to X$ étale, 
   $H^1(U,\prod_i \sL_i) = \prod_i H^1(U,\sL_i)$. 
   However, by assumption, $\prod_i \sL_i = \oplus_i \sL_i$ and
   $H^1(U,\oplus_i \sL_i) = \oplus_i H^1(U,\sL_i)$.   
      Thus we get that for every $U \to X$ étale, 
      \begin{align*} 
       \bigoplus_i H^1(U,\sL_i) = \prod_i H^1(U,\sL_i).
      \end{align*} 
     This must mean that for all but a finite number of $i \in \mathbb{N}$, we get 
     $H^1(U,\sL_i) = 0$. 
     
     We claim that this is not possible if there are infinitely many $i$ such that $\sL_i \not= 0$. We argue by contradiction. Indeed, since the canonical morphism
\[
\bigoplus_i \sL_i \to \prod_i \sL_i
\]
is an isomorphism, for all open immersions $U \to X$ there exists a 
positive integer $N(U)$ such that for all $i > N(U)$, $H^0(U, \sL_i\lvert_U) = 0$. 

Let $\overline{X}$ be the smooth projective curve which contains $X$ as an open sub-scheme. 
Let $\overline{X} \backslash X = \left\{x_1, \ldots x_j \right\}$. The Grothendieck-Ogg-Shafarevich formula, cf. \cite[Exposé X, Théor\`{e}me 7.1]{sga5} says that for every open immersion $V \hookrightarrow X$ (note that since $V$ is affine, $H^{k}(V, \sL_i\lvert_V) = 0$ for $k > 1$) 

\begin{equation} \label{eq:groosh}
\dim H^{0}(V, \sL_i\lvert_V) - \dim H^{1}(V, \sL_i\lvert_V) = 
\dim \sL_{i, \overline{\eta}}\cdot(2-\lvert \overline{X} \backslash V \rvert - 2g(\overline{X})) - \sum_{1 \leq m \leq j}\mathrm{Swan}_{x_m}(\sL_i).
\end{equation}

Since $H^0(V, \sL_i\lvert_V) = 0$ for all $i > N(V)$, \eqref{eq:groosh} shows that the dimension of $H^1(V, \sL_i\lvert_V)$ is at least
 $\lvert \overline{X} \backslash V \rvert - 2$  provided $\sL_i \not= 0$. 
 Hence we get a contradiction.   
\end{proof} 

We immediately obtain a stronger result in the case $X = \mathbb{A}^1_k$.

\begin{cor} \label{cor:5.17god}
Let $X = \mathbb{A}^1_k$. For every $i \in \mathbb{N}$, let $\sL_i$ be a finite local system on $X$ such that 
$\oplus_i \sL_i = \prod_i \sL_i$. Suppose, $\oplus_i \sL_i$ is a reflexive étale sheaf. Then we must have that for all but finitely 
many $i \in \mathbb{N}$, $\sL_i = 0$. 
In particular, $\oplus_i \sL_i$ is constructible. 
\end{cor}

\begin{proof}
By Proposition \ref{the case of the direct sum}, it suffices to show that $\oplus_i \sL_i^{\wedge} = \prod_i \sL_i^{\wedge}$. Since the prime-to-$p$ part of $\pi_1(\mathbb{A}^1_k)$ is trivial, this follows from Lemma \ref{non empty implies dual non empty}. 
\end{proof}

We finish this section by showing that the condition imposed in Proposition \ref{the case of the direct sum} is not empty. That is the existence of a countably infinite family $(\sL_i)_{i \in \mathbb{N}}$ of non-zero local systems 
on a $k$-variety such that 
$\oplus_{i \in \mathbb{N}} \sL_i \simeq \prod_{i \in \mathbb{N}} \sL_i$ and 
$\oplus_{i \in \mathbb{N}} \sL_i^{\wedge} \simeq \prod_{i \in \mathbb{N}} \sL_i^{\wedge}$.

\begin{lem} \label{non empty implies dual non empty}
  Let $X$ be a $k$-variety and let $\sL$ be a local system of $\Lambda$-modules on $X$.
  Suppose $\sL$ corresponds to a $\pi_1(X)$-representation that factors through a finite group $G$ 
  whose order $|G|$ is coprime to the prime $\ell$. If $\sL(X) \neq 0$ then $\sL^{\wedge}(X) \neq 0$. 
\end{lem} 
\begin{proof} 
The local system $\sL$ corresponds to a finite 
$\pi_1(X)$-representation which we denote $V$ i.e. $V$ is a $\Lambda$-module
endowed with an action by $\pi_1(X)$. 
We have that 
$$\sL(X)  = V^{\pi_1(X)}.$$ 
Hence there exists a non-trivial sub-representation $W \subseteq V$ on which $\pi_1(X)$ acts trivially. 
By assumption, the action of $\pi_1(X)$ on $V$ factors through 
a finite quotient $G$ whose order $|G|$ is coprime to $\ell$. 
By Maschke's theorem, 
there exists $W' \subseteq V$ such that 
$V = W \oplus W'$ as $G$-representations. 
Hence we see that $V^{\wedge} = W^{\wedge} \oplus W'^{\wedge}$. 
Since $G$ acts trivially on $W$, we see that $W^{\wedge} \simeq W$ (as $G$-representations) and hence 
$W^{\wedge} \subseteq [V^{\wedge}]^G$. 
By definition, $[V^{\wedge}]^G = \sL^{\wedge}(X)$ and 
hence $\sL^{\wedge}(X)$ is not trivial. 
\end{proof} 

  The following example was inspired by a discussion with Pierre Deligne and shows that the hypothesis 
  in Proposition \ref{the case of the direct sum} is not empty. 

\begin{es}
\emph{We provide an example of a countably infinite family $(\sL_i)_{i \in \mathbb{N}}$ of non-zero local systems 
on a $k$-variety such that 
$\oplus_{i \in \mathbb{N}} \sL_i \simeq \prod_{i \in \mathbb{N}} \sL_i$ and 
$\oplus_{i \in \mathbb{N}} \sL_i^{\wedge} \simeq \prod_{i \in \mathbb{N}} \sL_i^{\wedge}$. 
Let $k$ be algebraically closed of characteristic $p$ such that $p$ divides $\ell - 1$. 
Let $\chi : \mathbb{F}_p \to (\mathbb{Z}/\ell\mathbb{Z})^*$ be a non-trivial character. 
Let $X := \mathbb{A}^1_k$. 
By \cite[Remark 5.8.6]{Sza}, $\mathrm{Hom}(\pi_1(X),\mathbb{F}_p)$ is 
infinite (here the morphisms are continuous group homomorphisms, with the discrete topology on $\mathbb{F}_p$). Let $\{\phi_i : \pi_1(X) \to \mathbb{F}_p\}_{i \in \mathbb{N}}$ be a family
of distinct elements of $\mathrm{Hom}(\pi_1(X),\mathbb{F}_p)$. 
 For every $i$, let $\sL_i$ be the rank 1 local system defined by the 
$(\mathbb{Z}/\ell\mathbb{Z})^*$ character given by the composition 
$\pi_1(X) \overset{\phi_i}{\longrightarrow} \mathbb{F}_p \to (\mathbb{Z}/\ell\mathbb{Z})^*$.
We claim 
$$\oplus_{i \in \mathbb{N}} \sL_i \simeq \prod_{i \in \mathbb{N}} \sL_i$$ and 
$$\oplus_{i \in \mathbb{N}} \sL_i^{\wedge} \simeq \prod_{i \in \mathbb{N}} \sL_i^{\wedge}.$$
By Lemma \ref{non empty implies dual non empty}, 
it suffices to verify that if $U \to X$ is an étale morphism of connected varieties then the set of $i \in \mathbb{N}$ such that   
$\sL_i(U) \neq 0$ is finite. 
Let $V_i$ denote the $\pi_1(X)$-representation corresponding 
to $\sL_i$. By definition, $\sL_i(U) = V_i^{\pi_1(U)}$. 
Since, $\sL_i$ is of rank one, we get that $V_i^{\pi_1(U)}$ is either $V_i$ or $0$. 
If $V_i^{\pi_1(U)} = V_i$ then $\pi_1(U)$ acts trivially on $V_i$ and hence by construction
$\pi_1(U)$ lies in the kernel of the map $\pi_1(X) \overset{\phi_i}{\longrightarrow} \mathbb{F}_p$.
Let $U_i \to X$ be the finite étale cover of $X$ that trivializes 
$\sL_i$ i.e. $\pi_1(U_i)$ is the kernel of the homomorphism $\phi_i : \pi_1(X) \to \mathbb{F}_p$. 
It follows that if $k(U)$, $k(U_i)$ and $k(X)$ denote the function fields of $U$, $U_i$ and $X$ then 
since $\pi_1(U), \pi_1(U_i)$ and $\pi_1(X)$ can be identified with quotients of 
$\mathrm{Gal}(k(X)^{\mathrm{sep}}/k(X))$, we get from the Galois correspondance that 
the field extension $k(X) \hookrightarrow k(U)$ factors through the extension $k(X) \hookrightarrow k(U_i)$ and an 
extension 
$k(U_i) \hookrightarrow k(U)$. 
As $U \to X$ is étale, the field extension $k(X) \hookrightarrow k(U)$ is finite and separable (in particular, by the primitive element theorem, it only has a finite number of intermediate fields). 
The family $\{U_i \to X\}_i$ is a countable set 
of distinct $p$-covers and hence the corresponding family of field extensions $\{k(X) \hookrightarrow k(U_i)\}$ are 
pairwise non-isomorphic. 
It follows that there can be only finitely many $i$ such that $k(U_i)$ embeds into $k(U)$. 
We have thus shown that there are only finitely many $i$ such that $\sL_i(U) \neq 0$.}
\end{es} 

Finally we specialize to the case $X = \mathbb{A}^1_k$, where the characteristic of $k$ is zero. In this case the finite local systems on $X$ are constant and we claim that 
\[
\bigoplus_{i \in \mathbb{N}} \Lambda
\]
cannot be reflexive.

\begin{lem} \label{lem:finpeuofsh}
In the situation of the preceding paragraph we have that $R^1 \prod \Lambda = 0$.
\end{lem}

\begin{proof}
By \cite[Tag 0940, Remark 21.22.4]{stacks-project} $R^1 \prod \Lambda = 0$ is the sheafification of the presheaf
\[
(U \xrightarrow{\text{étale}} X) \mapsto \prod_{\mathbb{N}} H^1(U, \Lambda).
\]
But $H^1(U, \Lambda)$ is a finite $\Lambda$-vector space, and hence the underlying set is finite say $\{s_1, s_2, \ldots, s_m \}$. Each $s_i$ corresponds to a $\Lambda$-torsor $V_i \to U$ which is trivialized by a covering $U_i \to U$. However the sections then vanish in 
$$H^1(U_1 \times_U U_2 \times_U \cdots \times_U U_m, \Lambda).$$
Hence the associated sheaf is zero.
\end{proof}

\begin{prop} \label{prop:5223}
Let $X = \mathbb{A}^1_k$ where the characteristic of $k$ is zero. Then \[
\bigoplus_{i \in \mathbb{N}} \Lambda
\]
is not reflexive.
\end{prop}

\begin{proof}
By the spectral sequence in the beginning of this section, if $\bigoplus_{i \in \mathbb{N}} \Lambda$ is reflexive then we must have by Lemma \ref{lem:finpeuofsh}:
\[
\Homs(\prod_{i \in \mathbb{N}} \Lambda, \Lambda) = \bigoplus_{i \in \mathbb{N}} \Lambda.
\]
Taking global sections gives
\begin{equation} \label{eq:aofns}
\textrm{Hom}_{\textrm{Sh}_{\text{ét}}(\mathbb{A}^1_k, \Lambda)}(\prod_{i \in \mathbb{N}} \Lambda, \Lambda) = \bigoplus_{i \in \mathbb{N}} \Lambda,
\end{equation}
where the $\textrm{Hom}$ on the LHS of \eqref{eq:aofns} is in the category of étale sheaves in $\Lambda$-vector spaces on $\mathbb{A}^1_k$. However 
\begin{equation} \label{eq:fjkspsf}
\textrm{Hom}_{\textrm{Sh}_{\text{ét}}(\mathbb{A}^1_k, \Lambda)}(\prod_{i \in \mathbb{N}} \Lambda, \Lambda) = \textrm{Hom}_{\Lambda}(\prod_{i \in \mathbb{N}} \Lambda, \Lambda)
\end{equation}
where the $\textrm{Hom}$ on the RHS of \eqref{eq:fjkspsf} is in the category of $\Lambda$-vector spaces. We explain why \eqref{eq:fjkspsf} is true. Let $\alpha \in \textrm{Hom}_{\textrm{Sh}_{\text{ét}}(\mathbb{A}^1_k, \Lambda)}(\prod_{i \in \mathbb{N}} \Lambda, \Lambda)$. For an étale morphism $U \to \mathbb{A}^1_k$ (with $U$ connected) one demands to have a commutative square (noting that global sections commute with products)
$$
\begin{tikzcd} [row sep = large, column sep = large] 
\prod_{i \in \mathbb{N}} \Lambda(\mathbb{A}^1_k) \arrow[r, "\alpha_{\mathbb{A}^1_k}"] \arrow[d, "\id"] &
\Lambda (\mathbb{A}^1_k) \arrow[d, "\id"] \\
\prod_{i \in \mathbb{N}} \Lambda(U) \arrow[r, "\alpha_U"] &
\Lambda(U)  
\end{tikzcd}
$$
and we see that the morphism on global sections (i.e. $\alpha_{\mathbb{A}^1_k}$) determines the morphism on local sections (i.e. $\alpha_U$). Conversely $\alpha_{\mathbb{A}^1_k}$ is a morphism of vector spaces $\prod_{i \in \mathbb{N}} \Lambda \to \Lambda$ and clearly any such morphism gives rise to a unique morphism of sheaves.

Now \eqref{eq:aofns} and \eqref{eq:fjkspsf} imply that
\[
\textrm{Hom}_{\Lambda}(\prod_{i \in \mathbb{N}} \Lambda, \Lambda) = \bigoplus_{i \in \mathbb{N}} \Lambda
\]
and this implies that the infinite dimensional vector space $\bigoplus_{i \in \mathbb{N}} \Lambda$ is reflexive, which is a contradiction.
\end{proof}

  \bibliographystyle{plain}
\bibliography{library}

   \end{document}